\documentclass[11pt]{article}
\usepackage{fullpage}
\usepackage{amssymb,latexsym,amsmath,amsthm}     
\usepackage{bbm}
\usepackage{eqnarray}
\usepackage[utf8]{inputenc}
\usepackage{enumerate}
\usepackage{enumitem}
\usepackage[english]{babel}
\usepackage{color}
\usepackage[dvipsnames]{xcolor}
\usepackage[toc,page]{appendix}
\usepackage{scalerel,stackengine}
\usepackage{titling}
\usepackage{mathrsfs}
\usepackage{float}
\usepackage{marginnote}
\usepackage[leftcaption]{sidecap}
\usepackage{hyperref}
\hypersetup{
    colorlinks=true,
    linkcolor=blue,
    filecolor=blue,      
    urlcolor=blue,
    citecolor=blue,
    linktocpage=true
}
\sidecaptionvpos{figure}{t}
\usepackage{titlesec}
\titleformat{\chapter}{\normalfont\bfseries}{\Large\thechapter}{1em}{\Large}
\newtheorem{theorem}{Theorem}[section] 
\newtheorem{proposition}[theorem]{Proposition} 
\newtheorem*{proposition*}{Proposition}
\newtheorem{definition}[theorem]{Definition}
\newtheorem*{definition*}{Definition}
\newtheorem{lemma}[theorem]{Lemma}
\newtheorem*{lemma*}{Lemma}
\newtheorem{corollary}[theorem]{Corollary}
\newtheorem{remark}[theorem]{Remark}
\newtheorem{assumption}[theorem]{Assumption}
\newcommand\dist{\operatorname{dist}}
\newcommand{\SO}[1]{\operatorname{SO}(#1)}

\newcommand{\erwartung}[1]{\mathbb{E} \left[ {#1} \right]}
\newcommand{\erwartunglol}[1]{\mathbb{E}_{L} \left[ {#1} \right]}
\newcommand{\e}{_{\varepsilon}}
\newcommand{\h}{_{\mathrm{hom}}}

\newcommand{\brac}[1]{\left({#1}\right) }
\newcommand{\cb}[1]{\left\lbrace {#1} \right\rbrace}
\newcommand{\norm}[1]{\left\lVert#1\right\rVert}
\newcommand{\R}{\mathbb{R}}
\newcommand{\Z}{\mathbb{Z}}
\newcommand{\N}{\mathbb{N}}

\newcommand{\dive}{\mathrm{div}}

\newcommand{\F}{\mathcal{F}}
\newcommand{\mbbE}{\mathbb{E}}
\newcommand\Id{\operatorname{Id}}

\newcommand{\nbs}{\nobreakspace}

\newcommand{\per}{\mathrm{per}}

\newcommand{\lol}{_{L}}

\newcommand{\varft}{\varphi_{F}^{T}}
\newcommand{\varftt}{\varphi_{F}^{2T}}
\newcommand{\varfgt}{\varphi_{F,G}^{T}}
\newcommand{\varfgtt}{\varphi_{F,G}^{2T}}
\newcommand{\varfhtt}{\varphi_{F,H}^{2T}}
\newcommand{\varfht}{\varphi_{F,H}^{T}}
\newcommand{\varfght}{\varphi_{F,G,H}^T}
\newcommand{\varfghtt}{\varphi_{F,G,H}^{2T}}
\newcommand{\hatvarft}{\widehat{\varphi}_{F}^{T}}
\newcommand{\hatvarfgt}{\widehat{\varphi}_{F,G}^{T}}
\newcommand{\hatvarfht}{\widehat{\varphi}_{F,H}^{T}}
\newcommand{\hatvarfght}{\widehat{\varphi}_{F,G,H}^T}
\newcommand{\hatomega}{\widehat{\omega}}
\newcommand{\loc}{_{\mathrm{loc}}}
\newcommand{\step}[1]{\medskip\noindent\textbf{Step #1. }}
\newcommand{\substep}[1]{\medskip\noindent\textit{Substep #1. }}

\makeatletter
\newcommand{\setword}[2]{%
  \phantomsection
  #1\def\@currentlabel{\unexpanded{#1}}\label{#2}%
}
\makeatother

\newcommand{\twopartdef}[4]
{
	\left\{
		\begin{array}{ll}
			#1 &  #2 \\
			#3 &  #4
		\end{array}
	\right.
}

\def\Xint#1{\mathchoice
{\XXint\displaystyle\textstyle{#1}}%
{\XXint\textstyle\scriptstyle{#1}}%
{\XXint\scriptstyle\scriptscriptstyle{#1}}%
{\XXint\scriptscriptstyle\scriptscriptstyle{#1}}%
\!\int}
\def\XXint#1#2#3{{\setbox0=\hbox{$#1{#2#3}{\int}$ }
\vcenter{\hbox{$#2#3$ }}\kern-.6\wd0}}

\def\dashint{\Xint-}

\usepackage{authblk}
\title{Quantitative stochastic homogenization of nonlinearly elastic, random laminates}

\author[1]{Stefan Neukamm\thanks{stefan.neukamm@tu-dresden.de}} \author[2]{Mathias Sch{\"a}ffner\thanks{mathias.schaeffner@tu-dortmund.de}}
\author[1]{Mario Varga\thanks{mario.varga@tu-dresden.de}} 
\affil[1]{Fakult\"at Mathematik, Technische Universit\"at Dresden}
\affil[2]{Fakult\"at f\"ur Mathematik, Technische Universit\"at Dortmund}
\begin{document}
\maketitle

\begin{abstract}
In this paper we study quantitative stochastic homogenization of a nonlinearly elastic composite material with a laminate microstructure. We prove that  for deformations close to the set of rotations the homogenized stored energy function $W_{\hom}$ is $C^3$ and that $W_{\hom}$, the stress-tensor $DW_{\hom}$, and the tangent-moduli $D^2W_{\hom}$ can be represented with help of stochastic correctors. Furthermore, we study the error of an approximation of these quantities via representative volume elements. More precisely, we consider periodic RVEs obtained by periodizing the distribution of the random material. For materials with a fast decay of correlations on scales larger than a unit scale, we establish error estimates on the random and systematic error of the RVE with optimal scaling in the size of the RVE and with a multiplicative random constant that has exponential moments.
\end{abstract}
\setlength{\parindent}{0pt}
\tableofcontents

\section{Introduction}

A standard model for a nonlinearly elastic, composite material is given by the non-convex energy functional $\mathcal E_\varepsilon(u):=\int_O W(\tfrac{x}{\e},\nabla u)\,dx$, where $O\subseteq\R^d$ denotes the reference domain occupied by the elastic body, $u:O\to\R^d$ its deformation, and $W: \R^d\times\R^{d\times d}\to\R\cup\{+\infty\}$ the stored energy function, which encodes the mechanical properties of the material. We are interested in the macroscopic behavior of composites that oscillate on a microscopic scale $0<\varepsilon\ll 1$. Therefore, we consider the homogenization limit $\varepsilon\to 0$ of $\mathcal E_\varepsilon$ in the framework of $\Gamma$-convergence. The first results in this direction are due to Braides \cite{Braides85} and M\"uller \cite{Mueller87}: For periodic composites (that is, $x\mapsto W(x,F)$ is periodic for all $F$) and assuming additional growth conditions for $W$ (in particular, standard $p$-growth with $1<p<\infty$), they show that $\mathcal E_\varepsilon$ $\Gamma$-converges to an energy functional of the form $u\mapsto \int_O W_{\hom}(\nabla u)\,dx$ with a homogenized stored energy function given by the \textit{multi-cell formula},
\begin{equation}\label{neukamm:2}
  W_{\hom}(F):=\lim\limits_{L\to\infty}W_{\hom,L}(F),\quad W_{\hom,L}(F) := \inf_{\varphi \in W^{1,p}_{\mathrm{per}}(\Box_L; \R^d)}\dashint_{\Box_L} W(x, F+ \nabla \varphi)\, dx,
\end{equation}
where $\Box_L:=[-\frac L2,\frac L2)^d$ and $W^{1,p}_{\mathrm{per}}(\Box_L; \R^d)$ denotes the space of $L$-periodic Sobolev functions. A similar homogenization result and, in particular, formula \eqref{neukamm:2} are also valid in the random case, that is, when $x\mapsto W(x,F)$ is a stationary and ergodic random field, see \cite{DalMaso1985, DalMaso198621, messaoudi1994stochastic}.
\smallskip

The multi-cell formula for the homogenized stored energy function $W_{\hom}$ is the main subject of this paper. Its definition invokes a nonconvex minimization problem and an asymptotic limit $L\to\infty$, and, therefore, a priori not much is known about its analytic properties. $W_{\hom}$ describes the effective mechanical behavior of the composite. It is especially important to understand the \textit{first Piola-Kirchhoff stress tensor} and the \textit{tangent modulus}, i.e., the Jacobian $DW_{\hom}(F)$ and the Hessian $D^2W_{\hom}(F)$.
Therefore, it is natural and relevant to investigate the following questions:
\begin{itemize}
\item[(Q1)] Is $W_{\hom}$ twice continuously differentiable?
\item[(Q2)] How can $W_{\hom}(F)$, $DW_{\hom}(F)$ and $D^2W_{\hom}(F)$ be evaluated or approximated?
\end{itemize}
In this paper we study these questions in the regime of small, but finite strains for random laminates composed of frame indifferent materials with a stress-free, non-degenerate reference state (see Definition~\ref{def:walphap} below). In a nutshell, the main results of this paper can be summarized as follows:
\begin{itemize}
\item (Qualitative results and answer to (Q1)). We show $W_{\hom}\in C^2(U)$ for an open neighborhood $U$ of $\SO d$, see Theorem~\ref{T:1:0}~(iii). This result is based on a representation of $W_{\hom}(F)$ for $F\in U$ with help of a sublinear corrector $\varphi_F$ that solves the nonlinear corrector equation $-{\rm div}DW(x,F+\nabla\varphi_F)=0$ in $\R^d$, see  Theorem~\ref{T:1:0}~(ii). Since $W$ is non-convex, this corrector representation is highly non-trivial and extends a recent results for periodic composites by the first two authors  \cite{neukamm2018quantitative,neukamm2019lipschitz}.
\item (Quantitative results and answer to (Q2)). We prove optimal error estimates for a periodic representative volume element approximation of $W_{\hom}(F)$, $DW_{\hom}(F)$ and $D^2W_{\hom}(F)$ for $F$ close to $\SO d$, see Theorem~\ref{Theorem:2} and the discussion at the end of this introduction.
\end{itemize}

Before we present our results in detail, we discuss previous, related works.\smallskip

{\it Convex case.} We first note that a complete understanding of (Q1) is currently only available for convex integrands with quadratic growth. Indeed, as shown by M\"uller \cite{Mueller87}, if $W(x,F)$ is periodic in $x$ and convex in $F$, the multi-cell formula reduces to a \textit{one-cell formula} that can be represented with help of a \textit{corrector}, i.e., $W_{\hom}(F)=W_{\hom,1}(F)=\int_{\Box}W(y,F+\nabla\varphi_F)\,dy$, where the corrector $\varphi_F$ is a minimizer of the minimization problem in the definition of $W_{\hom,1}(F)$. Furthermore, in \cite[Theorem~5.4]{Mueller93} it is shown that, if $W(x,F)$ is additionally $C^2$ in $F$ and satisfies a quadratic growth condition in $F$, then $W_{\hom}$ is $C^2$. The argument exploits the \textit{corrector representation} of $W_{\hom}$ and only uses basic energy estimates for the corrector. It extends verbatim to the random, convex case.

To reach regularity beyond $C^2$ is a nontrivial task and requires improved regularity of the corrector. For convex potentials with quadratic growth, the corrector $\varphi_F$ can be characterized by a monotone, uniformly elliptic system with quadratic growth and thus, a stronger regularity theory is available; we refer to \cite{neukamm2018quantitative} for $C^3$-regularity of the homogenized potential in the periodic, vector-valued case, and \cite{fischer2019optimal, armstrong2020higher, clozeau2021quantitative} for related result for random, monotone equations.

{\it Non-convex case.}
The  results mentioned so far do not extend to the case of nonlinear elasticity, since there the stored energy function is necessarily non-convex.
In fact, M\"uller \cite[Theorem 4.3]{Mueller87} provides a counterexample in form of a laminate material that features a buckling instability under compressive loading; in particular, one has $W_{\hom}(F)<W_{\hom,L}(F)$ for some $F\not\in\SO d$ and all $L\in\mathbb N$. This shows that we cannot even expect a one-cell formula nor a corrector representation to hold for general deformations $F$. However, a better behaviour can be expected for deformations close to $\SO d$ for materials with a non-degenerate, stress-free reference state at the identity. Indeed, in \cite{MN11} the first author and M\"uller show for periodic composites that $W_{\hom}$ admits a quadratic expansion at identity:
\begin{equation}\label{eq:intro1}
  W_{\hom}(Id+G)=Q_{\hom}(G)+o(|G|^2)
\end{equation}
where $Q_{\hom}(G):=\inf_{\varphi \in H^{1}_{\mathrm{per}}(\Box; \R^d)}\int_{\Box} Q(x, G+ \nabla \varphi)\, dx$, and $Q(x,G):=\frac12D^2W(x,Id)[G,G]$ denotes the quadratic term in the expansion of $W$ at identity. The argument uses soft properties of the corrector $\varphi_F$ and appeals to the geometric rigidity estimate of \cite{friesecke2002theorem}. An analogous statement holds in the random case, see \cite{GN11}. Identity \eqref{eq:intro1} says that the tangent modulus of $W_{\hom}$ at  identity is given by the homogenization of the tangent modulus of $W$.
\smallskip

{\it Differentiability for small, but finite strains.} To establish an expansion similar to \eqref{eq:intro1} for $F\not\in SO(d)$ is nontrivial --- and, in fact, not always possible, cf.~M\"uller's counterexample mentioned above. Therefore, it is natural to focus on the regime of small strains, i.e., when $F$ is close to $\SO d$. The first differentiability result for $W_{\hom}$ away from $\SO d$ has been recently obtained by the first and second author in \cite{neukamm2018quantitative,  neukamm2019lipschitz} for a periodic composite with a ``regular'' microstructure, e.g., a matrix material with smooth, possibly touching inclusions. It is shown that if $W(x,\cdot)$ is $C^2$ in a neighbourhood of $\SO d$, then the homogenized stored energy function $W_{\hom}$ is $C^2$ in a neighbourhood of $\SO d$, say in $U_{\rho}:=\big\{F\in\R^{d\times d}\,:\,{\rm dist}(F,SO(d))<\rho\,\big\}$ for some $\rho>0$. The argument critically relies on the observation that for $F\in U_{\rho}$ the multi-cell formula simplifies to a single-cell formula that features a representation with help of a periodic corrector. More precisely, in \cite{neukamm2018quantitative,  neukamm2019lipschitz} we prove that for any $F\in U_{\rho}$ there exists a periodic corrector  $\varphi_F\in H^1_{\mathrm{per}}(\Box; \R^d)$ such that
\begin{equation}\label{eq:intro2}
  W_{\hom}(F)=W_{\hom,1}(F)=\dashint_{\Box}W(x,F+\nabla\varphi_F(x))\,dx.
\end{equation}
Moreover, we obtain a similar representation for the  derivatives of $W_{\hom}$, in particular, we show that the tangent modulus admits the respresentation
\begin{equation*}
  D^2W_{\hom}(F)H\cdot G=\,\dashint_{\Box}D^2W(x,F+\nabla \varphi_F)(H+\nabla \varphi_{F,H}) \cdot G\,dx
\end{equation*}
for the periodic, linearized corrector $\varphi_{F,H}\in H^{1}_{\mathrm{per}}(\Box; \R^d)$. These corrector representations are the starting point for a quantitative analysis of $W_{\hom}$ and minimizers of $\mathcal E_\varepsilon$; in particular, in \cite{neukamm2018quantitative, neukamm2019lipschitz} we prove error estimates for the nonlinear two-scale expansion and we establish Lipschitz estimates for minimizers of $\mathcal E_\varepsilon$ that are uniform in $\varepsilon>0$. The general strategy of proof for \eqref{eq:intro2} relies on a reduction to the convex case based on two major ingredients:
\begin{itemize}
\item A \textit{matching convex lower bound}: Inspired by \cite{CDKM06} we construct in \cite[Corollary 2.3]{neukamm2018quantitative} a strongly convex integrand $V$ with quadratic growth such that
  \begin{align}
    &W(x,F)+\mu \det F\ \geq V (x,F)\qquad\mbox{for all $F\in \R^{d\times d}$,} \nonumber \\
    \label{eq:intro3b}
    &W(x,F)+\mu \det F\ =V (x,F)\qquad \mbox{for all }F\in U_{\rho},
  \end{align}
  where $\mu>0$ and $\rho>0$ only depend on $W$. 
With help of $V$, convex homogenization theory and the fact that $F\mapsto\det F$ is a Null-Lagrangian, we obtain the lower bound $W_{\hom}(F)+\mu\det F\geq V_{\hom}(F)=\int_\Box V(x,F+\nabla\varphi_F)$  with a convex corrector $\varphi_F$ given as the unique (up to a constant) sublinear solution to the monotone corrector problem, 
\begin{equation}\label{intro:convexcorrectoreq}
-\nabla\cdot DV(x,F+\nabla\varphi_F)=0\qquad\mbox{in $\R^d$.}
\end {equation}  

\item \textit{Global Lipschitz regularity.} By the regularity theory for uniformly elliptic,  monotone systems with piecewise constant, periodic coefficients, we establish the following global  Lipschitz estimate for the convex corrector
  \begin{equation*}
    \|{\rm dist}(F+\nabla\varphi_F,SO(d))\|_{L^\infty(\Box)}\leq C{\rm dist}(F,SO(d)),
  \end{equation*}
  where for $d>2$ we require the right-hand side to satisfy an additional smallness condition, see \cite[Corollary 1]{neukamm2019lipschitz}. In view of the matching property of $V$ (cf.~\eqref{eq:intro3b}), this allows us to deduce the sought for corrector representation \eqref{eq:intro2}.
\end{itemize}
\smallskip

{\it Random composites and laminates.} The construction of the matching convex lower bound $V$ verbatim extends to the random case, see Lemma~\ref{C:wv} below. Nevertheless the analysis in \cite{neukamm2018quantitative, neukamm2019lipschitz} is restricted to periodic composites, since periodicity is critically used to obtain the required global Lipschitz estimate for the corrector: Indeed, in \cite{neukamm2019lipschitz} we obtain the Lipschitz estimate by combining a small-scale Lipschitz estimate (cf.~\cite[Theorem 4]{neukamm2019lipschitz}) of the form $\|\nabla\varphi_F\|_{L^\infty(\Box_{1/2})}\leq C\|\nabla\varphi_F\|_{L^2(\Box)}$ with the energy estimate $\|\nabla\varphi_F\|_{L^2(\Box)}\leq C{\rm dist}(F,SO(d))$ for sublinear solutions to \eqref{intro:convexcorrectoreq}. While the latter is standard in the periodic case, in the random setting such an estimate generally only holds in the weaker form of a \textit{large scale} $L^2$-estimate: For all $L$ larger than a random minimal radius we have $L^{-\frac d2}\|\nabla\varphi_F\|_{L^2(\Box_L)}\leq C{\rm dist}(F,SO(d))$; we refer to \cite{fischer2019optimal} where such an estimate has been established for monotone systems.  Thus, in the random setting we may only expect the bound
\begin{equation*}
  \|{\rm dist}(F+\nabla\varphi_F,SO(d))\|_{L^\infty(B_1(x))}\leq \mathcal C(x){\rm dist}(F,SO(d)),
\end{equation*}
where $\mathcal C(x)$ is a stationary random field with stretched exponential moments. We note that this Lipschitz estimate is not global, which prevents us to follow the strategy of \cite{neukamm2019lipschitz}.

In contrast, for random laminates, as considered in this paper, the situation is better, since then the corrector problems for $\varphi_F$ and the linearized corrector $\varphi_{F,G}$, simplify to ordinary differential equations, see~Theorem~\ref{T:1:0}(ii)(b). This allows us to retrieve global, deterministic Lipschitz estimates by appealing to ODE-arguments, see Lemma~\ref{lemma:388:new} below.
\medskip

{\it Quantitative periodic RVEs.} We finally comment on our quantitative analysis of the RVE approximation for $W_{\hom}$, $DW_{\hom}$, and $D^2W_{\hom}$. As we explain in detail in the next section, in our paper we consider parametrized models to describe random laminates. More precisely, we denote by $(\Omega,\mathbb P)$ a probability space of parameter fields $\omega:\R\to\R$ where $\mathbb P$ is stationary and ergodic w.r.t.~the shifts $\omega\mapsto\omega(\cdot+z)$, $z\in\R$. We then consider stored energy functions of the form $(x,F)\mapsto W(\omega(x_d),F)$, where $x_d$ denotes the $d$th coordinate of $x\in\R^d$, see Assumption~\ref{ass:W:1}. Our qualitative Theorem~\ref{T:1:0} then yields the corrector representation formula
\begin{equation}\label{neukamm:21}
  W_{\hom}(F)=\mathbb E\Big[\dashint_{\Box}W(\omega(x_d),F+\nabla\varphi_F)\,dx\Big],
\end{equation}
for $F$ sufficiently close to $SO(d)$. A direct evaluation of \eqref{neukamm:21} is not possible in practice, since the expression invokes two ``infinities'': The domain of the corrector equation \eqref{intro:convexcorrectoreq} is unbounded and the probability space has infinite degrees of freedom. Hence, a suitable approximation is required.  The \textit{representative volume element} (RVE) method is a well-established procedure for this purpose. In this method, the corrector problem is considered (for a finite number of samples) on a domain of a \textit{finite} size $L$ together with suitable boundary conditions. It is an ongoing discussion in the computational mechanics community how to choose the size of the RVE and the appropriate boundary conditions, e.g.~see \cite{kanit,schneider2021} and the references therein. In particular, we note that \cite{kanit} also provides a numerical study of the convergence rate for the RVE-approximation  of the tangent modulus of a nonlinear material.

The first analytic convergence results with optimal scaling in $L$ for the RVE-approximation have been obtained by Gloria and Otto in \cite{GloriaOtto, GloriaOtto2} for discrete, linear elliptic equations. Periodic-RVEs have been first investigated in \cite{gloria2015quantification} in the discrete setting and recently in \cite{fischer2019optimal} in the case of monotone systems. In periodic-RVEs (in the form of \cite{gloria2015quantification,fischer2019optimal}) the original probability measure $\mathbb P$ is approximated by a stationary, probability measure $\mathbb P_L$ that is supported on $L$-periodic parametrizations. The effect of this approximation is that for $\mathbb P_L$-almost every realization, the corrector equation simplifies to an equation on a \textit{finite} domain (of size $L$) with \textit{periodic boundary conditions}. These equations are well-posed pathwise and thus for $\mathbb P_L$-a.e.~parameter field $\omega_L\in\Omega$ the proxy
\begin{equation*}
  W_{\hom, L}(\omega_L,F):=\dashint_{\Box_{L}}W(\omega_L(x_d),F+\nabla\varphi_F(\omega_L,\cdot))\,dx
\end{equation*}
is well-defined (provided $F\in U_{\rho}$). As already observed in \cite{GloriaOtto} the RVE-error naturally decomposes as
\begin{equation*}
  W_{\hom,L}(F)-W_{\hom}(F)=\underbrace{W_{\hom,L}(F)-\mathbb E_L\Big[W_{\hom,L}(F)\Big]}_{=:{\rm err}_{\rm rand}(L)} + \underbrace{\mathbb E_L\Big[W_{\hom,L}(F)-W_{\hom}(F)\Big]}_{=:{\rm err}_{\rm syst}(L)},
\end{equation*}
into a random error (due to random fluctuations w.r.t. $\mathbb P_L$) and a systematic error due to differences of $\mathbb P_L$ and $\mathbb P$. As already observed in \cite{gloria2015quantification}, in order to quantify the error, the assumption of ergodicity needs to be quantified. Moreover, both errors are sensitive to the decay of correlations of the random material, see also \cite{gloria2019quantitative} for a detailed analysis on the impact of the decay rate of correlations and the scaling of the sublinear corrector.
In the present paper we study the periodic-RVE approximation under the strongest possible assumption on the decay of correlations, namely,  materials that feature a rapid decay of correlations on scales larger than $1$. In the existing literature two different approaches exist to quantify ergodicity: Either by means of mixing conditions as in \cite{armstrong2016quantitative,ArmstrongKuusiMourrat16} or in form of functional inequalities, see \cite{naddaf1998estimates, GloriaOtto, GloriaOtto2, gloria2019quantitative, DuerinckxGloria2020b, duerinckx2017multiscale}. In the present paper we follow the second approach and work with a spectral gap estimate. In essence, we make the following assumptions:
\begin{itemize}
\item \textit{Existence of  an $L$-periodic approximation}:  There exists a shift-invariant probability measure $\mathbb P_L$ that is supported on $L$-periodic fields in $\Omega$ and that recovers the distribution of $\mathbb P$ in $\Box_{L/2}$, see Definition~\ref{def:pl} for details.
\item \textit{Fast decorrelations}: $(\Omega,\mathbb P)$ and $(\Omega,\mathbb P_L)$ feature \textit{spectral gap estimates} that encode fast decorrelations on scales larger than $1$, see Definition~\ref{sgs}.
\end{itemize}
If $W(\omega,F)$ is sufficiently regular in $\omega$ and $F$, and for deformations $F$ that are sufficiently close to $SO(d)$, our main result then yields estimates on the error of the RVE-approximation of $W_{\hom}(F)$, the stress tensor $DW_{\hom}$ and the tangent modulus $D^2W_{\hom}$. In particular, for $W_{\hom,L}$ we get,
\begin{align*}
|{\rm err}_{\rm rand}|\,\leq \mathcal C\,{\rm dist}^2(F,SO(d))L^{-\frac12},\qquad |{\rm err}_{\rm syst}(L)|\,\leq C\, {\rm dist}^2(F,SO(d)) \frac{\ln L}{L},
\end{align*}
where $\mathcal C$ denotes a random constant with $\mathbb E_L\big[\exp(\tfrac1 C\mathcal C)\big]\leq 2$, and $C$ denotes a deterministic constant, see Theorem~\ref{Theorem:2}. These estimates are optimal w.r.t.~scaling in $L$ and w.r.t.~the integrability of $\mathcal C$. Moreover, they are the first analytical estimates for RVEs for geometrically nonlinear composites. Due to the different scaling of the random and systematic error, our results allows for optimizing the ratio between the number of Monte-Carlo iterations and the size of the RVE, see Corollary~\ref{cor:T:2}.
\bigskip

{\bf Structure of the paper:} We present the assumptions and main results of this paper in detail in Section \ref{sec:setting}. In Section \ref{sec:determ} we collect and prove the deterministic auxiliary results needed for the proofs of the main results. In Section \ref{sec:pr:qual} the proof of the main qualitative result is presented. In Section \ref{sec:stochastic} we collect and prove the stochastic auxiliary estimates that are needed for the proof of the quantitative error estimates. Section \ref{sec:rve} is dedicated to the proof of the main quantitative results. 

\smallskip
{\bf Notation.} 
\begin{itemize}
\item We write $\dashint_{A}\cdot$ for the averaged integral $\frac{1}{|A|}\int_{A}\cdot$. We frequently write $(f)_{A}$ to denote $\dashint_{A}f$.
\item For $L>0$ and $d\in \N$, we use the notation $\Box_{L}:=(0,L)^d$.
\item We identify one-dimensional functions $\R \ni x \mapsto u(x)\in \R$ with their extension to $\R^d$ that depends only on the $d$-th coordinate direction, i.e., $\R^{d} \ni x \mapsto u(x_d)\in \R$. We usually write $x_d$ for a one-dimensional variable in $\R$. Derivatives with respect to $x_d$ are denoted by $\partial_{x_d}$.  In particular, for $\varphi\in H^1\loc(\R;\R^m)$ we write $\nabla \varphi$ to denote $\partial_{x_d}\varphi \otimes e_d$ and not $\partial_{x_d}\varphi$. The same applies to divergence, in particular, $\dive$ always denotes the $d$-dimensional divergence operator.  
\item For a matrix valued mapping $v: A\to \R^{m\times d}$, we denote by $v_d: A \to \R^{m}$ its $d$-th column, i.e., $v_d(x)=[v_{1d}(x),...,v_{md}(x)]$. 
\item We write $W^{1,p}_{\mathrm{per},0}(\Box_{L}):= \cb{\varphi \in W^{1,p}_{\mathrm{per}}(\Box_{L}): \; \dashint_{\Box_{L}}\varphi \; dx = 0}$, $H^1_{\mathrm{per},0}(\Box_{L}):=W^{1,2}_{\mathrm{per},0}(\Box_{L})$.
\item We set $W_{\rm uloc}^{1,p}(\R^d):=\{u\in W_{\rm loc}^{1,p}(\R^d)\,:\,\sup_{x\in\R^d}\|u\|_{W^{1,p}(B_1(x))}<\infty\}$, $H_{\rm uloc}^1(\R^d)=W_{\rm uloc}^{1,2}(\R^d)$
\item For $\delta > 0$, we use the notation $U_{\delta}:= \cb{F \in \R^{d \times d} : \: \dist(F,\SO d)< \delta}$.
\end{itemize}

\section{Setting and main results}\label{sec:setting}

In this section we first introduce the standard setting for the description of stationary random media. Consequently we present the main results about nonconvex integral functionals in the vicinity of rotations: In Theorem \ref{T:1:0} we argue that the RVE approximations are well-defined and the so-called stochastic one-cell formula holds. Theorem \ref{Theorem:2} describes the precision of the RVE method in terms of the size of the representative elements.

\bigskip
We consider the parameter space $\Omega=\cb{\omega: \R\to \R \quad \text{measurable}}$ equipped with a $\sigma$-algebra $\mathcal{S}$ and a probability measure $\mathbb{P}$. The mathematical expectation is denoted by $\erwartung{\cdot}:= \int_{\Omega} \cdot d\mathbb{P}(\omega)$. A sample $\omega \in \Omega$ describes the spatially-varying (layered) material properties and for this reason we frequently refer to it as a \textit{configuration}. We assume the following standard assumptions for the random configuration space $(\Omega,\mathcal{S},\mathbb{P})$:
\begin{enumerate}[label = (P\arabic*)]
\item \label{stationarity} (\textit{Stationarity}.) For any $z \in \R$, the vectors $\brac{\omega(x_1),...,\omega(x_n)}$ and $\brac{\omega(x_1+z),..., \omega(x_n+z)}$  have the same joint distribution for arbitrary $x_1,..., x_n \in \R$ and $n\in \N$.
\item \label{ass:ergodicity} (\textit{Ergodicity}.) For any  random variable $\mathcal{F}\in L^1(\Omega)$ we have
\begin{equation*}
\lim_{R\to \infty} \dashint_{0}^R \mathcal{F}(\omega(\cdot+ x)) dx = \erwartung{\mathcal{F}}.
\end{equation*} 
\end{enumerate}

We state the assumption on the considered energy density $W$ as follows. We first introduce a class of frame-indifferent stored energy functions that are minimized, non-degenerate and smooth at identity, and satisfy the growth condition stated below.
\begin{definition}\label{def:walphap}
  For $\alpha>0$ and $p>1$, we denote by $\mathcal W_{\alpha}^p$ the class of Borel functions $W:\R^{d\times d}\to[0,+\infty]$ which satisfies the following  properties:
  \begin{itemize}
  \item[(W1)] $W$ satisfies $p$-growth from below, i.e.
    \begin{equation*}
      \alpha|F|^p-\frac{1}{\alpha}\leq W(F)\quad\mbox{for all $F\in\R^{d\times d}$};
    \end{equation*}
  \item[(W2)] $W$ is frame indifferent, i.e.
    \begin{equation*}
      W(RF)=W(F)\quad\mbox{for all $R\in\SO d$, $F\in\R^{d\times d}$};
    \end{equation*}
  \item[(W3)] $F=\Id$ is a natural state and $W$ is non-degenerate, i.e.\ $W(\Id)=0$ and 
    \begin{align}
      W(F)&\geq \alpha\dist^2(F,\SO d)\quad\mbox{for all $F\in\R^{d\times d}$;}\label{ass:onewell}
    \end{align}
  \item[(W4)] $W$ is $C^3$ in a neighborhood of $\SO d$ and
    \begin{equation}\label{ass:wd2lip}
      \|W\|_{C^3(\overline {U_{\alpha}})}<\frac{1}{\alpha}.
    \end{equation}
    \end{itemize}
 \end{definition}
\begin{assumption}\label{ass:W:1}
  Fix $\alpha>0$ and $p\geq d$. We suppose that $W:\R\times \R^{d\times d}\to [0,+\infty]$ is a Borel function and $W(\omega,\cdot)\in\mathcal W_{\alpha}^p$ for almost every $\omega\in \R$.
\end{assumption}
 \begin{remark}
   Note that we use the same constant $\alpha$ in the growth condition
   (W1), the non-degeneracy condition (W3) and the regularity
   assumption (W4). The only reason for this is that we want to reduce the number
   of parameters invoked in the assumption for $W$. Let us anticipate
   that the region, in which the multi-cell formula reduces to a
   single-cell expression, will depend (in a quite implicit way) on
   the constants in (W1), (W3) and (W4). Hence, working with the
   single parameter $\alpha$ simplifies the presentation. Note that we
   have $W^p_{\alpha}\subset W^{p}_{\alpha'}$ for $0<\alpha'<\alpha$.
 \end{remark}
We associate with $W$ the following approximation for the homogenized stored energy function 
\begin{equation}\label{eq:364:2}
W_{\mathrm{hom},L}: \Omega \times \R^{d\times d} \to \R, \quad W_{\mathrm{hom}, L}(\omega,F) = \inf_{\varphi \in W^{1,p}_{\mathrm{per}, 0}(\Box_L; \R^d)}\dashint_{\Box_L} W(\omega(x_d), F+ \nabla \varphi) dx.
\end{equation}

We state the first main result of this paper as follows. Based on a merely qualitative ergodicity assumption combined the laminated structure of the material, we show that in the vicinity of rotations correctors exist and the homogenized stored energy function is characterized by a stochastic one-cell formula. We also discuss regularity properties of the homogenized stored energy function. 

\begin{theorem}[Corrector representation and regularity of {$W_{\mathrm{hom}}$}]\label{T:1:0} Let $\brac{\Omega, \mathcal{S}, \mathbb{P}}$ satisfy \ref{stationarity}-\ref{ass:ergodicity} and let Assumption~\ref{ass:W:1} hold. Then there exists $\overline{\delta}=\overline{\delta}(\alpha,d,p)>0$ and $c=c(\alpha,d,p)\in[1,\infty)$ such that for all $F \in U_{\overline{\delta}}$ the following statements hold true:
\begin{enumerate}[label = (\roman*)]
\item (Homogenized energy density). The following limit exists and it is deterministic 
\begin{align*}
 W_{\hom}(F)
 :=&\lim_{L\to\infty} W_{\hom,L}(\omega,F)\qquad\mbox{for $\mathbb P$-a.a. $\omega\in \Omega$.}
\end{align*}
\item (Corrector). There exists a unique random field $\varphi_F:\Omega\times\R^d\to\R^d$ with the following properties:
\begin{enumerate}
 \item For $\mathbb P$-a.a. $\omega\in \Omega$ it holds $\varphi_F(\omega,\cdot)\in W^{1,\infty}_{\rm loc}(\R^d;\R^d)$, $\dashint_{\Box}\varphi_F(\omega,\cdot)=0$ and 
   \begin{equation}\label{eq:3912}
  -\dive DW(\omega(x_d),F+\nabla \varphi_F)=0 \quad \text{in } \R^d,
 \end{equation} 
and the corrector $\varphi_F$ is sublinear in the sense
 \begin{equation}\label{eq:correctorsublinear}
  \lim_{R \to \infty} \frac{1}{R}\|\varphi_F(\omega,\cdot)\|_{L^\infty(\Box_{R})} = 0.
 \end{equation}
 \item The corrector $\varphi_F$ is one-dimensional in the sense of $\varphi_F(\omega,x)=\varphi_F(\omega,x_d)$. 
 \item The gradient of the corrector $\varphi_F$ is stationary and satisfies $\erwartung{\nabla \varphi}=0$ and
 \begin{equation}\label{est:stationarygradient}
 \|\nabla \varphi_F(\omega,\cdot)\|_{L^\infty(\R^d)}\leq c\dist(F,\SO d)\qquad\mbox{for $\mathbb P$-a.a. $\omega\in \Omega$}.
 \end{equation}
\end{enumerate} 
The homogenized energy density given in (i) satisfies the stochastic one-cell formula
\begin{equation}\label{eq:whomviastationarycorrector}
 W_{\hom}(F)=\erwartung{\dashint_{\Box}W(\omega(x_d),F+\nabla\varphi_F(\omega))dx}.
\end{equation}
\item (Regularity). The energy densities $W_{\hom,L}(\omega,\cdot)$ and $W_{\hom}$ defined in \eqref{eq:364:2} and in (i) are of class $C^{3}(U_{\overline{\delta}})$ and it holds, for all $G,H \in \R^{d\times d}$,
\begin{align}
 DW_{\hom}(F)\cdot G=& \erwartung{\dashint_{\Box}DW(\omega(x_d),F+\nabla \varphi_F)\cdot G \;dx},\label{eq:def:dwhom} \\
 D^2W_{\hom}(F)H\cdot G=&\erwartung{ \dashint_{\Box}D^2W(\omega(x_d),F+\nabla \varphi_F)(H+\nabla \varphi_{F,H}) \cdot G\,dx},\label{eq:def:ddwhom}
\end{align}
where the linearized corrector $\varphi_{F,H}:\Omega\times\R^d\to\R^d$ is uniquely characterized by
\begin{equation}\label{eq:linearizedcorrector:1}
-\dive \brac{D^2W(\omega(x_d), F+ \nabla \varphi_F)\brac{H+\nabla \varphi_{F,H}}}= 0 \quad \text{in }\R^d, \quad\mbox{for $\mathbb P$-a.a.\ $\omega \in \Omega$,}
\end{equation}
and
\begin{equation}\label{eq:linearizedcorrector:2}
\dashint_{\Box} \varphi_{F,H} dx = 0, \quad \nabla \varphi_{F,H} \text{ is stationary,} \quad \erwartung{\nabla \varphi_{F,H}} =0, \quad \erwartung{|\nabla \varphi_{F,H}|^2}< \infty.
\end{equation}
Moreover, $\varphi_{F,H}$ is one-dimensional in the sense $\varphi_{F,H}(\omega,x)=\varphi_{F,H}(\omega,x_d)$ and for all $H\in\R^{d\times d}$ it holds $\lim_{R\to\infty}\frac1R\|\varphi_{F,H}\|_{L^\infty(\Box_R)}=0$ and $\norm{\nabla \varphi_{F,H}(\omega,\cdot)}_{L^{\infty}(\R^d)}\leq c |H|$ for $\mathbb P$-a.a. $\omega\in \Omega$.

\item (Strong rank-one convexity). 
\begin{equation}\label{limit:nondegenerate}
  D^2W_{\hom}(F)[a\otimes b,a\otimes b]\geq  \frac1c|a\otimes b|^2\quad\mbox{for all $a,b\in\R^d$}.
\end{equation}
\end{enumerate} 
\end{theorem}
The proof is presented in Section\nobreakspace\ref{sec:pr:qual}.
\begin{remark}[$C^k$-regularity]
In Theorem\nbs\ref{T:1:0}, we show that $ \sup_{\omega\in \R}\norm{W(\omega,\cdot)}_{C^3(\overline{U_{\alpha}})}<\infty$  implies $W_{\mathrm{hom},L}(\omega,\cdot), W\h \in C^{3}(U_{\overline{\delta}})$. By an inductive argument that exploits the laminate structure, it is also possible to show that $\sup_{\omega\in \R}\norm{W(\omega,\cdot)}_{C^k(\overline{U_{\alpha}})}<\infty$ implies $W_{\mathrm{hom},L}(\omega,\cdot), W\h \in C^{k}(U_{\overline{\delta}})$ for an arbitrary $k \in \N$. Moreover,  explicit formulas for the derivatives in terms of linearized correctors of higher order can be obtained, for example, we have, for $F,G,H,I\in \R^{d \times d}$,
\begin{align}\label{eq:def:d3whom}
D^3W_{\mathrm{hom}}(F)IHG  = &\erwartung{  \dashint_{\Box} D^3W(\omega(x_d),F+ \nabla \varphi_{F})\brac{I+ \nabla \varphi_{F,I}}\brac{H+ \nabla \varphi_{F,H}} \cdot (G+\nabla \varphi_{F,G}) dx}.
\end{align}
\end{remark}

For deterministic, periodic composites, the conclusion of Theorem~\ref{T:1:0} was proven in \cite{neukamm2018quantitative,neukamm2019lipschitz} under various assumptions on the regularity of $x\mapsto W(x,F)$. However, Theorem~\ref{T:1:0} is the first corrector representation result for nonlinearly elastic materials with a random microstructure. Furthermore, the smallness assumption $F\in U_{\bar \delta}$ in Theorem~\ref{T:1:0} cannot be relaxed as can be seen by M\"uller's counterexample \cite[Theorem 4.7]{Mueller87} which features a periodic laminate material that undergoes buckling for sufficiently strong compressions.
\medskip

We remark that Assumption~\ref{ass:W:1} does not impose any condition on the growth of $W$ from above; $W(\omega,F)$ might even be equal to $+\infty$ for $F$ outside of an open neigbourhood of $\SO d$. Therefore, it is unclear if the energy functional $u\mapsto \int_{O}W(\omega(\frac{x_d}{\varepsilon}), \nabla u)dx$ homogenizes (in the sense of $\Gamma$-convergence). However, if $W$ additionally satisfies $p$-growth of the form
\begin{equation}\label{eq:standard:pgrowth}
W(\omega,F) \lesssim1 + |F|^p \quad \text{for all }F \in \R^{d\times d}\mbox{ and $\mathbb P$-a.a.\ $\omega\in\Omega$,}
\end{equation}
the corresponding $\Gamma$-limit takes the form $u\mapsto \int_{O}\overline{W}_{\mathrm{hom}}(\nabla u)dx$ where
\begin{equation*}
\overline{W}\h(F) = \lim_{L\to \infty} \inf_{\varphi \in W^{1,p}_0(\Box\lol; \R^d)}\dashint_{\Box\lol} W(\omega(x_d),F+ \nabla \varphi(x))dx.
\end{equation*}
The following corollary shows that $W\h$ and $\overline{W}\h$ match in the vicinity of rotations.
\begin{corollary}\label{cor:T:1}
Consider the setting of Theorem~\ref{T:1:0}. There exists ${\overline{\overline \delta}}={\overline{\overline \delta}}(\alpha,d,p)>0$ such that
\begin{equation*}
W\h(F) = \overline{W}\h(F)\qquad\mbox{ for all $F \in U_{\overline{\overline{\delta}}}$}.
\end{equation*}
\end{corollary}
(See Section\nobreakspace\ref{sec:pr:qual} for the proof.)

\begin{remark}
 Homogenization of $u\mapsto \int_{O}W(\omega(\frac{x_d}{\varepsilon}), \nabla u)dx$ in the sense of $\Gamma$-convergence has been established under weaker conditions than the upper-growth condition \eqref{eq:standard:pgrowth}, e.g., see the classic textbooks~\cite{BraidesBookHom,ZhikovBook}, and more recent works~\cite{DG16,HM11,HNS21,NSS17}. However, to our knowledge all existing results require some control from above or some convexity condition on the set $\{W<\infty\}$ that rule out the application to stored energy functions that satisfy physical growth conditions. In contrast,  with help of the corrector representation of $W_{\rm hom}$ it is possible to prove homogenization in a small-strain regime without imposing an additional growth condition for $W$. More precisely, by adapting the method of \cite[Theorem 3]{neukamm2018quantitative} we get the following: Let $O\subset\R^d$ be sufficiently smooth and let $(u\e)\e\subset g+W_0^{1,p}(O;\R^d)$ be such that
\begin{equation*}
 \mathcal I_{\varepsilon}(u\e)=\inf_{g+W_0^{1,p}(A)}\mathcal I\e+o(1)\quad\mbox{where}\quad \mathcal I_{\varepsilon}(u):=\int_OW(\omega(\tfrac{x_d}\e),\nabla u)-f\cdot u\,dx.
\end{equation*} 
Suppose that the assumptions of Theorem~\ref{T:1:0} are satisfied and assume the smallness condition $\|f\|_{L^q(A)}+\|g-{\rm id}\|_{W^{2,q}(A)}\ll1$ for some $q>d$. Then it holds
\begin{equation}\label{hom:qualitative}
\lim_{\varepsilon\to0}\|u\e-u_0\|_{L^p(O)}=0\qquad \mbox{$\mathbb P$-a.s.,}
\end{equation}
where $u_0$ denotes the unique minimizer of the homogenized energy functional
\begin{align*}
 \mathcal I_{\hom}(u):=\int_O W_{\hom}(\nabla u)-f\cdot u\,dx\qquad\mbox{subject to}\quad u-g\in W_0^{1,p}(O).
\end{align*}
Since in Theorem~\ref{T:1:0} we do not assume any growth conditions from above on $F\mapsto W(\omega,F)$, the convergence \eqref{hom:qualitative} does not follow from the known homogenization theory. 
\end{remark}

\medskip

In the following we discuss an approximation of $W_{\mathrm{hom}}$ and its derivatives by periodic RVEs. In contrast to the approximation \eqref{eq:364:2}, which relies on the introduction of periodic boundary conditions, the periodic RVE that we discuss below is analogous to the one considered in \cite{fischer2019optimal} and is based on a periodic approximation of the probability space $(\Omega, \mathcal{S}, \mathbb{P})$:

\begin{definition}[$L$-periodic approximation of $\mathbb P$]\label{def:pl}
Let $L\geq 1$. A probability space $(\Omega_L = \Omega, \mathcal{S}\lol = \mathcal{S}, \mathbb{P}\lol)$ is called an $L$-periodic approximation of $(\Omega, \mathcal{S}, \mathbb{P})$ if the following conditions hold:
\begin{enumerate}[label = (\roman*)]
\item \label{periodic_stationary}  $\mathbb{P}\lol$ concentrates on $L$-periodic functions and it is stationary in the sense of \ref{stationarity}.
\item \label{periodization}  If the random field $\Omega \times \R \ni (\omega,x) \mapsto \omega(x)$ is distributed according to $\mathbb{P}$ and $\Omega\lol \times \R \ni (\omega\lol,x) \mapsto \omega\lol(x)$ according to $\mathbb{P}\lol$, then the random fields
\begin{equation*}
\Omega \times B_{\frac{L}{4}}(0) \ni (\omega,x) \mapsto \omega(x) \quad \text{and} \quad \Omega\lol \times B_{\frac{L}{4}}(0) \ni (\omega\lol,x) \mapsto \omega\lol(x)
\end{equation*}
have the same distribution. 
\end{enumerate}
\end{definition}

In order to quantify the statement in Theorem \ref{T:1:0} (i) we will make an additional quantitative ergodicity assumption on the probability space and its periodic approximation. This relies on a Malliavin type functional calculus, see \cite{duerinckx2017multiscale,DuerinckxGloria2020b} for a systematic discussion.  

\begin{definition}[Spectral gaps] \label{sgs} Let $L \geq 1$. Let $(\Omega, \mathcal{S}, \mathbb{P})$ be a probability space and $(\Omega\lol, \mathcal{S}\lol, \mathbb{P}\lol)$ its $L$-periodic approximation.
\begin{enumerate}[label=(\roman*)]
\item \label{def:spectral:gap} We say that $(\Omega,\mathcal{S}, \mathbb{P})$ satisfies a spectral gap estimate with constant $\rho>0$ if the following holds:
For any random variable $\mathcal{F}: \Omega \to \R$, we have
\begin{equation*}
\erwartung{|\F-\erwartung{\F}|^2}\leq \frac{1}{\rho^2} \erwartung{\int_{\R}\brac{\int_{B_1(s)} \left| \frac{\partial\F}{\partial \omega} \right|}^2 ds},
\end{equation*}
where $\int_{B_1(s)} \left| \frac{\partial\F}{\partial \omega} \right|$ denotes
\begin{equation*}
\sup_{\delta \omega} \limsup_{t\to 0} \frac{|\F(\omega+t \delta \omega)-\F(\omega)|}{t},
\end{equation*}
where the supremum is taken over all configurations $\delta \omega: \R\to \R$ supported in $B_1(s)$ with $\norm{\delta \omega}_{L^{\infty}(\R)}\leq 1$.

\item \label{definition:379}(\textit{Periodic spectral gap}.) We say that $(\Omega\lol,\mathcal{S}\lol, \mathbb{P}\lol)$ satisfies a periodic spectral gap estimate with constant $\rho>0$ if the following holds: For any random variable $\F\lol: \Omega\lol \to \R$, we have
\begin{equation*}
\mbbE\lol\left[ |\F - \mbbE\lol[\F]|^2 \right]\leq \frac{1}{\rho^2}\mbbE\lol\left[ \int_{0}^L \brac{\int_{B(s)}\left| \frac{\partial \F}{\partial \omega\lol}\right|}^2 ds\right],
\end{equation*}
where $\int_{B(s)}\left| \frac{\partial \F}{\partial \omega\lol}\right|$ denotes
\begin{equation*}
\sup_{\delta\omega\lol} \limsup_{t\to 0}\frac{|\F(\omega\lol + t \delta \omega\lol)-\F(\omega\lol)|}{t},
\end{equation*}
where the supremum is taken over all $L$-periodic configurations $\delta \omega\lol: \R \to \R$ supported in $B_1(s)+L\Z$ and with $\norm{\delta \omega\lol}_{L^{\infty}(0,L)}\leq 1$. 

\end{enumerate}
\end{definition}

\begin{remark}[Example of Gaussian type]It is convenient to give examples of admissible distributions in the class of stationary Gaussian random fields, since the latter are uniquely determined by their covariance functions and allow to apply Malliavin calculus to establish spectral gap inequalities. A specific example is the following: Suppose that $\mathbb P$ describes a stationary, centered Gaussian random field $\omega:\R\to\R$ with a bounded and compactly supported covariance function $\mathcal C(s):={\rm Cov}(\omega(x+s),\omega(s))$ satisfying  $|\mathcal C(s)|=0$ for $|s|\geq\ell/2$ for some correlation length $\ell>0$. Then $(\Omega,\mathcal S,\mathbb P)$ satisfies the spectral gap estimate of Definition~\ref{sgs} (i), e.g.~see~\cite{duerinckx2017multiscale}. For $L\geq 4\ell$ we may define an $L$-periodic approximation by periodizing the covariance function: Set $\mathcal C_L(x):=\mathcal C(\{x\}_L)$, where $\{x\}_L\in[-\frac{L}2,\frac{L}2)$ is uniquely defined by $x-\{x\}_L\in L\Z$. Let $\mathbb P_L$ describe the stationary, centered Gaussian random field with covariance function $\mathcal C_L$. Then $\mathbb P_L$ is indeed an $L$-periodic approximation of $\mathbb P$ in the sense of Definition~\ref{def:pl} and additionally satisfies the spectral gap estimate of Definition~\ref{sgs} (ii) with a constant $\rho$ that is independent of $L$.
\end{remark}

We present our main quantitative result:
\begin{theorem}[Quantitative RVE approximations]\label{Theorem:2} Suppose Assumption~\ref{ass:W:1} is satisfied and in addition that $W$ satisfies 
\begin{equation*}
\norm{D^4W(\omega,\cdot)}_{C(\overline{U_{\alpha}})} + \|\partial_\omega W(\omega,\cdot)\|_{C^3({\overline{U_\alpha}})}\leq \frac1\alpha  \quad \text{for all }\omega \in \R.
\end{equation*}
We let $\brac{\Omega, \mathcal S, \mathbb P}$ satisfy \ref{stationarity} and the spectral gap estimate with constant $\rho>0$ of Definition\nbs\ref{sgs} (i). Let $L\geq 2$ and let $\brac{\Omega,\mathcal{S}, \mathbb P_L}$ be an $L$-periodic approximation of $\mathbb P$ in the sense of Definition~\ref{def:pl} which satisfies the periodic spectral gap estimate with constant $\rho>0$ of Definition \ref{sgs} (ii). $W_{\hom,L}$ denotes the corresponding representative volume element approximation defined in \eqref{eq:364:2}. 

Then, there exists $c=c(\alpha,p, d, \rho)\in[1,\infty)$ such that for all $F \in U_{\overline{\delta}}$ with $\overline\delta>0$ as in Theorem~\ref{T:1:0} the following statements hold true:
\begin{enumerate}[label = (\roman*)]
\item (\textit{Estimate on random fluctuations}). 
\begin{align}
|W_{\mathrm{hom},L}(\cdot,F)- \erwartunglol{W_{\mathrm{hom},L}(F)}|& \leq \mathcal C(\cdot) \dist^2(F,\SO d)L^{-\frac12},\label{est:fluctwhl}\\
|DW_{\mathrm{hom},L}(\cdot,F)- \erwartunglol{DW_{\mathrm{hom},L}(F)}| &  \leq \mathcal C(\cdot) \dist(F,\SO d)L^{-\frac12},\label{est:fluctdwhl} \\
|D^2W_{\mathrm{hom},L}(\cdot,F)- \erwartunglol{D^2W_{\mathrm{hom},L}(F)}| & \leq \mathcal C(\cdot)  L^{-\frac12},\label{est:fluctddwhl}
\end{align} 
where $\mathcal C$ denotes a random variable  satisfying  
\begin{equation*}
\erwartunglol{\exp(\tfrac1c\mathcal C)}\leq 2.
\end{equation*}

\item (\textit{Systematic error.}) 
\begin{align*}
|\erwartunglol{W_{\mathrm{hom},L}(F)}-W\h(F)| & \leq c \dist^2(F,\SO d) \frac{\ln L}{L},\\
|\erwartunglol{DW_{\mathrm{hom},L}(F)}-DW\h(F)| & \leq c\dist(F,\SO d) \frac{\ln L}{L},\\
|\erwartunglol{D^2W_{\mathrm{hom},L}(F)}-D^2W\h(F)| & \leq c \frac{\ln L}{L}.
\end{align*}
\end{enumerate}
\end{theorem}
The proof is presented in Section \ref{sec:rve}.

The total error is dominated by the random component of order $L^{-\frac{1}{2}}$. However, the deviations from the mean may be reduced by empirical averaging: For $N \in \N$, we consider  $\cb{\omega_i}_{i \in \cb{1,...,N}}$ $N$ independent copies of the coefficient field sampled according to $\mathbb{P}\lol$  and we define the Monte Carlo approximation by
\begin{equation*}
W_{\mathrm{hom},L,N}(\omega,F) = \frac{1}{N} \sum_{i=1}^{N} W_{\mathrm{hom},L}(\omega_i,F).
\end{equation*}
With Theorem~\ref{Theorem:2} at hand, we obtain the following:
\begin{corollary}[Monte Carlo approximation]\label{cor:T:2} Let the assumptions of Theorem\nobreakspace\ref{Theorem:2} be satisfied. There exists $c= c(\alpha,d,p,\rho)>0$ such that for all $F\in U_{\overline{\delta}}$ with $\overline \delta>0$ as in Theorem~\ref{Theorem:2} and for $\ell\in\{0,1,2\}$ 
\begin{align*}
\erwartunglol{|D^\ell W_{\mathrm{hom},L,N}(F) - D^\ell W_{\rm hom}(F)|^2}^\frac12 & \leq c \; \mathrm{dist}^{2-\ell}(F,\SO d)\biggl(\frac1{\sqrt{NL}}+\frac{\log(L)}{L}\biggr).
\end{align*}
\end{corollary}

The proof of Corollary~\ref{cor:T:2} is standard and thus omitted here. Note that the total $L^2$-error is of order $L^{-1}\ln L$ if we choose $N \sim \frac{L}{\ln^2 L}$.
\section{Deterministic intermediate results}\label{sec:determ}
In the following we present the deterministic ingredients of our analysis. In particular, in the first part we recall a matching convex lower bound construction that allows us to relate the nonconvex corrector problem with a convex corrector problem. In the second part we collect some standard results about elliptic equations, and in the third part we establish deterministic Lipschitz estimates for correctors that rely on the laminate structure of the material.
\subsection{Reduction to a convex problem}

The proof of Theorem~\ref{T:1:0} follows the strategy of \cite{neukamm2018quantitative}. The starting point is the observation that $W\in\mathcal W_\alpha^p$ implies the existence of a ``matching convex lower bound''. For the precise statement we introduce the following class of strongly convex functions:
\begin{definition}[Convex energy density]\label{def:vbeta}
  For $\beta>0$ we denote by $\mathcal V_{\beta}$ the set of functions $V\in C^2(\R^{d\times d})$ satisfying for all $F,G\in\R^{d\times d}$
   \begin{align*}
   &\beta |F|^2-\frac{1}{\beta}\leq V(F)\leq \frac{1}{\beta}(|F|^2+1),\\
      & |DV(F)[G]|\leq \frac1\beta(1+|F|)|G|,\\
   &\beta|G|^2\leq D^2V(F)[G,G]\leq \frac{1}{\beta}|G|^2. 
 \end{align*}
\end{definition}

The following lemma is proven in \cite{neukamm2018quantitative} extending a construction that appeared earlier in the context of discrete energies in \cite{CDKM06,FT02}.
\begin{lemma}[Matching convex lower bound, see {\cite[Corollary~2.3]{neukamm2018quantitative}}]\label{C:wv} Let
 Assumption~\ref{ass:W:1} be satisfied. Then there exist $\delta,\mu,\beta>0$ (depending on $\alpha,d$ and $p$), and a Borel measurable function $V: \R \times \R^{d \times d} \to \R$ satisfying for almost every $\omega \in\R$, $V(\omega,\cdot)\in \mathcal V_\beta$, $V(x,\cdot)\in C^3(\R^{d\times d})$, and 
\begin{align}
  &W(\omega,F)+\mu \det F\ \geq V (\omega,F)\qquad\mbox{for all $F\in \R^{d\times d}$,}\label{WgeqV}\\
  &W(\omega,F)+\mu \det F\ =V (\omega,F)\qquad \mbox{for all }F\in\R^{d\times d}\mbox{ with }\dist(F,\SO d)<\delta,\label{W=V}\\
  &V(\omega,RF)=V(\omega,F)\qquad\text{for all }F\in \R^{d\times d}\mbox{ and }R\in\SO d.\notag
\end{align}
\end{lemma}
As a consequence of the matching property \eqref{W=V} together with the strong convexity of $V$ and the fact that $\det$ is a Null-Lagrangian, we obtain:
\begin{corollary}\label{cor:613}
We consider the assumptions and setting of Lemma\nbs\ref{C:wv}. Let $\omega \in \Omega$ and $\varphi \in W^{1,p}_{\mathrm{loc}}(\R^d; \R^d)$ satisfy
\begin{equation}\label{eq:615:3}
\dist(F+ \nabla\varphi(\cdot),\SO d)< \delta \quad \text{a.e.}
\end{equation}
Then the following statements hold:
\begin{enumerate}[label=(\roman*)]
\item For every $u\in H_{\rm uloc}^1(\R^d;\R^d)$ and $v,w\in W_{\rm loc}^{1,\infty}(\R^d;\R^d)$, the following identities hold in the sense of distributions
\begin{align}
\dive DW(\omega(x_d),F+\nabla \varphi)=&\dive DV(\omega(x_d),F+\nabla \varphi),\label{eq:cor:613:1}\\
\dive (D^2W(\omega(x_d),F+\nabla \varphi)\nabla u)=&\dive (D^2V(\omega(x_d),F+\nabla \varphi)\nabla u),\label{eq:cor:613:2}\\
\dive (D^3W(\omega(x_d),F+\nabla \varphi)\nabla v\cdot \nabla w)=&\dive (D^3V(\omega(x_d),F+\nabla \varphi)\nabla v\cdot \nabla w).\label{eq:cor:613:3}
\end{align}
\item There exists $c=c(\alpha,d,p)>0$ such that $\mathbb{L}(x):=D^2W(\omega(x_d),F+ \nabla \varphi(x))$ satisfies
\begin{align*}
\int_{\R^d} |\nabla \eta|^2 dx & \leq c \int_{\R}\mathbb{L}(x)\nabla \eta \cdot \nabla \eta dx \quad \text{for all }\eta \in C^{\infty}_{c}(\R^d; \R^d),\\
|\mathbb{L}(\cdot)| & \leq c \quad \text{a.e.}
\end{align*} 
\end{enumerate} 
\end{corollary}
The proof is presented in Section\nbs\ref{det:proofs}. 

This corollary implies that the corrector equations corresponding to $W$ and $V$ have the same solutions for $F$ in a sufficiently small neighborhood of $\SO d$, see, e.g., the proof of Lemma \ref{lemma:388:new}. In particular, this also implies that 
\begin{equation*}
W_{\mathrm{hom},L}(F)+ \mu \det(F) = V_{\mathrm{hom},L}(F), \quad W_{\mathrm{hom}}(F) + \mu \det(F) = V_{\mathrm{hom}}(F),
\end{equation*}
where $V_{\mathrm{hom}}$ is the homogenized integrand corresponding to $V$ and $V_{\mathrm{hom},L}$ its corresponding RVE approximation, see Section \ref{sec:pr:qual} for the details. 
\subsection{Standard estimates for elliptic equations}

We formulate three simple deterministic existence, uniqueness and regularity statements for uniformly elliptic equations. 

\begin{definition}
Let $d,m \in \N$. For given $0<\lambda\leq\Lambda<\infty$ we denote by $\mathcal A_{\lambda}^\Lambda$ the set of functions $A:\R^{m\times d}\to\R^{m\times d}$ satisfying for all $F_1,F_2\in\R^{m\times d}$.
 \begin{align}
\begin{split}\label{eq:202}
\brac{A(F_1)-A(F_2)}\cdot (F_1-F_2)\geq \lambda |F_1-F_2|^2,\\
|A(F_1)-A(F_2)|\leq \Lambda |F_1-F_2|, \quad A(0)=0.
\end{split}
\end{align}
\end{definition}

\begin{lemma}\label{lemma:222} For $d,m \in \N$, let $A: \R^{d} \times \R^{m\times d} \to \R^{m\times d}$ be Borel measurable and there exists $0<\lambda\leq\Lambda<\infty$ such that $A(x,\cdot)\in\mathcal A_\lambda^\lambda$ for all $x\in \R^d$. Then the following statements hold:
\begin{enumerate}[label=(\roman*)]
\item For every $g \in L^2(\R^d;\R^{m\times d})$, $f\in L^2(\R^d;\R^m)$ and $T>0$ there exists a unique weak solution $\varphi \in H^1(\R^d;\R^m)$ of
\begin{equation}\label{eq:218}
\frac{1}{T}\varphi- \dive A(x,\nabla \varphi) = \dive g + f \quad \text{in }\R^d,
\end{equation}
and it holds
\begin{equation}\label{est:monotonesystemfull}
\int_{\R^d}\frac{1}{T}|\varphi|^2+ \lambda|\nabla \varphi|^2 dx \leq  \int_{\R^d}\frac1\lambda|g|^2 + T |f|^2 dx.
\end{equation}
\item For every $g \in L^{\infty}(\R^d;\R^{m\times d})$, $f\in L^2(\R^d;\R^m)$ and $T>0$, there exists a unique $\varphi \in H^1_{\mathrm{uloc}}(\R^d;\R^d)$ satisfying equation \eqref{eq:218}. Moreover, there exists $c\in[1,\infty)$ such that for all $x_0 \in \R^d$, $R\geq \sqrt{T}$,
\begin{align*}
\int_{\R^d}(\frac{1}{T} |\varphi|^2 + \lambda|\nabla \varphi|^2)\eta dx & \leq c \int_{\R^d}  \biggl(\frac1\lambda|g|^2+T|f|^2\biggr) \eta dx
\end{align*}
where $\eta: \R^d\to \R$ is given by $\eta(x)=\exp(-\frac{\gamma}{R}|x-x_0|)$ with $\gamma= \gamma(\frac{\lambda}{\Lambda})\in(0,1]$. In particular, we have
\begin{equation}\label{eq:378}
\dashint_{B_{R}(x_0)}\frac{1}{T}|\varphi|^2 + \lambda |\nabla \varphi|^2 dx \leq c \brac{\frac1{\lambda\gamma^d}\norm{g}^2_{L^{\infty}(\R^d)}+\frac{1}{R^d} \int_{\R^d}T \eta |f|^2 dx}.
\end{equation}
\end{enumerate}
\end{lemma}
Part (i) is completely standard and part (ii) follows by testing \eqref{eq:218} with $\eta\varphi$ and a suitable choice of $\gamma$, see {\cite[Lemma~36]{fischer2019optimal} for details.

The following simple lemma is central to the analysis of the present paper. Based on the observation that in the case of laminates the PDE \eqref{eq:218} reduces to an ODE, we establish the following crucial Lipschitz estimate:
\begin{lemma}\label{lemma:279} 
Let $d \in \N$. Let $A: \R \times \R^{d\times d}\to \R^{d\times d}$ be a Borel measurable  function and suppose that there exist $\Lambda\geq\lambda>0$ such that $A(z,\cdot)\in\mathcal A_\lambda^\Lambda$ for all $z\in \R$. For $T>0$ and $g\in L^{\infty}(\R)^{d\times d}$, let $\varphi\in H^1_{{\rm uloc}}(\R^d;\R^d)$ be the unique solution to
\begin{equation*}
\frac{1}{T} \varphi(x) - \dive A(x_d,\nabla \varphi(x)) = \dive (g(x_d)) \quad \text{in }\R^d.
\end{equation*}
Then the following statements are true:
\begin{enumerate}[label=(\roman*)]
\item $\varphi$ is one-dimensional in the sense that $\varphi(x)=\varphi(x_d)$ and it solves 
\begin{equation}\label{eq:616}
\frac{1}{T} \varphi(x_d) - \partial_{x_d} a(x_d,\partial_d \varphi(x_d))= \partial_{x_d}g_d(x_d) \quad \text{in }\R,
\end{equation}
where $a: \R \times \R^d \to \R^d$ is given by $a(x_d,f)= A(x_d,f\otimes e_d)e_d$, and $g_d(x_d)=g(x_d)e_d$.

If we replace $g\in L^{\infty}(\R; \R^{d\times d})$ by $g \in L^2(\R; \R^{d\times d})$, the analogous claim holds for the weak solution of the equation.
\item There exists $c=c(\frac\Lambda\lambda)\in[1,\infty)$ such that
\begin{equation}\label{lemma:279:est}
\norm{\nabla \varphi}_{L^{\infty}(\R^d)} \leq \frac{c}\lambda \norm{g}_{L^{\infty}(\R)}.
\end{equation}
\end{enumerate}
\end{lemma}
The proof is presented in Section\nobreakspace\ref{det:proofs}.

\begin{lemma}\label{lem:425} Let $L> 0$ and $d,m \in \N$. Let $A: (0,L) \times \R^{d\times d}\to \R^{d\times d}$ be a Borel measurable  function and suppose that there exist $\Lambda\geq\lambda>0$ such that $A(z,\cdot)\in\mathcal A_\lambda^\Lambda$ for all $z\in(0,L)$. For given $g\in L^\infty((0,L);\R^{d\times d})$, there exists $\varphi \in H^1_{\mathrm{per}}(\square_L;\R^d)$ a unique weak solution to
\begin{equation*}
-\dive A(x_d,\nabla \varphi(x)) = \dive \brac{g(x_d)} \quad \text{in }\square_L
\end{equation*}
satisfying $\dashint_{\square_L}\varphi \; dx = 0$.
Moreover, $\varphi$ is one-dimensional in the sense that $\varphi(x) = \varphi(x_d)$ and it weakly solves 
\begin{equation*}
- \partial_{x_d} a(x_d,\partial_d \varphi(x_d))= \partial_{x_d}g_d(x_d) \quad \text{in }\R,
\end{equation*}
where $a$ and $g_d$ are defined as in Lemma~\ref{lemma:279}. Also, there exists $c=c(\lambda,\Lambda)\in[1,\infty)$ such that
\begin{equation*}
\norm{\nabla \varphi}_{L^{\infty}(\square_L)} \leq c\norm{g}_{L^{\infty}(0,L)}.
\end{equation*}
\end{lemma}

The proof of Lemma~\ref{lem:425} is similar to the proof of Lemma~\ref{lemma:279} and it is thus omitted.

\begin{remark}\label{rem:layereded}
The properties \eqref{eq:202} imply that the operator $a:\R\times \R^{d}\to \R^d$ given in Lemmas\nobreakspace\ref{lemma:279} and \ref{lem:425} satisfies for all $x_d$ and all $f_1,f_2\in \R^d$
\begin{align*}
(a(x_d,f_1)-a(x_d,f_2))\cdot (f_1-f_2)\geq \lambda |f_1-f_2|^2,\\
|a(x_d,f_1)-a(x_d,f_2)|\leq \Lambda |f_1-f_2|, \quad a(x_d,0)=0.
\end{align*}
If additionally $A(x_d,\cdot)$ is linear and symmetric, then $a$ is also linear and symmetric.
\end{remark}

\subsection{Lipschitz estimates and differentiability of localized correctors}

For $T>0$, $F\in \R^{d \times d}$ and $\omega \in \Omega$, we consider the following localized corrector equation
\begin{equation}\label{eq:372:new}
\frac{1}{T}\varphi^{T}_{F} - \dive DW(\omega(x_d),F + \nabla \varphi^{T}_{F}) = 0 \quad \text{in }\R^d.
\end{equation}

\begin{lemma}\label{lemma:388:new} Let $T>0$ and $\omega \in \Omega$. Suppose Assumption~\ref{ass:W:1} is satisfied. Then there exist $\overline{\delta}=\overline{\delta}(\alpha,p,d)>0$ and $c=c(\alpha,p,d)>0$ such that the following statements hold:
\begin{enumerate}[label=(\roman*)]
\item For all $F \in U_{\overline{\delta}}$, there exists a unique solution $\varphi_{F}^T \in W^{1,\infty}_{\mathrm{uloc}}(\R^d; \R^d)$ to \eqref{eq:372:new} which satisfies 
$$
\|F+\nabla \varphi_{F}^T\|_{L^\infty(\R^d)}<\delta,
$$
where $\delta=\delta(\alpha,d,p)>0$ is as in Lemma~\ref{C:wv}. Moreover, $\varft$ is one-dimensional in the sense that $\varft(x)= \varft(x_d)$ and satisfies
\begin{equation}\label{eq:393:new}
\norm{\nabla \varft}_{L^{\infty}(\R^d)} \leq c \dist(F,\SO d).
\end{equation}

\item The function $U_{\overline{\delta}}\ni F \mapsto \varft \in W_{\rm uloc}^{1,\infty}(\R^d;\R^d)$ is $C^2$ in the sense of Fr{\'e}chet. The first derivative $\varphi_{F,G}^T:=D_F\varft(F)G \in W^{1,\infty}_{\mathrm{uloc}}(\R^d;\R^d)$ at $F \in U_{\overline{\delta}}$ in direction $G\in \R^{d\times d}$ is characterized by 
\begin{equation}\label{eq:516:new}
\frac{1}{T}\varfgt - \dive \brac{D^2W(\omega(x_d),F+ \nabla \varft)\brac{G+ \nabla \varfgt}}= 0 \quad \text{in }\R^d.
\end{equation}
Moreover, it is one-dimensional in the sense that $\varfgt(x)= \varfgt(x_d)$  and satisfies
\begin{equation}\label{eq:520:new}
\norm{\nabla \varfgt}_{L^{\infty}(\R^d)} \leq c|G| \quad \text{for all }G \in \R^{d \times d}.
\end{equation}

\item The second derivative $\varphi^T_{F,G,H} := D^2_F \varft HG$ at $F\in U_{\overline \delta}$ with $G,H\in\R^{d\times d}$ is characterized by
\begin{eqnarray}\label{eq:611:new}
& & \frac{1}{T}\varfght - \dive\brac{D^2W(\omega(x_d),F+ \nabla \varft)\nabla \varfght} \nonumber \\ & = & \dive \brac{D^3W(\omega(x_d),F + \nabla \varft)\brac{H+ \nabla \varphi_{F,H}^T }\brac{G+ \nabla \varfgt}} \quad \text{in }\R^d. 
\end{eqnarray}
Moreover, it is one-dimensional in the sense that $\varfght(x)= \varfght(x_d)$  and satisfies
\begin{equation}\label{eq:615:new}
\norm{\nabla \varfght}_{L^{\infty}(\R^d)} \leq c |H||G| \quad \text{for all }G,H \in \R^{d \times d}.
\end{equation}
\end{enumerate}
\end{lemma}
The proof is presented in Section\nobreakspace\ref{det:proofs}.

For the proof of the systematic error estimates of Theorem\nbs\ref{Theorem:2}, we need the following result. In the following Lemma, we denote by $\varft,\, \varfgt,\, \varfght$ the correctors from Lemma\nobreakspace\ref{lemma:388:new} and by $\hatvarft,\, \hatvarfgt,\, \hatvarfght$ the same objects if we replace $\omega\in \Omega$ by another realization $\widehat{\omega}\in \Omega$. 
\begin{lemma}\label{L:esttruncatedcorrector}
Suppose that Assumption~\ref{ass:W:1} is satisfied and that it holds $W(\omega,\cdot)\in C^4(\overline{U_{\alpha}})$ with 
\begin{equation*}
\norm{D^4W(\omega,\cdot)}_{C(\overline{U_{\alpha}})} \leq \frac1\alpha \quad \text{for all }\omega \in \R.  
\end{equation*}
Let $T,L>0$, $\omega, \widehat{\omega} \in \Omega$ with $\omega(x_d) = \widehat{\omega}(x_d)$ for $x_d \in [-\frac{L}{4},\frac{L}{4}]$. Then there exists $c=c(\alpha,p,d)\in[1,\infty)$ such that for all $F \in U_{\overline{\delta}}$ with  $\overline{\delta}=\overline{\delta}(\alpha,p,d)>0$ as in Lemma~\ref{lemma:388:new}, we have
\begin{align}
 \dashint_{B_{\sqrt{T}}} |\nabla \hatvarft - \nabla \varft|^2 dx_d \leq& c \dist^2(F,\SO d) \exp\brac{-\frac{1}{c} \frac{L}{\sqrt{T}}}, \label{eq:1122} \\
 \dashint_{B_{\sqrt{T}}}  |\nabla \hatvarfgt - \nabla \varfgt|^2 dx_d \leq& c |G|^2 \exp\brac{-\frac{1}{c}\frac{L}{\sqrt{T}}}, \label{eq:1123}\\
\dashint_{B_{\sqrt{T}}} |\nabla \hatvarfght - \nabla \varfght|^2 dx_d  \leq & c|G|^2 |H|^2 \exp\brac{-\frac{1}{c}\frac{L}{\sqrt{T}}}. \label{eq:1124}
\end{align}
\end{lemma}
The proof is presented in Section\nobreakspace\ref{det:proofs}.

\subsection{Proofs of deterministic intermediate results}\label{det:proofs}

\begin{proof}[Proof of Corollary \ref{cor:613}]
\step 1 Proof of part (i).

Equation \eqref{eq:cor:613:1} is a direct consequence of \eqref{eq:615:3}, \eqref{W=V} and the fact that $\dive D\det(F+ \nabla\varphi)=0$ in the sense of distributions since $\det$ is a Null-Lagrangian.

Next we show \eqref{eq:cor:613:2}. We suppose $u\in C^\infty(\R^d;\R^d)$, the general claim follows by approximation. For every $\eta\in C_c^\infty(\R^d;\R^d)$ there exists $h_0>0$ such that  $\|\dist(F+\nabla \varphi+h\nabla u,\SO d)\|_{L^\infty({\rm supp}\eta)}<\delta$ $\forall h\in[0,h_0]$ and thus
\begin{align*}
&\frac1h\int_{\R^d}(DW(\omega(x_d),F+\nabla \varphi+h\nabla u)-DW(\omega(x_d),F+\nabla \varphi))\cdot\nabla \eta\,dx\\
=&\frac1h\int_{\R^d}(DV(\omega(x_d),F+\nabla \varphi+h\nabla u)-DV(\omega(x_d),F+\nabla \varphi))\cdot\nabla \eta\,dx\qquad \forall h\in(0,h_0].
\end{align*}
Letting $h\to0$, we obtain
$
 \int_{\R^d}D^2W(\omega(x_d),F+\nabla \varphi)\nabla u\cdot\nabla \eta\,dx=\int_{\R^d}D^2V(\omega(x_d),F+\nabla \varphi)\nabla u\cdot\nabla \eta\,dx
$
and \eqref{eq:cor:613:2} follows from the arbitrariness of $\eta\in C_c^\infty(\R^d;\R^d)$. The argument for \eqref{eq:cor:613:3} is analogous and left to the reader.

\step 2 By \eqref{eq:615:3} and \eqref{W=V} we have
\begin{align*}
\int_{\R^d}\mathbb{L}(x)\nabla \eta \cdot \nabla \eta dx & = \int_{\R^d} D^2V(\omega(x_d),F+ \nabla \varphi)\nabla \eta \cdot \nabla \eta - \mu D^2\det(F+ \nabla \varphi)\nabla \eta \cdot \nabla \eta dx \\
& \geq \beta \int_{\R^d}|\nabla \eta|^2 dx,
\end{align*}
where the inequality follows by strong convexity of $V(\omega(x_d),\cdot)$ and $\int_{\R^d}D^2\det(F+\nabla \varphi)\nabla \eta \cdot \nabla \eta = 0$,
which follows from the fact that $\det$ is a Null-Lagrangian. Estimate $|\mathbb{L}(\cdot)|\leq c$ is a consequence of $W\in C^2(\overline{U_{\delta}})$.
\end{proof}

\begin{proof}[Proof of Lemma~\ref{lemma:279}]
\step 1 Proof of part (i). For any $h \in \R$ and $i\in \cb{1,..., d-1}$, consider $w:=\varphi(\cdot+he_i)-\varphi$. Obviously, we have $w\in H_{\rm uloc}^1(\R^d;\R^m)$ and 
\begin{equation*}
\frac{1}{T} w(x) - \dive\brac{A(x_d,\nabla \varphi(x) +\nabla w(x))-A(x_d,\nabla \varphi(x))} = 0 \quad \text{in }\R^d.
\end{equation*}
Hence, $w\equiv0$ and thus $\varphi$ depends only on $x_d$ and satisfies \eqref{eq:616}.

\step 2 Proof of part (ii).

Without loss of generality we suppose $\lambda=1$, the general case follows by replacing $a$, $T$ and $g$ by $a/\lambda$, $T\lambda$ and $g/\lambda$. Assumption \eqref{eq:202} and Remark~\ref{rem:layereded} imply that for any ball $B\subset \R$ we have
\begin{align}\label{eq:672}
\begin{split}
\dashint_{B} |\partial_{x_d} \varphi|^2\, dx_d & \leq  \dashint_{B} |a(x_d,\partial_{x_d}\varphi)|^2\, dx_d \\ & \leq 2 \brac{\dashint_{B}|a(x_d,\partial_{x_d}\varphi)+g_d|^2 dx_d + \dashint_{B}|g_d|^2 dx_d} \\ & \leq 2\brac{\dashint_{B}|\Psi(x_d)|^2 dx_d + \norm{g}_{L^{\infty}(\R)}^2},
\end{split}
\end{align}
where we use the shorthand $\Psi(x_d):=a(x_d,\partial_{x_d} \varphi(x_d)\otimes e_d)+ g_d(x_d)$.

We prove the following: There exists $c=c(\Lambda)\in[1,\infty)$ such that for every $x_0\in \R$ it holds

\begin{equation}\label{eq:lemmalip:keyest}
\sup_{r\in (0,1)} \dashint_{r \widetilde{B}}|\Psi(x_d)|^2dx_{d} \leq c \norm{g}_{L^{\infty}(\R)}^2\quad\mbox{where}\quad r\widetilde{B}:=rB_{\sqrt{T}}(x_0) \subset \R.
\end{equation}
Clearly, \eqref{eq:lemmalip:keyest} together with \eqref{eq:672} and the arbitrariness of $x_0$ imply the claimed estimate \eqref{lemma:279:est}.

In the following we use the notation $(u)_{B}:= \dashint_{B}u(x_d) dx_d$

We have
\begin{equation*}
\dashint_{r \widetilde{B}}|\Psi(x_d)|^2dx_{d} \leq 2 \dashint_{r \widetilde{B}}|\Psi(x_d)-(\Psi)_{r\widetilde{B}}|^2dx_{d}+ 2 |(\Psi)_{r\widetilde{B}}|^2,
\end{equation*}
where we use the shorthand $(u)_{B}:= \dashint_{B}u(x_d) dx_d$. With help of the (one-dimensional) Poincar{\'e} inequality, equation \eqref{eq:616}, and \eqref{eq:378}, we estimate the first term on the right-hand side 
\begin{eqnarray}\label{est:psidiff}
\dashint_{r \widetilde{B}}|\Psi(x_d)-(\Psi)_{r\widetilde{B}}|^2dx_{d} & \leq &  |r\widetilde{B}|^2 \dashint_{r\widetilde{B}}|\partial_{d}\Psi(x_d)|^2 dx_d 
 \stackrel{\eqref{eq:616}}{\leq}  |r\widetilde{B}| \int_{r\widetilde B} |\frac{1}{T} \varphi|^2 dx_d \notag \\ 
 & = &  r \dashint_{B_{\sqrt{T}}(x_0)} \frac{1}{T} |\varphi|^2 dx_d
  \stackrel{\eqref{eq:378}}{\leq}  cr \norm{g}_{L^{\infty}(\R)}^2.
\end{eqnarray}
To estimate $|(\Psi)_{r\widetilde{B}}|$, we use \eqref{est:psidiff} and a standard dyadic argument. Firstly, we observe that for $r\in[\frac14,1]$, there exists $c=c(\Lambda)\in[1,\infty)$ such that 
\begin{equation}\label{eq:696}
|(\Psi)_{r\widetilde{B}}| \leq 2 \brac{\dashint_{B_{\sqrt{T}}}|\Psi|^2 dx_d}^{\frac{1}{2}} \leq 2 \brac{\Lambda\brac{\dashint_{B_{\sqrt{T}}}|\partial_{d} \varphi|^2 dx_d}^{\frac{1}{2}}+\norm{g}_{L^{\infty}(\R)}} \stackrel{ \eqref{eq:378}}{\leq} c \norm{g}_{L^{\infty}(\R)}.
\end{equation}
For $r\in (0,\frac{1}{4})$, we choose $q \in (\frac{1}{4},\frac{1}{2})$ and $k\in \N$ satisfying $r= q^k$. We have 
\begin{align*}
|(\Psi)_{r\widetilde{B}}| & \leq  |\sum_{i=0}^{k-1} (\Psi)_{q^{i+1}\widetilde{B}}- (\Psi)_{q^{i}\widetilde{B}}| +  |(\Psi)_{\widetilde{B}}|
\stackrel{\eqref{eq:696}}{\leq}  \sum_{i=0}^{k-1}|(\Psi)_{q^{i+1}\widetilde{B}}- (\Psi)_{q^{i}\widetilde{B}}|+ c(\tfrac\Lambda\lambda)\norm{g}_{L^{\infty}(\R)}.
\end{align*}
Using
\begin{align*}
|(\Psi)_{q^{i+1}\widetilde{B}}- (\Psi)_{q^{i}\widetilde{B}}| & \leq \dashint_{q^{i+1}\widetilde{B}}|\Psi - (\Psi)_{q^i\widetilde{B}}| dx_d\\
  &\leq \frac1q\brac{\dashint_{q^{i}\widetilde{B}}|\Psi-(\Psi)_{q^i\widetilde{B}}|^2 dx_d}^{\frac{1}{2}} \stackrel{ \eqref{est:psidiff}}{\leq} c(\sqrt{q})^{i-2} \norm{g}_{L^{\infty}(\R)}
\end{align*}
we obtain
\begin{align*}
\sum_{i=0}^{k-1}|(\Psi)_{q^{i+1}\widetilde{B}}- (\Psi)_{q^{i}\widetilde{B}}| \leq \frac{c}{q} \norm{g}_{L^{\infty}(\R)} \sum_{i=0}^{k-1}(\sqrt{q})^i & \leq \frac{c}{q(1-\sqrt q)} \norm{g}_{L^{\infty}(\R)}
\end{align*}
which concludes the proof.

\end{proof}

\begin{proof}[Proof of Lemma~\ref{lemma:388:new}]

Throughout the proof we write $\lesssim$ if $\leq$ holds up to a multiplicative constant depending only on $\alpha,d$ and $p$.

\step 1 Proof of part (i).

First, we consider the convex lower bound  $V$ for $W$ from Lemma\nobreakspace\ref{C:wv} and the equation
\begin{equation}\label{eq:795:2}
\frac{1}{T}\varphi^{T}_{F} - \dive DV(\omega(x_d),F + \nabla \varphi^{T}_{F}) = 0 \quad \text{in }\R^d.
\end{equation}
Since $DW(\omega(x_d),R) = 0$ and thus $DV(\omega(x_d),R)= \mu D\det(R)$ for any $R\in \SO d$, we may rewrite this equation as
\begin{equation*}
\frac{1}{T}\varft -\dive A(x_d,\nabla \varft) = \dive g  \quad \text{in }\R^d,
\end{equation*}
with
\begin{equation*}
A(x_d,G):=DV(\omega(x_d),F+G)-DV(\omega(x_d),F)\quad\mbox{and}\quad g(x_d):=DV(\omega(x_d),F)-DV(\omega(x_d),R).
\end{equation*}
Note that $A$ satisfies the assumptions of Lemma \ref{lemma:279} and it holds $\|g\|_{L^\infty(\R)} \lesssim |F-R|$.  
Hence, Lemma~\ref{lemma:222} and Lemma~\ref{lemma:279} imply that there exists a unique $\varphi_F^T\in H_{\rm uloc}^1(\R^d;\R^d)$ satisfying \eqref{eq:795:2} and which is one-dimensional and satisfies
$\norm{\nabla\varft}_{L^{\infty}(\R)} \lesssim  |F-R|$. Minimizing over $R\in \SO d$ we obtain $\norm{\nabla\varft}_{L^{\infty}(\R)} \lesssim \dist(F,\SO d)$. Hence, we find $\overline \delta=\overline\delta(\alpha,d,p)>0$ such that for all $F\in U_{\overline \delta}$, it holds
\begin{equation}\label{eq:fpvarftinudelta}
\|\dist(F+ \nabla \varft,\SO d)\|_{L^\infty(\R^d)} <\delta,
\end{equation}
with $\delta>0$ as in Lemma~\ref{C:wv}. Combining, \eqref{eq:fpvarftinudelta} with Corollary\nbs\ref{cor:613}\nbs(i), we obtain that for all $F\in \SO d$, $\varft$ solves \eqref{eq:372:new}. On the other hand, any solution of \eqref{eq:372:new} satisfying \eqref{eq:fpvarftinudelta} satisfies \eqref{eq:795:2}.  However, the latter equation admits a unique solution. This implies the claimed uniqueness.

\step 2 Proof of part~(ii) except $C^2$-regularity of $F\mapsto \varphi_F^T$.

\substep{2.1} Lipschitz estimate on $F\mapsto \varphi_F$. For $F,F'\in\R^{d\times d}$ let $\varphi_F^T$ and $\varphi_{F'}^T$ be defined via \eqref{eq:795:2}. We claim
\begin{equation}\label{eq:691new}
\|\nabla (\varphi_F^T-\varphi_{F'}^T)\|_{L^\infty(\R^d)}\lesssim |F-F'|.
\end{equation}
Indeed, $\Phi := \varphi_{F}^T - \varphi_{F'}^T$ solves
\begin{eqnarray*}
& &\frac{1}{T}\Phi -\dive \brac{DV(\omega(x_d),F + \nabla \Phi + \nabla \varphi_{F'}^T)-DV(\omega(x_d),F + \nabla  \varphi_{F'}^T)} \\
 & = & \dive \brac{DV(\omega(x_d), F+ \nabla \varphi_{F'}^T)-DV(\omega(x_d),F'+ \nabla \varphi_{F'}^T)} \quad \text{in }\R^d.
\end{eqnarray*}
Lemma \ref{lemma:279} together with the Lipschitz continuity of $DV(\omega(x_d),\cdot)$ imply \eqref{eq:691new}.

\substep{2.2} Differentiability of $F\mapsto \varphi_F$. 

Lemma~\ref{lemma:279} and the one-dimensionality of $\varft$  imply that
\begin{equation}\label{eq:1068:2}
\frac{1}{T}\varfgt - \dive \brac{D^2V(\omega(x_d),F+ \nabla \varft)\brac{G+ \nabla \varfgt}}= 0 \quad \text{in }\R^d,
\end{equation}
admits a unique solution $\varfgt\in H_{\rm uloc}^1(\R^d;\R^d)$ which depends only on $x_d$ and satisfies \eqref{eq:520:new}. Combining \eqref{eq:795:2} and \eqref{eq:1068:2}, we obtain that $\Psi:= \varphi_{F+G}^T-\varft - \varfgt$ solves 
\begin{equation}\label{eq:538}
\frac{1}{T} \Psi - \dive\brac{D^2V(\omega(x_d),F+ \nabla \varft)\nabla \Psi}=\dive g \quad \text{in }\R^d,
\end{equation}
where
\begin{align*}
g:=& DV(\omega(x_d),F+G+\nabla \varphi_{F+G}^T)-DV(\omega(x_d),F+ \nabla \varft)\\
&-D^2V(\omega(x_d),F+ \nabla \varft)(G+ \nabla \varphi_{F+G}^T - \nabla \varft).
\end{align*}
For $|G|$ small enough, we have $F+G \in U_{\overline{\delta}}$ and thus $F+G + \nabla\varphi_{F+G}^T \in U_{\delta}$ a.e. Since the smoothness of $W$ and \eqref{W=V} imply $\sup_{F\in U_\delta}\|D^3V(\omega(x_d),F)\|\lesssim1$, we obtain by a Taylor expansion  
\begin{equation}\label{eq:691}
\norm{g}_{L^{\infty}(\R^d)}\lesssim \|G+ \nabla \varphi_{F+G}^T - \nabla \varft\|_{L^\infty(\R^d)}^2\stackrel{\eqref{eq:691new}}{\lesssim} |G|^2.
\end{equation}
Hence, Lemma~\ref{lemma:222} and Lemma~\ref{lemma:279} applied to \eqref{eq:538} yield
\begin{equation}\label{eq:705}
\|\nabla \Psi\|_{L^\infty(\R^d)}^2+\dashint_{B_{\sqrt{T}}(x_0)}\frac{1}{T} |\Psi|^2 + |\nabla \Psi|^2 dx \lesssim |G|^4\qquad\forall x_0\in\R^d
\end{equation}
and thus, by Poincar\'e inequality,
$
\lim_{|G|\to0}|G|^{-1}\sup_{x_0\in\R^d}\|\varphi_{F+G}^T-\varft - \varfgt\|_{W^{1,\infty}(B_1(x_0))}=0.
$
Hence, $F \mapsto \varft$ is differentiable with derivative $\varfgt=D_F\varphi_F^TG$. 

\substep{2.3} Conclusion. Let $F\in U_{\overline \delta}$ and $G\in\R^{d\times d}$ be given. Combining Corollary~\ref{cor:613} and Step~1, we observe that equation \eqref{eq:516:new} admits a unique solution in $H_{\rm uloc}^1(\R^d;\R^d)$, which coincides with the solution of \eqref{eq:1068:2}. Hence, claim (ii) (except $C^2$-regularity) is proven.

\step 3 Proof of $C^2$-regularity of $F\mapsto \varphi_F^T$ and part (iii).

Fix $F\in U_{\overline \delta}$. For every $G,H\in\R^{d\times d}$, Lemma~\ref{lemma:279} implies that 
\begin{eqnarray}\label{eq:1127:2}
& & \frac{1}{T}\varfght - \dive\brac{D^2V(\omega(x_d),F+ \nabla \varft)\nabla \varfght} \nonumber \\ & = & \dive \brac{D^3V(\omega(x_d),F + \nabla \varft)\brac{H+ \nabla \varphi_{F,H}^T }\brac{G+ \nabla \varfgt}} \quad \text{in }\R^d 
\end{eqnarray}
admits a unique solution $\varfght$ which depends only on $x_d$ and satisfies \eqref{eq:615:new}. Set $\overline{\Psi}:= \varphi_{F+H,G}^T-\varfgt -\varfght$. In view of \eqref{eq:1127:2} and \eqref{eq:1068:2}, we have
\begin{equation*}
\frac{1}{T}\overline{\Psi} - \dive \brac{D^{2}V(\omega(x_d),F+\nabla\varft)\nabla \overline{\Psi}} = \dive \brac{g_A + g_B + g_C} \quad \text{in }\R^d,
\end{equation*}
where
\begin{align*}
g_A = & \brac{D^2V(\omega(x_d),F+H+ \nabla \varphi_{F+H}^T)-D^2V(\omega(x_d),F+ \nabla\varft)}(G+ \nabla \varphi_{F+H,G}^T)\\ & -D^3V(\omega(x_d),F+\nabla \varft)(H+\nabla \varphi_{F+H}^T-\nabla \varft)(G+ \nabla \varphi_{F+H,G}^T),\\
g_B = &  - D^3V(\omega(x_d),F+\nabla \varft)(H+ \nabla\varphi_{F,H}^T)(\nabla \varfgt - \nabla \varphi_{F+H,G}^T), \\
g_C = & - D^3V(\omega(x_d),F+ \nabla \varft)(\nabla\varphi_{F,H}^T-\nabla \varphi_{F+H}^T+ \nabla \varft)(G+ \nabla \varphi_{F+H,G}^T).
\end{align*}
Obviously, the functions $g_A,\; g_B,\; g_C$ depend only on $x_d$ and we claim that
\begin{equation}\label{est:gagbgc}
\norm{g_A + g_B + g_C}_{L^{\infty}(\R^d)} \lesssim |G||H|^2.
\end{equation}
Before proving \eqref{est:gagbgc}, we note that \eqref{est:gagbgc} implies the claim. Indeed, Lemma \ref{lemma:222}, Lemma~\ref{lemma:279}  and \eqref{est:gagbgc} yield 
\begin{equation}\label{eq:123456}
\|\overline\Psi\|_{L^\infty(\R^d)}^2+\dashint_{B_{\sqrt{T}}(x_0)}\frac{1}{T} |\overline{\Psi}|^2 + |\nabla \overline{\Psi}|^2 dx\lesssim |G|^2 |H|^4.
\end{equation}
Taking the supremum over $|G|\leq 1$ and letting $H\to 0$, the  differentiability of $F\mapsto D_F\varphi_{F}^T(\cdot)$ follows. Finally, in view of Corollary~\ref{cor:613} the function $\varfght$ is the unique solution of \eqref{eq:611:new}.

\medskip

Hence it is left to prove \eqref{est:gagbgc}. For $|H|$ sufficiently small, the smoothness of $W$ and a Taylor expansion imply
\begin{equation*}
\norm{g_A}_{L^{\infty}(\R^d)} \lesssim |G|\norm{H+\nabla \varphi_{F+H}^T- \nabla \varphi_{F}^T}^2_{L^{\infty}(\R^d)} \stackrel{\eqref{eq:691new}}{\lesssim}   |G||H|^2.
\end{equation*}
Further, it holds 
$
\norm{g_B}_{L^{\infty}(\R^d)} \lesssim  |H| \norm{\nabla \varfgt- \nabla \varphi_{F+H,G}^T}_{L^{\infty}(\R^d)}.
$
We estimate the last expression by noting that $\overline{\Phi}:= \varphi_{F+H,G}^T- \varfgt$ solves
\begin{equation*}
\frac{1}{T}\overline{\Phi} - \dive \brac{D^2V(\omega(x_d),F+\nabla \varft)\nabla \overline{\Phi}} =  \dive \overline{g} \quad \text{in }\R^d,
\end{equation*}
where
$$
\overline{g}:= \brac{D^2 V(\omega(x_d),F+H+ \nabla \varphi_{F+H}^T)- D^2V(\omega(x_d),F+\nabla \varft)}(G+\nabla \varphi_{F+H,G}^T)$$
satisfies $\norm{\overline{g}}_{L^{\infty}(\R^d)}\lesssim |G||H|$. Hence, we obtain with help of Lemma \ref{lemma:279} 
\begin{equation*}
\norm{g_B}_{L^{\infty}(\R^d)} \lesssim |H| \norm{\nabla \Phi}_{L^{\infty}(\R^d)}\lesssim  |G||H|^2.
\end{equation*}
Finally, we estimate 
\begin{equation*}
\norm{g_C}_{L^{\infty}(\R^d)} \lesssim |G| \norm{\nabla \varphi_{F,H}^T - \nabla \varphi_{F+H}^T + \nabla \varft}_{L^{\infty}(\R^d)}\lesssim  |G||H|^2,
\end{equation*}
where the last inequality follows by applying Lemma \ref{lemma:279} to equation \eqref{eq:538}, cf.\nbs \eqref{eq:691}. Summarizing these bounds, we obtain the desired estimate \eqref{est:gagbgc}.
\end{proof}
\begin{proof}[Proof of Lemma~\ref{L:esttruncatedcorrector}]

Throughout the proof we suppose $F\in U_{\bar\delta}$ and we write $\lesssim$ if $\leq$ holds up to a multiplicative constant depending only on $\alpha,d$ and $p$. 

\step 1 Proof of \eqref{eq:1122}. We first note that by \eqref{eq:795:2}, $v_F:=\hatvarft- \varft$ solves
\begin{equation}\label{eq:824:2}
\frac{1}{T} v_F - \dive \brac{A(x_d,\nabla v_F)} =  \dive \brac{g_{F}(x_d)} \quad \text{in }\R^d,
\end{equation}
where 
$$
A(\cdot,G) := DV(\hatomega,F + \nabla \varft + G)-DV(\hatomega,F + \nabla \varft )
$$ 
and, for an arbitrary $R\in \SO d$,
$$
g_F:= DV(\hatomega,F+ \nabla \varft)-DV(\widetilde{\omega},R)- (DV(\omega,F+ \nabla \varft)-DV({\omega},R))).
$$
Lemma\nobreakspace\ref{lemma:279} applies to \eqref{eq:824:2} and thus Lemma~\ref{lemma:222} yields
\begin{equation}\label{eq:842:2}
\int_{\R}\eta \frac{1}{T}|v_F|^2 + \eta |\nabla v_F|^2 dx_d \lesssim  \int_{\R}\eta |g_{F}|^2 dx_d
\end{equation}
with $\eta(x_d) = \exp(-\gamma \frac{|x_d|}{\sqrt{T}})$ for some $\gamma=\gamma(\alpha,p,d)\in (0,1]$. Since $\omega=\hatomega$ in $B_{\frac{L}{4}}(0)$, we have $g_F=0$ in $B_{\frac{L}{4}}(0)$ and \eqref{eq:393:new} implies $|g_F|\lesssim |F-R| +\norm{\nabla \varft}_{L^{\infty}}$. In combination with \eqref{eq:842:2}, we obtain
\begin{equation}\label{eq:1122:stronger}
 \int_{\R}\eta \frac{1}{T}|v_F|^2 + \eta |\nabla v_F|^2 dx_d \lesssim\int_{\R}\eta |g_F|^2 dx_d \lesssim  \dist^2(F,\SO d) \sqrt{T} \exp\brac{-\frac{\gamma}{4} \frac{L}{\sqrt{T}}}.
\end{equation}
and thus \eqref{eq:1122}.

\step 2 Proof of \eqref{eq:1123}. We first note that by \eqref{eq:1068:2}, $v_{F,G}:= \hatvarfgt- \varfgt$ solves
\begin{equation*}
\frac{1}{T}v_{F,G} - \dive \brac{D^2V(\hatomega(x_d),F+ \nabla \hatvarft) \nabla v_{F,G}} = \dive g_{F,G} \quad \text{in }\R^d,
\end{equation*}
where
$$
g_{F,G}:=\brac{D^2V(\hatomega,F+ \nabla \hatvarft) -D^2V(\omega,F+ \nabla \varft)}(G+ \nabla \varfgt).
$$
For $\eta$ as in Step~1, we have
\begin{equation}\label{est:wwtildevfg}
\int_{\R}\eta \frac{1}{T}|v_{F,G}|^2 + \eta |\nabla v_{F,G}|^2 dx_d \lesssim \int_{\R}\eta |g_{F,G}|^2 dx_d.
\end{equation}
The Lipschitz estimate of Lemma \ref{lemma:388:new} (ii) yields $|G+ \nabla \varfgt| \lesssim |G|$ and we have
\begin{eqnarray*}
|g_{F,G}|&\lesssim & |D^2V(\hatomega(x_d),F+ \nabla \hatvarft)- D^2V(\omega(x_d),F+ \nabla \varft)||G| \\ 
& \leq &
|D^2V(\hatomega(x_d),F+ \nabla \hatvarft)-D^2V(\hatomega(x_d),F+ \nabla \varft)||G| \\
& & + |D^2V(\hatomega(x_d),F+ \nabla \varft)-D^2V(\omega(x_d),F+ \nabla \varft) ||G| \\
& \lesssim  &  |\nabla (\hatvarft-\varft)||G| + \mathbf{1}_{\R \setminus B_{\frac{L}{4}}}|G|.
\end{eqnarray*}
In combination with \eqref{est:wwtildevfg} and \eqref{eq:1122:stronger}, we obtain
\begin{align}\label{eq:1123:stronger}
\int_{\R}\eta |\nabla v_{F,G}|^2 dx_d  & \lesssim  |G|^2 \brac{\int_{\R}\eta |\nabla  (\hatvarft-\varft)|^2 dx_d + \int_{\R \setminus B_{\frac{L}{4}}}\eta dx_d}\notag \\
& \lesssim  |G|^2(\dist^2(F,\SO d)+1) \sqrt{T}\exp\brac{-\frac{\gamma}{4}\frac{L}{\sqrt{T}}}
\end{align}
and \eqref{eq:1123} follows.
 
\step 3 Proof of \eqref{eq:1124}. We first note that by \eqref{eq:1127:2}, $v_{F,G,H}:= \hatvarfght - \varfght$ solves
\begin{equation*}
\frac{1}{T}v_{F,G,H} - \dive\brac{D^2V(\hatomega(x_d), F + \nabla \hatvarft) \nabla v_{F,G,H}} = \dive\brac{g_{F,G,H}+ f_{F,G,H}} \quad \text{in }\R^d, 
\end{equation*}
where 
\begin{align*}
g_{F,G,H}  = & D^3V(\hatomega(x_d), F + \nabla \hatvarft)\brac{H+ \nabla \hatvarfht}\brac{G+ \hatvarfgt} \\ & -D^3V(\omega(x_d), F + \nabla \varft)\brac{H+ \nabla \varfht}\brac{G+ \varfgt}, \\ 
f_{F,G,H} = & \brac{D^2 V(\hatomega(x_d), F+ \nabla \hatvarft)-D^2 V(\omega(x_d), F+ \nabla \varft)}\nabla \varfght. 
\end{align*}
As before, we obtain
\begin{equation}\label{est:wwtildevfgh}
\int_{\R}\eta \frac{1}{T}|v_{F,G,H}|^2 + \eta |\nabla v_{F,G,H}|^2 dx_d \lesssim \int_{\R}\eta \brac{|g_{F,G,H}|^2+ |f_{F,G,H}|^2} dx_d. 
\end{equation}
Furthermore, we compute with help of the estimates in Lemma\nbs\ref{lemma:388:new} and a Taylor expansion 
\begin{align*}
|g_{F,G,H}| 
\lesssim   |H||\nabla(\widehat\varphi_{F,G}^T -\varphi_{F,G}^T)|+|G||\nabla (\widehat\varphi_{F,G}^T -\varphi_{F,G}^T)|+ (\mathbf{1}_{\R\setminus B_{\frac{L}{4}}}+ |\nabla (\hatvarft- \varft)|)|H||G|
\end{align*}
and
$
|f_{F,G,H}| 
\lesssim   \brac{\mathbf{1}_{\R\setminus B_{\frac{L}{4}}} + |\nabla(\hatvarft- \varft)|}|H||G|.
$
As a result of the previous two estimates and relying on \eqref{eq:1122:stronger} and \eqref{eq:1123:stronger}, we obtain
\begin{eqnarray*}
& & \int_{\R}\eta \brac{|g_{F,G,H}|^2 +|f_{F,G,H}|^2}dx_d
\\
& \lesssim &  \int_{\R}\eta \brac{|H|^2|\nabla v_{F,G}|^2 +|G|^2|\nabla v_{F,H}|^2+ (\mathbf{1}_{\R\setminus B_{\frac{L}{4}}}+|\nabla v_F|^2)|H|^2|G|^2}dx_d 
\\
& \lesssim & \brac{\dist^2(F,\SO d)+1}|G|^2 |H|^2 \sqrt{T}\exp\brac{-\frac{\gamma}{4}\frac{L}{\sqrt{T}}}
\end{eqnarray*}
and in combination with \eqref{est:wwtildevfgh} we have \eqref{eq:1124}.
\end{proof}

\section{Proofs of qualitative results: Theorem \ref{T:1:0} and Corollary \ref{cor:T:1}}\label{sec:pr:qual}
In this section, we denote by $V$ the matching convex lower bound for $W$ from Lemma\nbs\ref{C:wv}. We also consider the corresponding RVE approximation
\begin{equation}\label{def:VhomL}
 V_{{\rm hom},L}(\omega,F)=\inf_{\varphi\in H_{\rm per,0}^1(\square_L)}\dashint_{\square_L}V(\omega(x_d),F+\nabla \varphi)\,dx.
\end{equation}
In particular, the infimum on the right-hand side is attained by $\varphi \in H_{\rm per,0}^1(\square_L; \R^d)$ characterized by the equation
\begin{equation}\label{eq:572:3}
-\dive DV(\omega(x_d),F+ \nabla \varphi_{F}^{L}) = 0 \quad \text{in }\Box_{L}.
\end{equation}
Under the assumptions of Theorem\nbs\ref{T:1:0}, standard homogenization results imply
\begin{equation*}
V\h(F) :=  \erwartung{\dashint_{\Box}V(\omega(x_d),F+ \nabla \varphi_{F})dx}=\lim_{L \to \infty} V_{{\rm hom},L}(\omega,F) \qquad\mbox{for $\mathbb P$-a.a.\ $\omega$},
\end{equation*}
where $\varphi_{F}:\Omega\times\R^d\to\R^d$ is the unique solution to
\begin{equation}\label{eq:580:3}
-\dive DV(\omega(x_d),F + \nabla \varphi_{F}) = 0 \quad \text{in }\R^d, \quad\mbox{for $\mathbb P$-a.a.\ $\omega$}
\end{equation}
with the conditions
\begin{align}\label{eq:580:4}
& \dashint_{\Box} \varphi_{F} dx = 0, \quad \nabla \varphi_F \text{ is a stationary random field,}\quad \erwartung{\nabla \varphi_{F}}=0, \quad \erwartung{|\nabla \varphi_{F}|^2}< \infty.
\end{align}
The existence of $\varphi_F$ can be obtained by considering the solution $\varphi_F^T\in H_{\rm uloc}^1(\R^d;\R^d)$ of
\begin{equation*}
\frac{1}{T}\varft - \dive DV(\omega(x_d),F+ \nabla \varft) = 0  \quad \text{in }\R^d,
\end{equation*}
and passing to the limit $T\to \infty$. In particular, it holds 
\begin{align}
& \lim_{T\to \infty}\erwartung{\dashint_{\Box}|\nabla \varft - \nabla \varphi_{F}|^2 dx} = 0, \label{eq:strong} \\
& \lim_{T\to 0}\erwartung{\dashint_{\Box} V(\omega(x_d),F + \nabla \varft)dx} = V\h(F). \label{eq:sth12345}
\end{align}

\begin{proof}[Proof of Theorem~\ref{T:1:0}]
Throughout the proof we write $\lesssim$ if $\leq$ holds up to a multiplicative constant depending only on $\alpha,d$ and $p$.

\step 1 Lipschitz estimates for $\varphi_{F}^{L}$ and $\varphi_F$. 

Let $\varphi_{F}^{L}$ be given by \eqref{eq:572:3}. Lemma\nbs\ref{lem:425} and arguing as in Lemma\nbs\ref{lemma:388:new}\nbs(i), we obtain that $\varphi_F^L$ is one-dimensional in the sense of $\varphi_F^L(x)=\varphi_F^L(x_d)$ and satisfies
\begin{equation}\label{est:LipvarphiL1}
\norm{\nabla \varphi_{F}^L}_{L^{\infty}(\R^d)} \lesssim \dist(F,\SO d)\quad\mbox{for all $L\geq1$}.
\end{equation} 
In Step~1 of the proof of Lemma\nbs\ref{lemma:388:new}, we showed that $\norm{\nabla \varft}_{L^{\infty}(\R^d)} \lesssim \dist(F,\SO d)$ and that $\varphi_F^T$ is one-dimensional. Hence, by \eqref{eq:strong} and weak$^*$ lower semicontinuity of norms, we have
\begin{equation}\label{est:nablavarphilinfty}
\norm{\nabla \varphi_{F}}_{L^{\infty}(\R^d)} \leq \liminf_{T\to\infty}\norm{\nabla \varphi_{F}^T}_{L^{\infty}(\R^d)}\lesssim \dist(F,\SO d)\qquad\mbox{$\mathbb P$-a.s.}
\end{equation}
This means that we may choose $\overline{\delta}= \overline{\delta}(\alpha,p,d)>0$ such that for $F \in U_{\overline{\delta}}$, we have 
\begin{equation}\label{eq:611:3}
\max_{\varphi\in\{\varphi_F^L,\varphi_F^T\}}\|\dist(F+\nabla \varphi,\SO d)\|_{L^\infty(\R^d)}<\delta\qquad\forall L,T\geq1,
\end{equation}
where $\delta$ is given by Lemma\nbs\ref{C:wv}. Moreover, for $\varphi=\varphi_F$ the estimate in \eqref{eq:611:3} holds $\mathbb P$-a.s.\ and \eqref{est:nablavarphilinfty} improves $\erwartung{|\nabla \varphi_F|^2}<\infty$ to
\begin{equation}\label{est:varphigradientP}
\mathbb P[|\nabla \varphi_F|\leq c\dist(F,\SO d)]=1,
\end{equation} 
where $c=c(\alpha,p,d)<\infty$ and $F\in U_{\overline \delta}$.

\step 2 We show that for $F \in U_{\overline{\delta}}$, we have
\begin{equation}\label{eq:618:3}
W_{\mathrm{hom},L}(\omega,F)  = V_{\mathrm{hom},L}(\omega,F) - \mu \det F. 
\end{equation} 
This implies (i), in particular, 
\begin{equation}\label{eq:lim:whomL}
W_{\mathrm{hom}}(F)  = \lim_{L\to \infty} W_{\mathrm{hom}, L}(\omega,F) = V_{\mathrm{hom}}(F) - \mu \det F \qquad\mbox{for $\mathbb P$-a.a. }\omega\in\Omega.
\end{equation}
We first prove
\begin{equation}\label{ineq:whlvhl}
W_{\hom,L}(\omega,F)\geq V_{\hom,L}(\omega,F)-\mu\det(F)\qquad\mbox{for all $F\in \R^{d\times d}$.}
\end{equation}
For every $\varphi \in W^{1,p}_{\mathrm{per}}(\Box\lol;\R^d)$, inequality \eqref{WgeqV} implies 
\begin{equation*}
\dashint_{\Box\lol} W(\omega(x_d),F+ \nabla \varphi)dx \geq \dashint_{\Box\lol} V(\omega(x_d),F+ \nabla \varphi) - \mu \det(F+ \nabla \varphi)dx \geq V_{\mathrm{hom},L}(\omega,F) - \mu \det(F),
\end{equation*}
where we use that $\det$ is a Null-Lagrangian in the form $\dashint_{\square_L}\det(F+\nabla \varphi)\,dx=\det(F)$ for all $\varphi\in W_{\per}^{1,p}(\square_L;\R^d)$ with $p\geq d$. Taking the infimum over $\varphi$, we obtain \eqref{ineq:whlvhl}.\\

Finally, we prove  \eqref{eq:618:3}. Fix $F\in U_{\bar \delta}$ and let $\varphi_F^L\in  W^{1,p}_{\mathrm{per}}(\Box\lol;\R^d)$ be as in Step~1. Then, it holds 
\begin{eqnarray*}
W_{\mathrm{hom},L}(\omega,F) & \leq & \dashint_{\Box\lol}W(\omega(x_d),F+ \nabla \varphi_{F}^L)dx \\
 & \stackrel{\eqref{eq:611:3},\eqref{W=V}}= &\dashint_{\Box\lol}V(\omega(x_d),F+ \nabla \varphi_{F}^L) - \mu \det(F+ \nabla \varphi_{F}^L)dx
  \\
 & =& V_{\mathrm{hom},L}(\omega,F)- \mu \det{F}, 
\end{eqnarray*}
and in combination with \eqref{ineq:whlvhl} we have \eqref{eq:618:3}.

\step 3 Proof of (ii). 

For $F\in U_{\overline \delta}$, let $\varphi_{F}:\Omega\times \R^d\to\R^d$ be characterized by \eqref{eq:580:3} and \eqref{eq:580:4}. A combination of \eqref{eq:611:3}, \eqref{W=V} and Corollary\nbs\ref{cor:613}\nbs(i) implies that $\varphi_{F}(\omega,\cdot)$ is a solution to \eqref{eq:3912} for $\mathbb P$-a.a.\ $\omega\in \Omega$. Moreover, \eqref{eq:580:4} and \eqref{est:varphigradientP} imply sublinearity of $\varphi_F$ in the sense of \eqref{eq:correctorsublinear}. Hence, we have existence of a random field satisfying properties (a)--(c) of Theorem~\ref{T:1:0} part (ii). Uniqueness follows as in the proof of Lemma\nbs\ref{lemma:388:new}\nbs(i) (based on uniqueness for solutions to \eqref{eq:580:3} and \eqref{eq:580:4}).

Finally, combining $\erwartung{\nabla \varphi_F}=0$ and the one dimensionality of $\varphi_F$ in the form $\nabla \varphi_F=(\partial_{x_d} \varphi_F)\otimes e_d$ together with the identity $\det(F+u\otimes v)=\det(F)(1+v^TF^{-1}u)$, we obtain $\erwartung{\det(F+\nabla \varphi_F)}=\det(F)$ and thus
\begin{eqnarray*}
 W_{\hom}(F)&\stackrel{\eqref{eq:lim:whomL}}{=}&V_{\hom}(F)-\mu \det(F)=\erwartung{\dashint_{\Box}V(\omega(x_d),F + \nabla \varphi_{F}) - \mu \det(F+ \nabla \varphi_{F})} \\
 &\stackrel{\eqref{eq:611:3}}{=}& \erwartung{\dashint_{\Box}W(\omega(x_d),F+ \nabla \varphi_{F})dx}.
\end{eqnarray*}
This completes the proof of (ii).

\step 4 $W_{\mathrm{hom},L}(\omega,\cdot)\in C^3(U_{\overline{\delta}})$.

The mapping $F\mapsto \varphi_{F}^{L}$ belongs to $C^2(U_{\overline \delta};W_{\rm per}^{1,\infty}(\Box,\R^d))$ and the derivatives $\varphi_{F,G}^{L}:= D\varphi_{F}^{L}G$ and $\varphi_{F,G,H}^{L}:= D^2\varphi_{F}^{L}HG$ are characterized as the unique (up to a constant) solutions to 
\begin{equation}\label{eq:dphi:per}
-\dive \brac{D^2W(\omega(x_d), F+ \nabla \varphi_{F}^{L})(G+ \nabla \varphi_{F,G}^{L})} = 0 \quad \text{in }\Box_{L}
\end{equation}
and
\begin{eqnarray}\label{eq:ddphi:per}
& & -\dive \brac{D^2W(\omega(x_d), F+ \nabla \varphi_{F}^{L})\nabla \varphi_{F,G,H}^{L}}\notag \\ & = & \dive \brac{D^3W(\omega(x_d),F+ \nabla\varphi_{F}^{L})\brac{H+ \nabla \varphi_{F,H}^L}\brac{G+ \nabla \varphi_{F,G}^L}}  \quad \text{in }\Box_{L}.
\end{eqnarray}
Moreover, they satisfy the Lipschitz bounds
\begin{equation}\label{est:lipschitz:linearperidiccor}
\norm{\nabla \varphi_{F,G}^L}_{L^{\infty}(\R^d)}\lesssim |G|, \qquad \norm{\nabla \varphi_{F,G,H}^L}_{L^{\infty}(\R^d)}\lesssim |G||H|.
\end{equation}
All of these statements follow analogously to the proof of Lemma\nbs\ref{lemma:388:new} (with help of Lemma\nbs\ref{lem:425}) and we omit their proofs. As a result of these facts, the claim follows with
\begin{align*}
DW_{\mathrm{hom},L}(F)G  = & \dashint_{\Box\lol} DW(\omega(x_d),F + \nabla \varphi_{F}^{L})\cdot G dx, \\
D^2W_{\mathrm{hom},L}(F)HG  = & \dashint_{\Box\lol} D^2W(\omega(x_d), F+ \nabla\varphi_{F}^{L})\brac{H+\nabla \varphi_{F,H}^L} \cdot G \\
D^3W_{\mathrm{hom},L}(F)IHG  = &  \dashint_{\Box\lol} D^3W(\omega(x_d),F+ \nabla \varphi_{F}^L)\brac{I+ \nabla \varphi_{F,I}^L}\brac{H+\nabla \varphi_{F,H}^L} \cdot (G+\nabla \varphi_{F,G}^L) dx,
\end{align*}
where we used in the last identity \eqref{eq:dphi:per} and \eqref{eq:ddphi:per} in the form
\begin{align*}
&\dashint_{\Box\lol} D^2W(\omega(x_d),F+ \nabla \varphi_{F}^L)\nabla \varphi_{F,H,I}^L\cdot G dx\\
=&\dashint_{\Box\lol} D^3W(\omega(x_d),F+ \nabla \varphi_{F}^L)\brac{I+ \nabla \varphi_{F,I}^L}\brac{H+\nabla \varphi_{F,H}^L} \cdot \nabla \varphi_{F,G}^L dx.
\end{align*}

\step 5 We show that $W\h\in C^3(U_{\overline{\delta}})$.

Consider $F\in U_{\bar \delta}$, where $\bar\delta$ is chosen as in Step~1. We note that in view of \eqref{eq:611:3} and \eqref{W=V} the unique function $\varphi_F:\Omega\times\R^d\to\R^d$ satisfying \eqref{eq:580:3} and \eqref{eq:580:4} also satisfies \eqref{eq:3912}.

\substep{5.1} Differentiability of $U_{\bar \delta}\ni F\mapsto \varphi_F$. We claim that the mapping $U_{\overline{\delta}}\in F \mapsto \varphi_{F} \in L^{2}(\Omega,W^{1,\infty}(\Box; \R^d))$ is differentiable with derivative $D\varphi_{F}G :=\varphi_{F,G}$ for every $G\in\R^{d\times d}$, where $\varphi_{F,G}:\Omega\times \R^d\to\R^d$ is uniquely characterized via \eqref{eq:linearizedcorrector:1} and \eqref{eq:linearizedcorrector:2} (with $H$ replaced by $G$).

\smallskip

For $G\in\R^{d\times d}$ and $\omega\in \Omega$, we consider $\varfgt(\omega)\in H^1_{\rm uloc}(\R^{d}; \R^d)$ as in Lemma~\ref{lemma:388:new}~(ii). Appealing to Corollary\nbs\ref{cor:613} and Proposition\nbs\ref{prop:appendix1}, we have
$
\lim_{T\to \infty} \erwartung{\dashint_{\Box}|\varfgt - \varphi_{F,G}|^2dx} = 0.
$
Since $\varfgt$ satisfies $\varfgt(\omega,x)=\varfgt(\omega,x_d)$ and $\|\nabla \varfgt\||_{L^\infty(\R^d)}\lesssim |G|$ (see Lemma\nbs\ref{lemma:388:new}\nbs(ii)), we obtain $\varphi_{F,G}(\omega,x)=\varphi_{F,G}(\omega,x_d)$ and $\|\nabla \varphi_{F,G}\|_{L^\infty(\R^d)}\lesssim |G|$ for $\mathbb P$-a.a.\ $\omega\in \Omega$. Furthermore, the proof of Lemma\nbs\ref{lemma:388:new}\nbs(ii) implies for $|G|$ small enough 
\begin{align}\label{eq:dstationarycorrector}
\erwartung{\dashint_{\Box}|\nabla( \varphi_{F+G} -\varphi_{F} -\varphi_{F,G})|^2}=&\lim_{T\to\infty}\erwartung{\dashint_{\Box}|\nabla (\varphi_{F+G}^T - \varphi_{F}^T - \varphi_{F,G}^T)|^2}\notag\\
\leq&\limsup_{T\to\infty} \erwartung{\|\nabla (\varphi_{F+G}^T - \varphi_{F}^T - \varphi_{F,G}^T)\|_{L^\infty(\Box)}^2} \stackrel{ \eqref{eq:705}}{\lesssim} |G|^4.
\end{align}
This means that the mapping $U_{\overline{\delta}}\in F \mapsto \varphi_{F} \in L^{2}(\Omega,H^1(\Box; \R^d))$ is differentiable with derivative $D\varphi_{F}G :=\varphi_{F,G}$ (recall $\dashint_\Box \varphi_F\,dx=0$). From \eqref{eq:dstationarycorrector} and the weak$^*$ lower semicontinuity, we deduce
$
\erwartung{\|\nabla (\varphi_{F+G} - \varphi_{F} - \varphi_{F,G})\|_{L^\infty(\Box)}^2}^\frac12\lesssim |G|^2,
$
and thus the claimed differentiability in $ L^{2}(\Omega,W^{1,\infty}(\Box; \R^d))$.

\substep{5.2} Differentiability of  $U_{\overline \delta}\ni F\mapsto \varphi_{F,G}$. We claim that for every $G\in\R^{d\times d}$ the mapping $U_{\overline{\delta}}\in F \mapsto \varphi_{F,G} \in L^{2}(\Omega,W^{1,\infty}(\Box; \R^d))$ is differentiable.

\smallskip

For $G,H\in\R^{d\times d}$ and $\omega\in \Omega$, we consider $\varfght(\omega)\in H^1_{\rm uloc}(\R^{d}; \R^d)$ as in Lemma~\ref{lemma:388:new}~(iii). Appealing to Corollary\nbs\ref{cor:613} and Proposition\nbs\ref{prop:appendix1}, we have
$
\lim_{T\to \infty} \erwartung{\dashint_{\Box}|\varfght - \varphi_{F,G,H}|^2dx} = 0,
$
where $\varphi_{F,G,H}:\Omega\times \R^d\to\R^d$ is uniquely characterized by
\begin{eqnarray}\label{eq:linearizedcorrector2:1}
& & - \dive\brac{D^2W(\omega(x_d),F+ \nabla \varphi_{F})\nabla \varphi_{F,G,H}} \nonumber \\ & = & \dive \brac{D^3W(\omega(x_d),F + \nabla \varphi_{F})\brac{H+ \nabla \varphi_{F,H} }\brac{G+ \nabla \varphi_{F,G}}} \quad \text{in }\R^d\quad\mbox{for $\mathbb P$-a.a.\ $\omega\in \Omega$}
\end{eqnarray}
and
\begin{align}\label{eq:linearizedcorrector2:2}
& \dashint_{\Box} \varphi_{F,G,H} dx = 0, \quad \nabla \varphi_{F,G,H} \text{ is stationary,} \quad \erwartung{\nabla \varphi_{F,G,H}} = 0, \quad \erwartung{|\nabla \varphi_{F,G,H}|^2}<\infty.
\end{align}
As in the previous substep, we deduce from Lemma\nbs\ref{lemma:388:new}\nbs(iii) that $\varphi_{F,G,H}$ is one-dimensional and satisfies $\norm{\nabla \varphi_{F,G,H}}_{L^{\infty}(\R^d)}\lesssim  |G||H|$ for $\mathbb P$-a.a.\ $\omega\in \Omega$. Furthermore, we obtain for $|H|>0$ small enough
\begin{align*}
\erwartung{\dashint_{\Box}|\nabla (\varphi_{F+H,G} -\varphi_{F,G} -\varphi_{F,G,H})|^2}\leq&\limsup_{T\to\infty}\erwartung{\|\nabla (\varphi_{F+H,G}^T -\varphi_{F,G}^T -\varphi_{F,G,H}^T)\|_{L^\infty(\Box)}^2}\\
  \stackrel{\eqref{eq:123456}}{\lesssim}&  |G|^2 |H|^4.
\end{align*}
From the above estimate, we conclude, as in the previous substep, that $U_{\overline{\delta}}\in F \mapsto \varphi_{F,G} \in L^{2}(\Omega,W^{1,\infty}(\Box; \R^d))$ is differentiable with derivative $D\varphi_{F,G}H :=\varphi_{F,G,H}$. 

\substep{5.3} Differentiability of $W_{\hom}$.

Since $W_{\hom,L}(\omega;\cdot)$ is bounded in $C^3(U_{\overline \delta})$ uniformly in $L$ and $\omega$ (see Step~4), we directly conclude from \eqref{eq:lim:whomL} that $W_{\hom}\in C^{2,\gamma}(U_{\overline \delta})$ for every $\gamma\in(0,1)$.

The $L^2(\Omega;W^{1,\infty}(\Box; \R^d))$ differentiability of $U_{\overline \delta}\ni F\mapsto \varphi_F$ combined with \eqref{eq:611:3} implies
\begin{align*}
DW_{\mathrm{hom}}(F)G  = &\erwartung{\dashint_{\Box} DW(\omega(x_d),F + \nabla \varphi_{F})(G+\nabla \varphi_{F,G}) dx}
\end{align*}
and with help of 
\begin{equation}\label{eq:DWhomvarfg}
\erwartung{\dashint_{\Box} DW(\omega(x_d),F + \nabla \varphi_{F})\nabla \varphi_{F,G} dx}=0
\end{equation}
we obtain the claimed formula \eqref{eq:def:dwhom} for $DW_{\hom}$. The idendity \eqref{eq:DWhomvarfg} is standard and relies on the corrector equation \eqref{eq:3912} together with the stationarity of $\nabla \varphi_F$ and $\nabla \varphi_{F,G}$:  Recall that for any stationary field $\F$ and $\eta\in C_c^\infty(\R)$ it holds
\begin{equation}\label{eq:stationaritycutoff}
\erwartung{\int_\R \eta F}=\erwartung{\F}\int_{\R}\eta dx_d\quad\mbox{and}\quad \erwartung{\int_{\R}\partial_{x_d}\eta \F}=\erwartung{\F}\int_{\R}\partial_{x_d}\eta dx_d = 0.
\end{equation}
Hence, by stationarity and one-dimensionality of $\nabla \varphi_{F}$ and $\nabla \varphi_{F,G}$, we have for any $\eta\in C_c^\infty(\R)$
\begin{align}\label{eq:usecorrectorequation}
&\erwartung{\dashint_{\Box} DW(\omega(x_d),F+\nabla \varphi_F)\nabla \varphi_{F,G}\,dx}\notag\\
\stackrel{\eqref{eq:stationaritycutoff}}{=}&\erwartung{\int_{\R} \eta DW(\omega(x_d),F+\nabla \varphi_F)\nabla\varphi_{F,G}\,dx_d}\notag\\
\stackrel{\eqref{eq:stationaritycutoff}}{=}&\erwartung{\int_{\R}  DW(\omega(x_d),F+\nabla \varphi_F)\nabla(\eta\varphi_{F,G})\,dx_d}\stackrel{\eqref{eq:3912}}{=}0.
\end{align}
Next, we come to the second and third derivative of $W_{\hom}$. The $L^2(\Omega;W^{1,\infty}(\Box; \R^d))$ differentiability of $U_{\overline \delta}\ni F\mapsto \varphi_F$, $U_{\overline \delta}\ni F\mapsto \varphi_{F,G}$ together with \eqref{eq:def:dwhom} and \eqref{eq:611:3} imply $W_{\hom}\in C^3(U_{\overline \delta})$ and
\begin{align*}
D^2W_{\mathrm{hom}}(F)HG  = & \erwartung{\dashint_{\Box} D^2W(\omega(x_d), F+ \nabla\varphi_{F})\brac{H+\nabla \varphi_{F,H}} \cdot G dx}\\
D^3W_{\hom}(F)GHI  = &\erwartung{  \dashint_{\Box} D^3W(\omega(x_d),F+ \nabla \varphi_{F})\brac{I+ \nabla \varphi_{F,I}}\brac{H+ \nabla \varphi_{F,H}} \cdot G dx} \\ & +\erwartung{\dashint_{\Box} D^2W(\omega(x_d),F+ \nabla \varphi_{F})\nabla \varphi_{F,H,I}\cdot G dx}.
\end{align*}
Hence, we obtain \eqref{eq:def:ddwhom} and with help of
\begin{align*}
&\erwartung{\dashint_{\Box} D^2W(\omega(x_d),F+ \nabla \varphi_{F})\nabla \varphi_{F,H,I}\cdot G dx}\notag\\
=&\erwartung{\dashint_{\Box} D^3W(\omega(x_d),F+ \nabla \varphi_{F})(I+\nabla \varphi_{F,I})(H+\nabla \varphi_{F,H})\nabla \varphi_{F,G} dx}
\end{align*}
(wich follows from \eqref{eq:linearizedcorrector:1}, \eqref{eq:linearizedcorrector:2}, \eqref{eq:linearizedcorrector2:1}, \eqref{eq:linearizedcorrector2:2} and a similar calculation as in \eqref{eq:usecorrectorequation}) we arrive at \eqref{eq:def:d3whom}. 

\step 6 Strong rank-one convexity.

This follows from the strong convexity of $V$ and thus $V_{\hom,L}$ and $V_{\hom}$ combined with the fact that $F\mapsto\det(F)$ is rank-one affine (see \cite[Proof of Theorem~2, Step~6]{neukamm2018quantitative} for details).

\end{proof}

\begin{proof}[Proof of Corollary~\ref{cor:T:1}]

Assuming $p$-growth of the form \eqref{eq:standard:pgrowth} the result follows by standard $\Gamma$-convergence arguments, see e.g.\ \cite[Lemma 3.2]{NSS17} in a general ergodic (and discrete) setting. Below, we provide a direct argument which makes use of existence and regularity of the corrector $\varphi_F$ for $F\in U_{\overline \delta}$.

We only show $W\h(F)\geq \overline{W}\h(F)$ since the reverse inequality is trivial. For given $L\in\N$ and $\varepsilon\in(0,1)$, we consider a cut-off function $\eta\in C^{\infty}_0(\Box_{L})$ satisfying $\eta = 1$ in $\Box_{(1-\varepsilon)L}$ and $|\nabla \eta| \lesssim \frac{1}{\varepsilon L}$, where here and for the rest of the proof $\lesssim$ means $\leq$ up to a multiplicative constant that depends on $\alpha,d$ and $p$. For $\varphi_F$ as in Theorem~\ref{T:1:0}, we have
\begin{eqnarray*}
& & \inf_{\varphi \in W^{1,p}_0(\Box_{L}; \R^d)}\dashint_{\Box_{L}} W(\omega(x_d),F+ \nabla \varphi(x))dx \nonumber \\ & \leq & \dashint_{\Box_{L}} W(\omega(x_d),F+ \nabla (\eta \varphi_F))dx \\
& \leq &  \dashint_{\Box\lol}W(\omega(x_d),F+ \nabla \varphi_F)dx + \frac{1}{|\Box_{L}|}\int_{\Box_{L}\setminus\Box_{(1-\varepsilon)L}}W(\omega(x_d),F+ \nabla (\eta \varphi_F))dx, \nonumber 
\end{eqnarray*}
where we used $W\geq0$ in the last inequality.  By the subadditive ergodic theorem (cf. \cite{AkcogluKrengel81}) the left-hand side converges to $\overline{W}\h(F)$ as $L\to\infty$ and thus by \eqref{eq:whomviastationarycorrector} we obtain 
\begin{align}\label{eq:111111}
\overline W\h(F)\leq W\h(F)+\limsup_{L\to\infty}\frac{1}{|\Box_{L}|}\int_{\Box_{L}\setminus\Box_{(1-\varepsilon)L}}W(\omega(x_d),F+ \nabla (\eta \varphi_F))dx\qquad\mbox{for $\mathbb P$-a.a.\ $\omega\in \Omega$}.
\end{align}
In order to bound the second term on the right-hand side in \eqref{eq:111111}, we observe
\begin{eqnarray*}
 \|\dist(F+ \nabla (\eta \varphi_F),\SO d)\|_{L^\infty(\Box_L)}&\leq& \dist(F;\SO d)+\|\nabla \varphi_F\|_{L^\infty(\R^d)}+\|\varphi_F\otimes \nabla \eta\|_{L^\infty(\Box_L)}\\
 &\stackrel{\eqref{est:stationarygradient}}{\lesssim}&\dist(F,\SO d)+\frac1{\varepsilon L}\|\varphi_F\|_{L^\infty(\Box_L)},
\end{eqnarray*}
where we use $|\nabla \eta|\leq (\varepsilon L)^{-1}$ in the second estimate. In view of the sublinearity of the corrector \eqref{eq:correctorsublinear}, we have for $\mathbb P$-a.a. $\omega\in\Omega$ 
$$
\|\dist(F+ \nabla (\eta \varphi_F),\SO d)\|_{L^\infty(\Box_L)}\lesssim \dist(F,\SO d)
$$
provided $L$ is sufficiently large (depending on $\alpha,d,\varepsilon$, $p$ and $\omega$). Hence, for ${\overline{\overline \delta}}={\overline{\overline \delta}}(\alpha,d,p)>0$ sufficiently small, we have $\|\dist(F+ \nabla (\eta \varphi_F),\SO d)\|_{L^\infty(\Box_L)}\leq \alpha$ for all $F\in U_{\overline{\overline \delta}}$ and thus 
$$
\limsup_{L\to\infty}\frac{1}{|\Box_{L}|}\int_{\Box_{L}\setminus\Box_{(1-\varepsilon)L}}W(\omega(x_d),F+ \nabla (\eta \varphi_F))dx\stackrel{\eqref{ass:wd2lip}}\lesssim \limsup_{L\to \infty} \frac{|\Box_{L}\setminus\Box_{(1-\varepsilon)L}|}{|\Box_{L}|}\lesssim \varepsilon
$$
and the desired inequality $\overline{W}\h(F)\leq W\h(F)$ follows from the arbitrariness of $\varepsilon>0$.
\end{proof}

\section{Stochastic intermediate results}\label{sec:stochastic}

In this section, we provide optimal estimates on the growth of the correctors $\varft,\varfgt$ and $\varfght$ that are constructed in Lemma~\ref{lemma:388:new}. For this, we recall in Section~\ref{sec:spectral:gap} some results from the literature. In Section~\ref{sec:correctors:statements}, we state the estimates and the corresponding proofs are in Section~\ref{proofs:interm:stoch}.

\subsection{Multiscale decomposition and spectral gap}\label{sec:spectral:gap}

The spectral gap inequalities of Definition~\ref{sgs}, imply $L^q$-versions of the spectral gap inequality:
{
\begin{lemma}[ {\cite[Proposition 1.10]{DuerinckxGloria2020b}}]\label{L:SQq}
Suppose $\mathbb P$ satisfies \ref{stationarity} and a spectral gap estimate of Definition \ref{sgs} (i). Then there exists $C=C(\rho)<\infty$ such that for any $q\geq1$
\begin{equation}\label{SQ:qth1}
\erwartung{| \F(\omega)-\erwartung{\F(\omega)}|^{2q}}^\frac1{2q}\leq Cq\erwartung{\int_{\R}\bigg|\brac{\int_{B_1(s)} \left| \frac{\partial\F}{\partial \omega} \right|}^2 ds\bigg|^q}^\frac1{2q}
\end{equation}
Moreover, the analogous statement holds with $\mathbb E$ replaced by $\mathbb E_L$.
\end{lemma}
}

\begin{remark}\label{rem:exponential}
 We will frequently deduce from Lemma~\ref{L:SQq} estimates of the form $\erwartung{\F(\omega)^{k}}^\frac1{k}\leq ck$ for all $k\in\mathbb N$ and a random variable $\F\geq0$. By the elementary inequality $\frac{(2ce)^{-k}\erwartung{\F(\omega)^{k}}}{k!}\leq \frac1{2^k}\frac{k^k}{e^k k!}\leq \frac1{2^k}$ this implies exponential moments of $\F$ in the form $\erwartung{\exp\brac{\frac1C \mathcal F}}\leq 2$ with $C=2 ce$. 
\end{remark}

The following lemma is a special case of \cite[Lemma 4.8]{BFFO17} and can be interpreted as an improved Poincar\'e inequality in wich spatial oscillations of the gradient are taken into account (see also \cite[Proposition 6.1]{ArmstrongKuusiMourrat16} for a related deterministic estimate)
\begin{lemma}[{\cite[Lemma 4.8]{BFFO17}}]\label{lemma:405} Let $m\geq2$ and $K\geq 0$. Let $u: \Omega\times \R \to \R$ be a measurable function such that 
\begin{equation*}
\erwartung{\biggl(\dashint_{B_1(x_0)}|\nabla u|^2 dx\biggr)^{\frac{m}2}}^{\frac{1}{m}}\leq K\quad\mbox{and}\quad
\erwartung{\brac{\int_{\R}\nabla u \cdot g}^{m}}^{\frac{1}{m}}\leq K r^{-\frac12}
\end{equation*}
for all $r\geq 2$, all $x_0 \in \R$ and all measurable $g: \R \to \R$ supported in $B_r(x_0)$ satisfying
\begin{equation*}
\brac{\dashint_{B_{r}(x_0)}|g|^{3}}^{\frac{1}{3}}\leq r^{-1}.
\end{equation*}
Then it holds
\begin{equation*}
\erwartung{\biggl(\dashint_{B_r(x_0)}\left| u- (u)_{B_r(x_0)} \right|^2 dx\biggr)^{\frac{m}2}}^{\frac{1}{m}}\leq cKr^\frac12.
\end{equation*} 
\end{lemma}

\subsection{Decay estimates for localized correctors}\label{sec:correctors:statements}

\begin{lemma}\label{lemma:388} Let the assumptions of Theorem\nbs\ref{Theorem:2} hold and let $\overline{\delta}=\overline{\delta}(\alpha,p,d)>0$ be as in Lemma~\ref{lemma:388:new}. Then there exist $c=c(\alpha,p,d, \rho)>0$ and a random variable $\mathcal C: \Omega \to \R$ satisfying $\erwartung{\exp(\frac1c \mathcal C)}\leq 2$ such that for any $F \in U_{\overline{\delta}}$, we have the following estimates:

\begin{enumerate}[label=(\roman*)]
\item For all $r\geq 2$, $x_0\in \R^d$, $T\geq 2$ and $R\geq \sqrt{T}$, it holds 
\begin{align}
\biggl(\dashint_{B_r(x_0)}|\varft -(\varft)_{B_{r}(x_0)}|^2 dx\biggr)^\frac12 & \leq  \mathcal C r^{\frac12}\dist(F,\SO d) \label{eq:387}, \\
\frac1{\sqrt{R}}\biggl(\dashint_{B_R(x_0)} |\varft|^2 dx\biggr)^\frac12+\sqrt{\frac{R}T}\left|\dashint_{B_R(x_0)}\varft \; dx \right| & \leq \mathcal C  \dist(F,\SO d).\label{eq:391}
\end{align}

\item  For all $x_0\in \R^d$, $T\geq 2$ and $R\geq \sqrt{T}$, it holds
\begin{equation}\label{eq:470}
\biggl(\dashint_{B_{R}(x_0)} \frac{1}{T}|\varphi_{F}^{2T}-\varft|^2 + |\nabla \varphi_{F}^{2T}-\nabla \varft|^2 dx\biggr)^\frac12 \leq \mathcal C\biggl(\frac{R}{T}\biggr)^\frac12 \dist(F,\SO d).
\end{equation}
\end{enumerate}

\end{lemma}
The proof is presented in Section\nbs\ref{proofs:interm:stoch}.
\begin{lemma}\label{lemma:525} 
Consider the situation of Lemma~\ref{lemma:388}. Then there exist $c=c(\alpha,p,d, \rho)>0$ and a random variable $\mathcal C: \Omega \to \R$ satisfying $\erwartung{\exp(\frac1c \mathcal C)}\leq 2$ such that for any $F \in U_{\overline{\delta}}$, we have the following estimates:

\begin{enumerate}[label=(\roman*)]
\item For all $r\geq 2$, $x_0\in \R^d$, $T\geq 2$, $R\geq \sqrt{T}$ and $G\in \R^{d \times d}$, we have
\begin{align}
\biggl(\dashint_{B_r(x_0)}|\varfgt -(\varfgt)_{B_{r}(x_0)}|^2 dx\biggr)^\frac12 \leq&  \mathcal C r^\frac12|G| \label{eq:525} \\
\frac1{\sqrt{R}}\biggl(\dashint_{B_R(x_0)} |\varfgt|^2 dx\biggr)^\frac12+\sqrt{\frac{R}{T}}\left|\dashint_{B_R(x_0)}\varfgt dx \right| \leq& \mathcal C |G|.\label{eq:531}
\end{align}

\item For all $x_0 \in \R^d$, $T\geq 2$, $R\geq \sqrt{T}$ and $G\in \R^{d\times d}$, it holds
\begin{equation}\label{eq:556}
\biggl(\dashint_{B_{R}(x_0)} \frac{1}{T}|\varphi_{F,G}^{2T}-\varfgt|^2 + |\nabla \varphi_{F,G}^{2T}-\nabla \varfgt|^2 dx\biggr)^\frac12 \leq \mathcal C \biggl(\frac{R}{T}\biggr)^\frac12|G| .
\end{equation}
\end{enumerate}

\end{lemma}
The proof is presented in Section\nbs\ref{proofs:interm:stoch}.
\begin{lemma}\label{lemma:13:58}
Consider the situation of Lemma~\ref{lemma:388}. Then there exist $c=c(\alpha,p,d, \rho)>0$ and a random variable $\mathcal C: \Omega \to \R$ satisfying $\erwartung{\exp(\frac1c \mathcal C)}\leq 2$ such that for any $F \in U_{\overline{\delta}}$, we have the following estimates:

\begin{enumerate}[label=(\roman*)]
\item For all $r\geq 2$, $x_0\in \R^d$, $T\geq 2$, $R\geq \sqrt{T}$ and $G,H\in \R^{d \times d}$, we have
\begin{align}
\biggl(\dashint_{B_r(x_0)}|\varfght -(\varfght)_{B_{r}(x_0)}|^2 dx\biggr)^\frac12 & \leq  \mathcal C |G||H| r^\frac12 \label{eq:620} \\
\frac1{\sqrt{R}}\brac{\dashint_{B_R(x_0)} |\varfght|^2 dx}^{\frac{1}{2}}+\sqrt{\frac{R}T}\left|\dashint_{B_R(x_0)}\varfght dx \right| & \leq \mathcal C |H| |G| .\label{eq:621}
\end{align}
\item For all $x_0\in \R^d$, $T\geq 2$, $R\geq \sqrt{T}$ and $G,H\in \R^{d\times d}$, it holds
\begin{equation}\label{eq:653}
\brac{\dashint_{B_{R}(x_0)} \frac{1}{T}|\varphi_{F,G,H}^{2T}-\varfght|^2 + |\nabla \varphi_{F,G,H}^{2T}-\nabla \varfght|^2}^{\frac{1}{2}} \leq \mathcal{C}|G| |H| \brac{\frac{R}{T}}^{\frac{1}{2}}.
\end{equation}
\end{enumerate}
\end{lemma}
The proof is presented in Section\nbs\ref{proofs:interm:stoch}.
\subsection{Proofs of stochastic intermediate results}\label{proofs:interm:stoch}
\begin{proof}[Proof of Lemma~\ref{lemma:388}]
Throughout the proof we write $\lesssim$ if $\leq$ holds up to a multiplicative constant depending only on $\alpha,d,p$ and $\rho$. Suppose $F\in U_{\bar \delta}$ and recall that, according to the proof of Lemma\nbs\ref{lemma:388:new}, $\varft$ is the unique solution to 
\begin{equation}\label{eq:1601:2}
\frac{1}{T}\varphi^{T}_{F} - \dive DV(\omega(x_d),F + \nabla \varphi^{T}_{F}) = 0 \quad \text{in }\R^d,
\end{equation}
where $V$ corresponds to the matching convex lower bound for $W$ from Lemma\nbs\ref{C:wv}. 

\step 1 Estimate \eqref{eq:387}. We consider the random variable
$
\F(\omega)= \int_{\R} \partial_{x_d} \varft(\omega) \cdot g \; dx_d,
$
where $g: \R\to \R^{d}$ is supported in $B_r(x_0)$ and it satisfies
\begin{equation}\label{ineq:gmssq}
\brac{\dashint_{B_r(x_0)}|g|^3 dx_d}^{\frac{1}{3}} \leq r^{-1}.
\end{equation}
We compute the derivative of equation \eqref{eq:1601:2} with respect to a bounded perturbation $\delta \omega$ supported in $B_1(s)$: The change $\delta \varft$ of the corrector under such perturbation satisfies the linear equation
\begin{equation*}
\frac{1}{T}\delta \varft - \dive \brac{D^2V(\omega(x_d),F+ \nabla \varft)\nabla \delta \varft}  = \dive \brac{\partial_{\omega}DV(\omega(x_d),F+\nabla \varft)\delta\omega} \quad \text{in }\R^d.
\end{equation*} 
According to Lemma \ref{lemma:279}, $\delta\varft$ depends only on $x_d$ and since
\begin{equation}\label{def:vdphif}
x_d \mapsto v_d(x_d):= \partial_{\omega}D_dV(\omega(x_d),F+ \nabla \varft(x_d))\delta \omega(x_d)
\end{equation}
is bounded and compactly supported, $\brac{x_d \mapsto \delta \varft(x_d)} \in H^1(\R;\R^d)$ is the unique weak solution to 
\begin{equation}\label{eq:491}
\frac{1}{T}\delta\varft - \partial_{x_d}\brac{ a(x_d)\partial_{x_d}\delta\varft} = \partial_{x_d}v_d \quad \text{in }\R,
\end{equation}
where $a(x_d): \R^{d}\to \R^{d}$ is given by
\begin{equation}\label{def:af}
 a(x_d)f :=\brac{D^2V(\omega(x_d),F+ \nabla \varft)(f\otimes e_d)}e_d.
\end{equation}
We denote by $h \in H^1(\R;\R^d)$ the unique weak solution to
\begin{equation}\label{eq:498}
\frac{1}{T}h - \partial_{x_d}\brac{a(x_d) \partial_{x_d}h}=\partial_{x_d} g.
\end{equation}
We compute $\delta \F := \lim_{t \to 0}\frac{\F(\omega + t\delta \omega)-\F(\omega)}{t}$ as follows
\begin{align*}
\delta \mathcal{F}  = \int_{\R}\partial_{x_d} \delta \varft \cdot g \; dx_d \stackrel{\eqref{eq:498}}{=} \int_{\R}\frac{1}{T}h\cdot \delta\varft + a\partial_{x_d}h \cdot \partial_{x_d}\varft \; dx_d  \stackrel{\eqref{eq:491}}{=}  \int_{\R} v_d \cdot \partial_{x_d}h \; dx_d. 
\end{align*}
Note that $v_d$ is supported in $B_1(s)$ (for $s\in \R$). Also, we have the following estimate
\begin{equation*}
|\partial_{\omega}D_d V(\omega,F+ \nabla \varft)|= |\partial_{\omega}D_d V(\omega,F+ \nabla \varft)-\partial_{\omega}D_dV(\omega,R)| \lesssim |F-R + \nabla \varft|
\end{equation*}
for an arbitrary $R \in \SO d$, where we use $DV(\omega,R)=D\det(R)$ and thus $\partial_\omega DV(\omega,R)=0$. In particular, with help of \eqref{eq:393:new}, this implies $|v_d|\lesssim \dist(F,\SO d)|\delta \omega|$. Thus we have
\begin{equation*}
\int_{B_1(s)}\left| \frac{\partial \F}{\partial \omega} \right| \lesssim \dist(F,\SO d) \int_{B_1(s)}|\partial_{x_d} h|dx_d. 
\end{equation*}
The $q$th version of the spectral gap inequality (see \eqref{SQ:qth1}) implies
\begin{eqnarray}\label{est:525:almostfinal}
\erwartung{|\F- \erwartung{\F}|^{2q}}^\frac1{2q} & \lesssim& q\dist(F,\SO d)\erwartung{\biggl|\int_{\R}\brac{\int_{B_1(s)}|\partial_{x_d} h|dx_d}^2 ds\biggl|^q}^\frac1{2q}\notag \\
&\lesssim& q\dist(F,\SO d) \erwartung{\biggl|\int_{\R}|\partial_{x_d} h|^2 dx_d\biggr|^q}^\frac1{2q}\notag\\
&\stackrel{\eqref{est:monotonesystemfull},\eqref{eq:498}}\lesssim&  q \dist(F,\SO d) \biggl(\int_{B_{r}(x_0)}|g|^2\,dx_d \biggr)^\frac12\notag \\ 
&\stackrel{\eqref{ineq:gmssq}}\lesssim&  q \dist(F,\SO d) r^{-\frac12}. 
\end{eqnarray} 
Since $\varft$ is a stationary field, we have that $\erwartung{\varft}$ is constant and thus $\erwartung{\F}=\int_\R g\cdot \partial_{x_d}\erwartung{\varft}\,dx_d=0$ . Hence, \eqref{eq:393:new}, \eqref{est:525:almostfinal} and Lemma \ref{lemma:405} imply
\begin{equation*}
\erwartung{\biggl(\dashint_{B_r(x_0)}|\varft-(\varft)_{B_r(x_0)}|^2\,dx\biggr)^{\frac{q}2}}^\frac1q\lesssim q\dist(F,\SO d)r^\frac12.
\end{equation*}
for every $q\geq2$ and we obtain \eqref{eq:387} (see Remark~\ref{rem:exponential}).

\step 2 Estimates \eqref{eq:391}. In view of \eqref{eq:387} and triangle inequality it suffices to show
\begin{equation}\label{eq:388}
\left|\dashint_{B_R(x_0)}\varft \; dx \right|  \leq \mathcal C \sqrt{\frac{T}R}\dist(F,\SO d).
\end{equation}
 We define random variables, for $i\in \cb{1,...,d}$, 
$
\F_i(\omega):= \dashint_{B_R}\varphi_{F}^T(\omega)\cdot e_i dx_d=\int_{\R}\varphi_{F}^T(\omega)\cdot f_i dx_d$, where $ f_i:=\frac{1}{|B_R|}\mathbf{1}_{B_R} e_i.
$ Analogously as in Step~1, we obtain
$
\delta \F_i = \int_{\R}\delta \varft \cdot f_i dx_d= \int_{\R} v_d \cdot \partial_{x_d}h_i dx_d,
$
where $v_d$ is given in \eqref{def:vdphif} and $h_i$ solves
\begin{equation}\label{eq:hi:391}
\frac{1}{T}h_i - \partial_{x_d}(a(x_d)\partial_{x_d} h_i)= f_i
\end{equation}
(see \eqref{def:af} for the definition of $a$). Using $\erwartung{\F_i}=0$ and \eqref{SQ:qth1},  we obtain  
\begin{eqnarray}\label{eq:644}
\erwartung{|\F_i|^{2q}}^\frac1{2q}  &
\lesssim & q\dist(F,\SO d) \erwartung{\biggl|\int_{\R}|\partial_{x_d} h_i|^2 ds\biggr|^q}^\frac1{2q}\notag\\
& \stackrel{\eqref{est:monotonesystemfull},\eqref{eq:hi:391}}{\lesssim}& q\dist(F,\SO d) \sqrt{T}\biggl(\int_{\R}|f_i|^2 ds\biggr)^\frac12 \notag\\ 
& \lesssim & q\dist(F,\SO d)\sqrt{\frac{T}{R}}. 
\end{eqnarray} 
From \eqref{eq:644} we deduce \eqref{eq:388}, which together with \eqref{eq:387} yields \eqref{eq:391}.

\step 3 Estimate \eqref{eq:470}. Using \eqref{eq:1601:2}, we obtain that $\Psi:= \varphi_{F}^{2T}-\varphi_{F}^T$ solves
\begin{equation*}
\frac{1}{2T}\Psi - \dive B(x_d,\nabla \Psi) = \frac{1}{2T}\varft \quad \text{in }\R^d,
\end{equation*}
where $B(x_d,G)= DV(\omega(x_d),F+ \nabla \varft + G )-DV(\omega(x_d),F+ \nabla \varft)$. Lemma~\ref{lemma:222} (ii) in combination with the fact that $\varft$ depends only on $x_d$ yields
\begin{equation*}
\dashint_{B_{R}}\frac{1}{T}|\Psi|^2 + |\nabla \Psi|^2 dx_d\lesssim  \frac{1}{RT} \int_{\R}\eta(x_d) |\varft|^2 dx_d,
\end{equation*}
where $\eta(z):=\exp(-\gamma \frac{|z|}R)$ for some $\gamma=\gamma(\alpha, p, d)\in(0,1]$. Finally, a dyadic decomposition of the integral on the right-hand side together with \eqref{eq:391} and the exponential decay of $\eta$ yield
\begin{eqnarray*}
\frac{1}{\sqrt{RT}} \biggl(\int_{\R}\eta(x_d) |\varft|^2 dx_d\biggr)^\frac12&\leq& \frac1{\sqrt{T}}\sum_{i=1}^\infty \exp(-\gamma 2^{i-1})2^{\frac{i}2}\biggl(\dashint_{B_{2^iR}} |\varft|^2 dx_d\biggr)^\frac12\\
&\stackrel{\eqref{eq:391}}{\leq}&\sqrt{\frac{R}{T}}\mathcal C\dist(F,\SO d)\sum_{i=1}^\infty \exp(-\gamma 2^{i-1})2^{i}
\end{eqnarray*}
and thus \eqref{eq:470} follows (since $\sum_{i=1}^\infty \exp(-\gamma 2^{i-1})2^{i}\lesssim1$).
\end{proof}
\begin{proof}[Proof of Lemma~\ref{lemma:525}]

Throughout the proof we write $\lesssim$ if $\leq$ holds up to a multiplicative constant depending only on $\alpha,d,p$ and $\rho$. As in the proof of Lemma~\ref{lemma:388}, we recall that for $F\in U_{\bar \delta}$ the function $\varfgt$ is the unique one-dimensional solution to 
\begin{equation}\label{eq:1805:2}
\frac{1}{T}\varphi^{T}_{F,G} - \dive \brac{D^2V(\omega(x_d),F + \nabla \varphi^{T}_{F})\brac{G+ \nabla \varfgt} }= 0 \quad \text{in }\R^d,
\end{equation}
where $V$ corresponds to the matching convex lower bound for $W$ from Lemma\nbs\ref{C:wv}.

\step 1 Estimate \eqref{eq:525}. We consider the random variable
$
\F(\omega)= \int_{\R}\partial_{x_d}\varfgt(\omega)\cdot g \; dx_d,
$
where $g: \R\to \R^{d}$ is supported in $B_r(x_0)$ and satisfies \eqref{ineq:gmssq}. We compute the derivative of equation\nbs\eqref{eq:1805:2} with respect to a bounded perturbation $\delta \omega$ supported in $B_1(s)$: First, we note that $\delta \varfgt \in H^1_{\mathrm{loc}}(\R^d;\R^d)$ is one-dimensional in the sense that $\delta\varfgt=\delta\varfgt(x_d)$ and it is the unique weak solution to the ODE (cf. Step~1 of proof of Lemma\nbs\ref{lemma:388})
\begin{align}\label{eq:525:pdeltaphigf}
&\frac{1}{T}\delta \varfgt - \partial_{x_d}\brac{a(x_d)\partial_{x_d}\delta\varfgt}\notag\\
=&\partial_{x_d}\partial_\omega D^2V(\omega(x_d),F+\nabla \varft)\brac{G+\nabla\varfgt\otimes e_d} e_d\delta \omega\notag\\
&+\partial_{x_d} D^3V(\omega(x_d),F+ \nabla \varft)(\nabla \delta \varft)(G+\nabla \varfgt)e_d \quad \text{in }\R,
\end{align}   
where $a(x_d)$ is given in \eqref{def:af}. We denote by $h_1$ the unique weak solution in $H^{1}(\R;\R^d)$ to
\begin{equation}\label{eq:525:h1}
\frac{1}{T}h_1 - \partial_{x_d} \brac{a(x_d)\partial_{x_d}h_1}= \partial_{x_d}g \quad \text{in }\R
\end{equation}
and by $h_2\in H^1(\R;\R^d)$ the unique weak solution to
\begin{equation}\label{eq:525:h2}
\frac{1}{T}h_1 - \partial_{x_d} \brac{a(x_d)\partial_{x_d}h_2}= \partial_{x_d}(D^3V(x_d,F+\nabla \varphi_F^T)(G+\nabla \varphi_{F,G}^T)(\partial_{x_d}h_1\otimes e_d))\quad \text{in }\R.
\end{equation}
Similarly as in Step~1 of the proof of Lemma~\ref{lemma:388}, by testing \eqref{eq:525:h1} with $\delta\varfgt$, \eqref{eq:525:pdeltaphigf} with $h_1$, \eqref{eq:525:h2} with $\delta\varft$ and finally \eqref{eq:491} with $h_2$, we obtain
\begin{align*}
\delta F=&\int_{\R}\partial_{x_d}\delta \varfgt \cdot g \; dx_d \\
 =& \int_{\R}\partial_\omega a(x_d)((G+\nabla\varphi_{F,G}^T)e_d)\delta\omega\cdot \partial_{x_d}h_1\,dx_d\\
 &+\int_\R D^3V(\omega(x_d),F+ \nabla \varft)(\nabla \delta \varft)(G+\nabla \varfgt)e_d \cdot \partial_{x_d}h_1 dx_d\\
 =&\int_{\R}\partial_\omega a(x_d)((G+\nabla\varphi_{F,G}^T)e_d)\delta\omega\cdot \partial_{x_d}h_1\,dx_d\\
 &+\int_\R \partial_\omega DV(\omega(x_d),F+ \nabla \varft)\delta\omega e_d \cdot \partial_{x_d}h_2 \; dx_d
\end{align*}
and thus for all $s\in\R$ (using \eqref{eq:393:new} and \eqref{eq:520:new})
\begin{equation*}
\dashint_{B_1(s)} \left| \frac{\partial \F}{\partial \omega} \right| \lesssim \dashint_{B_1(s)} |G||\partial_{x_d}h_1|+ \dist(F,\SO d)|\partial_{x_d}h_2| dx_d.
\end{equation*}
Using $\int_{\R}|\partial_{x_d} h_2|^2\lesssim |G|^2 \int_{\R}|\partial_{x_d} h_1|^2\lesssim |G|^2\int_{B_r(x_0)}|g|^2$, $|F|\lesssim1$ and $\erwartung{\F}=0$, we obtain
\begin{align*}
\erwartung{|\F|^{2q}}|^\frac1{2q} & \lesssim  q|G|\erwartung{\biggl|\int_{\R}|\partial_{x_d} h_1|^2 ds\biggr|^q}^\frac1{2q}\\
& \lesssim  q|G| \biggl(\int_{B_{r}(x_0)}|g|^2 ds\biggr)^\frac12 \\ 
& \lesssim  q|G| r^{-\frac12}
\end{align*} 
and thus applying Lemma~\ref{lemma:405} we deduce \eqref{eq:525}.

\step 2 Estimate \eqref{eq:531}. This follows analogously to Step 2 of the proof of Lemma \ref{lemma:388} and using similar arguments as in Step 1 of this proof.

\smallskip

\step 3 Estimate \eqref{eq:556}. Using equation \eqref{eq:1805:2}, we obtain that $\Psi:= \varphi_{F,G}^{2T}-\varphi_{F,G}^T$ solves
\begin{equation*}
\frac{1}{2T}\Psi - \dive \brac{D^2V(\omega(x_d),F+\nabla \varft)\nabla \Psi} = \frac{1}{2T}\varfgt + \dive \widetilde{g} \quad \text{in }\R^d,
\end{equation*}
where $\widetilde{g}= \brac{D^2V(\omega(x_d),F+ \nabla \varphi^{2T}_F)-D^2V(\omega(x_d), F+ \nabla \varphi^{T}_F)}(G+ \nabla \varphi_{F,G}^{2T})$. Note that $\Psi$ is one-dimensional and solves an ODE corresponding to the above equation. Thus Lemma\nbs\ref{lemma:222} and \eqref{eq:520:new} yield
\begin{align*}
\dashint_{B_{R}}\frac{1}{T}|\Psi|^2+ |\nabla \Psi|^2\,dx_d & \lesssim \frac{1}{R}\int_{\R}\eta |\widetilde{g}|^2+ \frac{1}{T}\eta |\varfgt|^2 dx_d \\ 
& \lesssim  \frac{1}{R}\int_{\R}|G|^2 \eta |\nabla \varphi_{F}^{2T}- \nabla \varphi_{F}^{T}|^2 + \frac{1}{T}\eta |\varfgt|^2 dx_d,
\end{align*}
where $\eta(z)=\exp(-\gamma\frac{|z|}R)$, $\gamma=\gamma(\alpha, p, d)>0$. Combining estimates \eqref{eq:531} and \eqref{eq:470} with a dyadic decomposition of $\R$ and the exponential decay of $\eta$, we obtain (as in Step~3 of Lemma \ref{lemma:388}) 
\begin{equation*}
\biggl(\frac1{R}\int_\R \eta(\frac1T|\varfgt|^2+|G|^2|\nabla(\varphi_F^{2T}-\varft)|^2)\,dx_d\biggr)^\frac12\leq \mathcal C|G|\biggl(\frac{R}T\biggr)^\frac12.
\end{equation*}
This completes the proof.
\end{proof}

\begin{proof}[Proof of Lemma~\ref{lemma:13:58}]
Throughout the proof we write $\lesssim$ if $\leq$ holds up to a multiplicative constant depending only on $\alpha,d,p$ and $\rho$. As before, we note that $\varfght$ is the unique one-dimensional solution to
\begin{eqnarray}\label{eq:1901:2}
& & \frac{1}{T}\varfght - \dive\brac{D^2V(\omega(x_d),F+ \nabla \varft)\nabla \varfght} \nonumber \\ & = & \dive \brac{D^3V(\omega(x_d),F + \nabla \varft)\brac{H+ \nabla \varphi_{F,H}^T }\brac{G+ \nabla \varfgt}} \quad \text{in }\R^d.
\end{eqnarray}

\step 1 Estimate \eqref{eq:620}. We consider the random variable
$
\F(\omega)= \int_{\R}\partial_{x_d}\varfght(\omega)\cdot g \; dx_d,
$
where $g: \R\to \R^{d}$ is supported in $B_r(x_0)$ and satisfies \eqref{ineq:gmssq}. We compute the derivative of equation\nbs\eqref{eq:1901:2} with respect to a bounded perturbation $\delta \omega$ supported in $B_1(s)$: First, we note that $\delta \varfght \in H^1_{\mathrm{loc}}(\R^d;\R^d)$ is one-dimensional in the sense that $\delta\varfght=\delta\varfght(x_d)$ and it is the unique weak solution to the ODE (cf. Step~1 of proof of Lemma\nbs\ref{lemma:388})
\begin{align*}
\frac{1}{T}\delta \varfght - \partial_{x_d} a(x_d)\partial_{x_d} \delta \varfght=\partial_{x_d}v(x_d)e_d \quad \text{in }\R,
\end{align*}
where $a(x_d)$ is given as in \eqref{def:af} and
\begin{align*}
\begin{split}
v = & \partial_{\omega} D^2V(\omega(x_d), F+\nabla \varft)\delta \omega \nabla \varfght \\ & + D^3V(\omega(x_d), F+ \nabla \varft)(\nabla \delta \varft) (\nabla \varfght) \\ & + \partial_{\omega}D^3V(\omega(x_d),F+ \nabla \varft)(H+ \nabla \varphi_{F,H}^T)(G+\nabla \varfgt)\delta \omega \\& +D^4V(\omega(x_d), F+ \nabla\varft)(\nabla \delta \varft)(H+\nabla \varphi_{F,H}^T)(G+\nabla \varfgt)\\& + D^3V(\omega(x_d), F+\nabla \varft)(\nabla \delta \varphi_{F,H}^T)(G+\nabla \varfgt) \\ & + D^3V(\omega(x_d), F+\nabla \varft)(H+\nabla  \varphi_{F,H}^T)(\nabla\delta \varfgt)
\end{split}
\end{align*}
As in Step~1 of the proof of Lemma \ref{lemma:525}, we compute $\delta \F := \lim_{t \to 0}\frac{\F(\omega + t\delta \omega)-\F(\omega)}{t}$ as follows
\begin{align*}
\delta \F = \int_{\R} v_d \cdot \partial_{x_d}h_1,
\end{align*}
where $h_1\in H^1(\R;\R^d)$ denotes the unique weak solution of \eqref{eq:498}. In a very similar manner to Step~1 of the proof of Lemma~\ref{lemma:525}, thus omitting the details here, we obtain the following estimate, for $q\geq1$,
$
\erwartung{|\F|^{2q}}|^\frac1{2q}  \lesssim  q|G||H|r^{-\frac12}.
$
This yields the claim.

\step 2 Estimate \eqref{eq:621}. These estimates follow analogously to Step~1 of this proof and Step~2 of the proof of Lemma \ref{lemma:388}.

\step 3 Estimate \eqref{eq:653}. We note that using equation \eqref{eq:1901:2}, $\Psi:= \varphi_{F,G,H}^{2T} - \varphi_{F,G,H}^{T}$ solves
\begin{equation}\label{eq:1974:2}
\frac{1}{2T}\Psi - \dive \brac{D^2V(\omega(x_d),F+ \nabla \varft)\nabla \Psi}=\frac{1}{2T}\varphi_{F,G,H}^T + \dive \widetilde{g} \quad \text{in }\R^d,
\end{equation}
where
\begin{align*}
\widetilde{g}= & \brac{D^2V(\omega(x_d),F+\nabla \varphi_F^{2T})-D^2V(\omega(x_d),F+\nabla \varphi_F^{T})}\nabla \varphi_{F,G,H}^{2T}\\ & + D^3V(\omega(x_d),F+\nabla \varphi_{F}^{2T})(H+ \nabla \varphi_{F,H}^{2T})(G+ \nabla \varphi_{F,G}^{2T})\\ & -D^3V(\omega(x_d),F+\nabla \varphi_{F}^{T})(H+ \nabla \varphi_{F,H}^{T})(G+ \nabla \varphi_{F,G}^{T})\\
= & \brac{D^2V(\omega(x_d),F+\nabla \varphi_F^{2T})-D^2V(\omega(x_d),F+\nabla \varphi_F^{T})}\nabla \varphi_{F,G,H}^{2T}\\ & + D^3V(\omega(x_d),F+\nabla \varphi_{F}^{2T})(H+ \nabla \varphi_{F,H}^{2T})(\nabla \varphi_{F,G}^{2T}-\nabla \varphi_{F,G}^{T})\\ & + D^3V(\omega(x_d),F+\nabla \varphi_{F}^{2T})(\nabla \varphi_{F,H}^{2T}-\nabla \varphi_{F,H}^{T})(G+\nabla \varphi_{F,G}^{T})\\ &+ \brac{D^3V(\omega(x_d),F+\nabla \varphi_{F}^{2T})-D^3V(\omega(x_d),F+\nabla \varphi_{F}^{T})}(H+ \nabla \varphi_{F,H}^{T})(G+ \nabla \varphi_{F,G}^{T}).
\end{align*}
The Lipschitz continuity of $D^2V(\omega,\cdot)$ and $D^3V(\omega,\cdot)$, combined with estimates \eqref{eq:520:new} and \eqref{eq:615:new} yield
\begin{equation*}
|\widetilde{g}|\lesssim  |H||G||\nabla \varphi_{F}^{2T}-\nabla \varphi_{F}^{T}| + |H||\nabla \varphi_{F,G}^{2T}-\nabla \varfgt|+|G| |\nabla \varphi_{F,H}^{2T}-\nabla \varphi_{F,H}^{T}|.
\end{equation*}
We remark that $\Psi$ is one-dimensional and thus \eqref{eq:1974:2} boils down to an ODE, hence Lemma~\ref{lemma:222} yields
\begin{eqnarray*}
\dashint_{B_{R}(x_0)}\frac{1}{T} |\Psi|^2 + |\nabla \Psi|^2 dx_d  & \lesssim & \frac{1}{R} \brac{\int_{\R} \eta |\widetilde{g}|^2 dx_d + \frac{1}{T} \int_{\R} \eta |\varphi_{F,G,H}^T|^2 dx_d}\\
& \lesssim &
\frac{1}{R} \int_{\R} \eta \brac{|H|^2|G|^2|\nabla \varphi_{F}^{2T}- \nabla \varft|^2} dx_d \\ & & +\frac{1}{R}\int_{\R}\eta\brac{|H|^2 |\nabla \varphi_{F,G}^{2T}-\nabla \varfgt|^2+ |G|^2 |\nabla \varphi_{F,H}^{2T}-\nabla \varphi_{F,H}^{T}|^2}dx_d\\
& & + \frac{1}{RT} \int_{\R} \eta |\varfght|^2 dx_d,
\end{eqnarray*}
where $\eta(z)=\exp(-\gamma\frac{|z|}{R})$ with $\gamma=\gamma(\alpha, p, d)>0$. Using a dyadic decomposition of $\R$ combined with \eqref{eq:470} and \eqref{eq:556}, we conclude \eqref{eq:653}.  
\end{proof}

\section{Proofs of quantitative results: Theorem \ref{Theorem:2} and Corollary\nobreakspace\ref{cor:T:2}}\label{sec:rve}

\subsection{Random fluctuations, Proof of Theorem~\ref{Theorem:2}, part (i)}

\begin{proof}[Proof of Theorem~\ref{Theorem:2}, part (i)]

Throughout the proof we write $\lesssim$ if $\leq$ holds up to a multiplicative constant depending only on $\alpha,d,p$ and $\rho$.

\step 1 Preparation.

For given $\omega_L\in \Omega_L$ and $F\in \R^{d\times d}$, we denote by $\varphi_F^L$ the unique corrector for the matching convex lower bound, i.e.\ the unique minimizer of \eqref{def:VhomL} (and solution to \eqref{eq:572:3}). We recall that $\varphi_F^L\in W^{1,\infty}(\square_L; \R^d)$ is one-dimensional in the sense that $\varphi_F^L(x)=\varphi_F^L(x_d)$ and satisfies estimate \eqref{est:LipvarphiL1}. Moreover, for all $F\in U_{\overline\delta}$ with $\overline\delta(\alpha,d,p)>0$ as in Theorem~\ref{T:1:0}, we have $\|\dist(F+\nabla \varphi_F^L,\SO d)\|_{L^\infty(\square_L)}<\delta$ with $\delta>0$ as in Lemma~\ref{C:wv} and thus
\begin{align}\label{eq:whvh:variance}
W_{\hom,L}(\omega_L,F)=&\dashint_{\square_L}W(\omega_L(x_d),F+\nabla \varphi_F^L)\,dx\notag\\
=&\dashint_0^LV(\omega_L(x_d),F+\nabla \varphi_F^L(x_d))\,dx_d-\mu\det(F).
\end{align}

\step 2  Fluctuations of $W_{{\rm hom},L}$, proof of \eqref{est:fluctwhl}.

We define a random variable $\mathcal{F}: \Omega_L \to [0,\infty)$ by
\begin{equation*}
\mathcal{F}(\omega_L):=W_{\hom,L}(\omega_L,F)\stackrel{\eqref{eq:whvh:variance}}{=}\dashint_{(0,L)}V(\omega_L(x_d), F+ \nabla \varphi_F^L) dx_d-\mu \det(F).
\end{equation*}
We compute its derivative with respect to a bounded periodic perturbation $\delta \omega\lol$ supported in $B_1(s)+L \Z$: 
\begin{align}\label{eq:320}
\begin{split}
\delta\mathcal{F}& := \lim_{t \to 0}\frac{\F(\omega\lol + t\delta \omega\lol)-\F(\omega\lol)}{t} \\
& =  \dashint_{(0,L)}\partial_\omega V(\omega_L(x_d),F+ \nabla \varphi^L_F) \delta \omega_L + D V(\omega\lol(x_d),F+ \nabla \varphi^L_F) \cdot \nabla \delta\varphi_F^L \; dx_d,
\end{split}
\end{align}
where $\delta\varphi_F^L$ is an $L$-periodic function with zero mean (solution to \eqref{eq:922}). Hence, the second term on the right-hand side in \eqref{eq:320} vanishes. In order to estimate the first term in the right-hand above we make use of $W(\cdot,R)=0$, $DW(\cdot,R)=0$ for all $R\in \SO d$ (hence $\partial_\omega V(\cdot,R)=0$, $\partial_\omega DV(\cdot,R)=0$) and thus 
\begin{align*}
&\partial_\omega V(\omega_L(x_d),F+\nabla \varphi_F^L)\\=&\int_0^1(1-t) D^2\partial_\omega V(\omega_L(x_d),R+t(F-R+\nabla \varphi_F^L))[F-R+\nabla \varphi_F^L,F-R+\nabla \varphi_F^L(x_d)]\,dt\\
\leq&\frac1{2}\|\partial_\omega D^2V\|_{C^0(\overline{U_\alpha})}|F-R+\nabla \varphi_F^L (x_d)|^2\qquad\mbox{for all $R\in\SO d$}.
\end{align*}
Therefore, \eqref{est:LipvarphiL1} and optimization in $R\in\SO d$ yield
\begin{equation*}
\dashint_{B_1(s)} \left| \frac{\partial\mathcal{F}}{\partial\omega\lol} \right| = \sup_{\delta \omega\lol} |\delta\mathcal{F}(\omega\lol,\delta\omega\lol)| \lesssim \frac{1}{L}\dist^2(F,\SO d).
\end{equation*}
Finally, the $q$-th version of the spectral gap inequality in the form of \eqref{SQ:qth1} yields
\begin{equation*}
\erwartunglol{|\mathcal{F}-\erwartunglol{\mathcal{F}}|^{2q}}^\frac1{2q}\lesssim qL^{-1}\dist^2(F,\SO d).
\end{equation*}
This implies the claim \eqref{est:fluctwhl} (see Remark~\ref{rem:exponential}).

\step 3 Fluctuations of $DW_{{\rm hom},L}$, proof of \eqref{est:fluctdwhl}.

As in the previous substep it suffices to show the corresponding claim for $DV$ instead of $DW$. For $F\in U_{\overline \delta}$ and $G \in \R^{d\times d}$, we consider the random variable 
\begin{equation*}
\mathcal{F}(\omega\lol)=\dashint_{(0,L)}DV(\omega\lol(x_d), F+\nabla \varphi^L_F)\cdot G dx_d=DV_{\hom,L}(F)[G].
\end{equation*} 
We compute its derivative with respect to a bounded periodic perturbation $\delta \omega\lol$ supported in $B_1(s)+L \Z$: 
\begin{equation*}
\delta F(\omega_L) = \dashint_{(0,L)}\partial_{\omega}DV(\omega\lol(x_d),F+ \nabla\varphi^L_{F})\delta \omega\lol \cdot G + \hat{\mathbb L}_L(\omega_L(x_d))\nabla \delta\varphi^L_{F} \cdot G \; dx_d,
\end{equation*}
where 
$$
\hat{\mathbb L}_L(\omega_L(x_d)):=D^2V(\omega_L(x_d),F+\nabla \varphi_F^L(x_d)),
$$
and $\delta\varphi^L_{F}\in W_{{\rm per},0}^{1,2}(\square_L)$ solves
\begin{equation}\label{eq:922}
-\dive \hat{\mathbb L}_F^L\nabla \delta \varphi^L_{F}= \dive \brac{\partial_{\omega}DV(\omega\lol(x_d), F+ \nabla \varphi^L_{F}) \delta \omega\lol}\quad \text{in }\square_L.
\end{equation}
As in \cite[Proof of Theorem~9a]{fischer2019optimal} we introduce the auxiliary function $h\in W_{{\rm per},0}^{1,2}(\square_L)$ satisfying
\begin{equation}\label{eq:defhvariancedv}
-\dive \hat{\mathbb L}_L\nabla h= \dive \hat{\mathbb L}_LG\quad \text{in }\square_L
\end{equation}
and obtain (by testing \eqref{eq:defhvariancedv} with $\delta\varphi_F^L$ and \eqref{eq:922} with $h$)
\begin{align*}
 \delta F(\omega_L)=\dashint_{\square_L} \partial_{\omega}DV(\omega\lol(x_d),F+ \nabla\varphi^L_{F})\delta \omega\lol \cdot (G+\nabla h)\,dx.
\end{align*}
Using 
\begin{align*}
 \|\partial_\omega DV(\omega\lol,F+ \nabla\varphi^L_{F})\|_{L^\infty(\square_L)}=\|\partial_\omega DV(\omega\lol,F+ \nabla\varphi^L_{F})\|_{L^\infty(0,L))}\lesssim \dist(F,\SO d),
\end{align*}
and by Lemma~\ref{lem:425} $h(x)=h(x_d)$ with
$
\|\nabla h\|_{L^\infty(\square_L)}\lesssim \|\nabla h\|_{L^\infty(0,L)}\lesssim |G|.
$
Combining the previous three formulas, we obtain
$
\sup_s\dashint_{B_1(s)}\left|\frac{\partial \F}{\partial \omega\lol} \right|\lesssim \dist(F,\SO d)|G|L^{-1}
$
and thus, appealing to the $q$th version of the spectral gap inequality in the form of \eqref{SQ:qth1}
\begin{equation*}
\erwartunglol{|\mathcal{F}-\erwartunglol{\mathcal{F}}|^{2q}}^\frac1{2q}\lesssim qL^{-1}\dist(F,\SO d)|G|.
\end{equation*}
Hence, \eqref{est:fluctdwhl} follows.

\step 4 Fluctuations of $D^2W_{{\rm hom},L}$, proof of \eqref{est:fluctddwhl}.

For $G,H\in \R^{d\times d}$, we define the random variable
\begin{equation*}
\mathcal{F}(\omega\lol):= D^2V_{\hom,L}(\omega_L,F)H\cdot G=\dashint_{[0,L]} D^{2}V(\omega\lol(x_d),F+ \nabla \varphi^{L}_F)\brac{H + \nabla \varphi^L_{F,H}}\cdot G \; dx_d.
\end{equation*}
We compute its derivative with respect to a bounded periodic perturbation $\delta \omega\lol$ supported in $B_1(s)+ L \Z$: 
\begin{align}\label{eq:381}
\delta \mathcal{F} = & \dashint_{(0,L)}  \partial_{\omega}\hat{\mathbb L}_F^L\delta \omega\lol \brac{H+ \nabla\varphi^L_{F,H}}\cdot G \; dx_d \nonumber\\
& + \dashint_{(0,L)} D^3V(\omega\lol(x_d),F+ \nabla \varphi^L_F)\nabla \delta\varphi^L_F \brac{H+ \nabla \varphi^L_{F,H}}\cdot G \; dx_d
\nonumber \\
& + \dashint_{(0,L)}  \hat{\mathbb L}_F^L\nabla \delta \varphi^L_{F,H} \cdot G \; dx_d \nonumber \\ =: &  I_1+ I_2 + I_3, 
\end{align}
where $\delta\varphi_F^L$ solves \eqref{eq:922} and $\delta\varphi^L_{F,H}\in W_{\per,0}^{1,2}(\square_L; \R^d)$ solves
\begin{align}\label{eq:958}
-\dive \brac{\hat{\mathbb L}_F^L \nabla \delta\varphi^L_{F,H}}=&\dive f\quad \text{in }\square_L,  
\end{align}
where
\begin{equation*}
f:=\partial_\omega\hat{\mathbb L}_F^L \brac{H+ \nabla\varphi^L_{F,H}}\delta \omega\lol+D^3V(\omega\lol(x_d),F+ \nabla \varphi^L_F)\nabla \delta\varphi^{L}_F \brac{H+ \nabla \varphi^L_{F,H}}.
\end{equation*}
We estimate the three terms on the right-hand side of \eqref{eq:381} seperately: For the first term, we use
\begin{equation*}
 |I_1|\lesssim |G|\|H+\nabla \varphi_{F,H}^L\|_{L^\infty(\square_L)}\dashint_0^L|\delta\omega_L|\,dx_1\stackrel{\eqref{est:lipschitz:linearperidiccor}}\lesssim |G||H|\dashint_0^L|\delta\omega_L|\,dx_1.
\end{equation*}
To estimate $I_2$, we introduce $h_1\in W_{\per,0}^{1,2}(\square_L; \R^d)$ satisfying
\begin{equation}\label{eq:defh1varianceddv}
-\dive \hat{\mathbb L}_L\nabla h_1= \dive D^3V(\omega\lol(x_d),F+ \nabla \varphi^L_F)[\brac{H+ \nabla \varphi^L_{F,H}},G]\quad \text{in }\square_L.
\end{equation}
and obtain (by testing \eqref{eq:defh1varianceddv} with $\delta\varphi_F^L$ and \eqref{eq:922} with $h_1$)
\begin{align*}
 I_2=\dashint_{\square_L} \partial_{\omega}DV(\omega\lol(x_d),F+ \nabla\varphi^L_{F})\delta \omega\lol \cdot \nabla h_1\,dx\lesssim \dist(F,\SO d)|H||G|\dashint_0^L|\delta \omega_L|.
\end{align*}
It remains to estimate $|I_3|$. Considering $h$ the solution to \eqref{eq:defhvariancedv} we obtain (testing \eqref{eq:defhvariancedv} with $\delta\varphi_{F,H}^L$ and \eqref{eq:958} with $h$)
\begin{align*}
I_3 
 = & \dashint_{(0,L)} \partial_{\omega}\hat{\mathbb L}_F^L \delta \omega\lol \brac{H+ \nabla\varphi^L_{F,H}}
 \cdot \nabla h \; dx_d \\
& + \dashint_{(0,L)} D^3V(\omega\lol(x_d),F+ \nabla \varphi^L_F)\nabla \delta\varphi^{L}_F \brac{H+ \nabla \varphi^L_{F,H}}\cdot \nabla h \; dx_d \\
 =: & I_4 + I_5.
\end{align*}
As before, we have $|I_4|\lesssim  |H||G| \dashint_{(0,L)}|\delta \omega\lol| dx_d$. The same argument as for the estimate of $I_2$ but with with $G$ replaced by $\nabla h$ yields
$
|I_5|\lesssim \dist(F,\SO d)|G||H| \dashint_{(0,L)} |\delta \omega\lol| dx_d.
$
Collecting all these estimates, we obtain
$
\sup_s\dashint_{B_1(s)}\left|\frac{\partial \F}{\partial \omega\lol} \right|\lesssim (1+\dist(F,\SO d))|G||H|L^{-1}
$
and the claim follows via a further application of the spectral gap inequality.

\end{proof}

\subsection{Systematic error, Proof of Theorem~\ref{Theorem:2}, part (ii)}

In order to treat the systematic error, we introduce a localized version of the RVE approximation: For any configuration $\omega: \R\to \R$, and for a compactly supported cut-off function $\eta$ satisfying
\begin{equation}\label{cutoff:eta}
 \eta\in C_c^\infty(\R),\qquad\eta\geq0,\qquad \int_{\R} \eta(x_d)\,dx_d= 1
\end{equation}
we set
\begin{equation}\label{eq:961}
W_{\eta,T}(F)= \int_{\R}\eta(x_d) W(\omega(x_d),F+ \nabla \varft(x_d))dx_d, \quad \text{for }F\in U_{\overline \delta},
\end{equation}
where the localized corrector $\varft$ and $\overline\delta$ are as in Lemma~\ref{lemma:388:new}. Furthermore, we define a cut-off for a realization $\omega\in \Omega$
\begin{equation}\label{def:piL}
\pi\lol \omega(x_d) = \twopartdef{\omega(x_d)}{\text{if }x_d\in [-\frac{L}{4},\frac{L}{4}],}{0}{\text{otherwise.}}
\end{equation}
We define $W_{\eta,T}(\pi\lol \omega,F)$ by replacing $\omega$ by $\pi\lol \omega$ in \eqref{eq:961}. Note that $\omega$ and $\omega\lol$ admit the same distribution on $B_{\frac{L}{4}}(0)$ according to Definition \ref{def:pl}. Thus we have $\erwartunglol{\F(\pi\lol \omega\lol)}=\erwartung{\mathcal{F}(\pi\lol \omega)}$ for any random variable $\F$. This motivates the following decomposition of the systematic error  of $W_{\hom}$:
\begin{eqnarray}\label{eq:1107}
\erwartunglol{W_{\mathrm{hom},L}(F)}- W\h(F)&=&\erwartunglol{W_{\mathrm{hom},L}(F)}-\erwartunglol{W_{\eta,T}(F)}\nonumber\\
& &+ \erwartunglol{W_{\eta,T}(F))}-\erwartunglol{W_{\eta,T}(\pi\lol \omega,F)} \nonumber \\
& &+ \erwartung{W_{\eta,T}(\pi\lol \omega,F)}-\erwartung{W_{\eta,T}(F)}\nonumber\\
& & + \erwartung{W_{\eta,T}(F)}-W\h(F).
\end{eqnarray}
The first and last term on the right-hand side in \eqref{eq:1107} measure the localization error and the middle terms
correspond to the error made by the ensemble periodization. In the following we treat separately these two sources of error, see Lemma~\ref{lem:1131} and \ref{lem:1279} below. In order to estimate the systematic error related to $DW_{\hom}$ and $D^2W_{\hom}$ we use the analogous decomposition as \eqref{eq:1107}, where $DW_{\eta,T}$ and $D^2W_{\eta,T}$ are given via
\begin{align}
DW_{\eta,T}(F)\cdot G =  \int_{\R} & \eta(x_d) DW(\omega(x_d),F+ \nabla \varft(x_d))\cdot \brac{G+ \nabla\varfgt(x_d)} dx_d,\label{def:DWetaT} \\
D^2W_{\eta,T}(F)H \cdot G = \int_{\R} & \eta(x_d) D^2W(\omega(x_d),F+ \nabla \varft(x_d))(H+\nabla \varphi_{F,H}^T(x_d))\cdot (G+\nabla \varfgt(x_d))\notag \\ 
& + \eta(x_d) DW(\omega(x_d),F+ \nabla \varft(x_d))\cdot \nabla \varfght(x_d) dx_d.\label{def:DDWetaT}
\end{align}

\begin{remark}\label{rem:1046}
Assumption \eqref{cutoff:eta} and the stationarity of the random field $(x_d,\omega)\mapsto \F(x_d,\omega):=W(\omega(x_d),F+ \nabla \varft(x_d))$ imply
$$
\erwartung{W_{\eta,T}(F)} = \erwartung{\int_{\R}\eta(x_d)\F(x_d,\omega)dx_d} = \erwartung{\F}\int_{\R}\eta dx_d = \erwartung{\F}
$$
Hence, $\erwartung{W_{\eta,T}}$ coincides for any choice of $\eta$ satisfying \eqref{cutoff:eta}. Clearly, the analogous statements hold for $DW_{\eta,T}$ and $D^2W_{\eta,T}$. 
\end{remark}
\begin{lemma}[Cost of localization]\label{lem:1131}  Suppose the assumptions of Theorem~\ref{Theorem:2} are satisfied. Let $L\geq3$ and let $\mathbb P_L$ be an $L$-periodic approximation of $\mathbb P$ in the sense of Definition~\ref{def:pl}, and denote by $W_{\hom,L}$ the corresponding representative volume approximation.  There exists $c=c(\alpha,p,d,\rho)\in[1,\infty)$ such that for all $F\in U_{\overline \delta}$ with $\overline\delta>0$ as in Lemma~\ref{lemma:388:new} the following estimates are valid

\begin{enumerate}[label=(\roman*)]
\item For all $T\geq \sqrt{2}$, we have for $\ell\in\{0,1,2\}$
\begin{align}
|\erwartung{D^\ell W_{\eta,T}(F)} - D^\ell W\h(F)| & \leq c\dist^{2-\ell}(F,\SO d) \frac{1}{\sqrt{T}}, \label{eq:986}
\end{align}
\item For all $T\in [\sqrt{2},L^2]$, we have for $\ell\in\{0,1,2\}$
\begin{align*}
|\erwartunglol{D^\ell W_{\eta,T}(F)} - \erwartunglol{D^\ell W_{\mathrm{hom},L}(F)}| & \leq c \dist^{2-\ell}(F,\SO d) \frac{1}{\sqrt{T}}.
\end{align*}
\end{enumerate}
\end{lemma}

\begin{lemma}[Cost of ensemble periodization]\label{lem:1279}  Suppose the assumptions of Theorem~\ref{Theorem:2} are satisfied. Let $L\geq3$ and let $\mathbb P_L$ be an $L$-periodic approximation of $\mathbb P$ in the sense of Definition~\ref{def:pl}, and denote by $W_{\hom,L}$ the corresponding representative volume approximation.  There exists $c=c(\alpha,p,d,\rho)\in[1,\infty)$ such that for all $F\in U_{\overline \delta}$ with $\overline\delta>0$ as in Lemma~\ref{lemma:388:new} the following estimates are valid
\begin{align}
|\erwartung{D^\ell W_{\eta\lol,T}(\pi\lol(\cdot),F) - D^\ell W_{\eta\lol,T}(\cdot,F)}| & \leq c\dist^{2-\ell}(F,\SO d) \frac{1}{L}, \label{eq:1109}
\end{align}
where $T := \brac{\frac{1}{C}\frac{L}{\ln(L)}}^2$ and $\eta_L$ denotes a nonnegative weight supported in $B_{\sqrt{T}}(0)$ with $|\eta\lol|\leq \frac{c}{\sqrt{T}}$ and $|\nabla \eta\lol|\leq \frac{c}{T}$ for some $c>0$. Moreover, the same estimates are valid with $\mathbb E$ replaced by $\mathbb E_L$.
\end{lemma}

\begin{proof}[Proof of Theorem~\ref{Theorem:2}, part (ii)]
Follows directly from Lemma~\ref{lem:1131} and \ref{lem:1279} and the decomposition of the systematic error  \eqref{eq:1107}.
\end{proof}

\begin{proof}[Proof of Lemma~\ref{lem:1131}]

We provide the argument only for part (i), the argument for part (ii) follows the same pattern. Throughout the proof we write $\lesssim$ if $\leq$ holds up to a multiplicative constant depending only on $\alpha,d,p$ and $\rho$. Moreover, we suppose that $\overline \delta$ is as in Lemma~\ref{lemma:388:new}.

\step 1 Estimate \eqref{eq:986} with $\ell=0$.

Since
$
 \lim_{T\to\infty}\erwartung{W_{\eta,T}(F)} =W\h(F) 
$, which follows from $\mathbb E[\dashint_\square \det(F+\nabla \varphi_F^T)]=\det(F)$, Remark~\ref{rem:1046}, \eqref{eq:sth12345} and \eqref{W=V}, we have
\begin{equation*}
W\h(F)-\erwartung{W_{\eta,T}(F)}= \sum_{i= 0}^{\infty} (\erwartung{W_{\eta,2^{i+1}T}(F)} - \erwartung{W_{\eta,2^{i}T}(F)}).
\end{equation*}
and thus it is sufficient to prove
\begin{equation*}
|\erwartung{W_{\eta,2T}(F)} - \erwartung{W_{\eta,T}(F)}|\lesssim  \dist^2(F,\SO d) \frac{1}{\sqrt{T}}\qquad\mbox{for all $T\geq1$.}
\end{equation*}
We compute
\begin{eqnarray*}
& & W_{\eta,2T}(F)-W_{\eta,T}(F) \\ & = & \int_{\R} \eta \brac{W(\omega(x_d),F+ \nabla \varphi_{F}^{2T})-W(\omega(x_d),F+ \nabla \varft)- DW(\omega(x_d),F+ \nabla \varft)\cdot \brac{\nabla \varphi_{F}^{2T}-\nabla \varft}} \\ 
& & + \eta DW(\omega(x_d),F+ \nabla \varft)\cdot\brac{\nabla \varphi_{F}^{2T}-\nabla \varft}dx_d\\
& = &  \int_{\R} 
\eta \brac{W(\omega(x_d),F+ \nabla \varphi_{F}^{2T})-W(\omega(x_d),F+ \nabla \varft)- DW(\omega(x_d),F+ \nabla \varft)\cdot \brac{\nabla \varphi_{F}^{2T}-\nabla \varft}}
\\ & & -\frac{1}{T}\eta \varft \cdot(\varphi_{F}^{2T}-\varft)-DW(\omega(x_d),F+ \nabla \varft)e_d\cdot \partial_{x_d} \eta (\varphi_{F}^{2T}-\varft)dx_d.
\end{eqnarray*}
Taking the expectation of the estimate above and using $\erwartung{\int_{\R}\partial_{x_d}\eta \F}=\erwartung{\F}\int_{\R}\partial_{x_d}\eta dx_d = 0$ for any stationary field $\F$, we obtain
\begin{eqnarray*}
& & |\erwartung{W_{\eta,2T}(F)}-\erwartung{W_{\eta,T}(F)}| \\ & \lesssim &  \erwartung{\int_{\R} \eta |\nabla \varphi_{F}^{2T}-\nabla \varft|^2 dx_d} + \erwartung{ \int_{\R} \frac{1}{T}\eta |\varft| |\varphi_{F}^{2T}-\varft|dx_d}\\ & \leq &  \erwartung{\int_{\R} \eta |\nabla \varphi_{F}^{2T}-\nabla \varft|^2 dx_d} + \erwartung{ \int_{\R} \frac{1}{T}\eta |\varphi_{F}^{2T}-\varft|^2dx_d}^{\frac{1}{2}}\erwartung{ \int_{\R} \frac{1}{T}\eta |\varft|^2 dx_d}^{\frac{1}{2}}.
\end{eqnarray*}
According to Remark \ref{rem:1046}, we may choose $\eta$ supported in $B_{\sqrt{T}}(0)$ with $|\eta|\lesssim \frac{1}{\sqrt{T}}$. In this way, \eqref{eq:391} and \eqref{eq:470} yield the claim.

\step 2 Estimate \eqref{eq:986} with $\ell=1$.

As in Step 1, it is sufficient to show that
\begin{equation*}
|\erwartung{DW_{\eta,2T}(F)\cdot G} - \erwartung{DW_{\eta,T}(F)\cdot G}|\lesssim \dist(F,\SO d)||G| \frac{1}{\sqrt{T}}.
\end{equation*}
We compute
\begin{eqnarray*}
& & DW_{\eta,2T}(F)\cdot G- DW_{\eta,T}(F)\cdot G \\ 
& = &
\int_{\R} \eta \brac{DW(\omega(x_d),F+ \nabla \varphi_{F}^{2T}) \cdot (G+ \nabla \varphi_{F,G}^{2T})-DW(\omega(x_d),F+ \nabla \varft) \cdot (G+ \nabla \varfgt)}dx_d \\
& = &  
\int_{\R} \eta \brac{DW(\omega(x_d),F+ \nabla \varphi_{F}^{2T}) -DW(\omega(x_d),F+ \nabla \varft)} \cdot (G+ \nabla  \varfgt)dx_d \\ 
& & 
+\int_{\R} \eta DW(\omega(x_d),F+ \nabla \varphi_{F}^{2T}) \cdot \brac{\nabla \varphi_{F,G}^{2T} -\nabla \varfgt} dx_d\\
& = & 
 \int_{\R} \eta (DW(\omega(x_d),F+ \nabla \varphi_{F}^{2T})-DW(\omega(x_d),F+ \nabla \varft))\cdot (G + 
\nabla \varfgt))dx_d \\ & & -\int_{\R} \eta D^2W(\omega(x_d),F+\nabla \varft)(\nabla \varphi_{F}^{2T}-\nabla \varft) \cdot (G+ \nabla \varfgt) dx_d
\\ 
& &
+\int_{\R} \eta DW(\omega(x_d),F+ \nabla \varphi_{F}^{2T}) \cdot (\nabla \varphi_{F,G}^{2T}-\nabla \varfgt) dx_d\\ & & + \int_{\R}\eta D^2W(\omega(x_d),F+\nabla \varft)(\nabla \varphi_{F}^{2T}-\nabla \varft)\cdot \brac{G+\nabla \varfgt} dx_d \\
& =: & I_0 + I_1 + I_2 + I_3.
\end{eqnarray*}
We treat the four integrals on the right-hand side in the following. In particular, for this purpose we rely on the cancellations coming from \eqref{eq:470}. Specifically, a Taylor expansion, combined with \eqref{eq:520:new}, \eqref{eq:470} and a suitable choice for $\eta$ (that is  ${\rm supp}\,\eta\subset B_{\sqrt{T}}(0)$ with $|\eta|\lesssim \frac{1}{\sqrt{T}}$) yields
\begin{equation*}
\erwartung{I_0 + I_1} \lesssim |G| \erwartung{\dashint_{B_{\sqrt{T}}(0)}|\nabla \varphi_{F}^{2T}-\nabla \varft|^2dx_d} \lesssim |G|\dist^2(F,\SO d) \frac{1}{\sqrt{T}}.
\end{equation*}
Further
\begin{equation*}
I_2 \stackrel{\eqref{eq:372:new}}{=} -\int_{\R} \frac{1}{T}\eta \varphi_{F}^{2T} \brac{\varphi_{F,G}^{2T}-\varfgt} +\partial_{x_d} \eta DW(\omega(x_d),F+ \nabla \varphi_{F}^{2T})e_d \cdot (\varphi_{F,G}^{2T}- \varfgt)dx_d. 
\end{equation*}
The second term on the right-hand side has zero expectation and thus \eqref{eq:391} and \eqref{eq:556} yield
\begin{equation*}
\erwartung{I_2}\lesssim \dist(F,\SO d)|G|\frac{1}{\sqrt{T}}.
\end{equation*}
Finally, we have
\begin{eqnarray*}
I_3 & = &
\int_{\R} D^2W(\omega(x_d),F+\nabla \varft)\brac{G+\nabla \varfgt}\cdot \eta \brac{\nabla \varphi_{F}^{2T}-\nabla \varft}dx_d\\
& \stackrel{\eqref{eq:516:new}}{=}&  
-\int_{\R}\frac{1}{T}\eta \varfgt \brac{\varphi_{F}^{2T}-\varphi_{F}^{T}} + \partial_{x_d}\eta \brac{D^2W(\omega(x_d),F+ \nabla \varft)(G+ \nabla\varfgt)e_d} \cdot \brac{\varphi_{F}^{2T}-\varphi_{F}^{T}}dx_d.
\end{eqnarray*}
Taking the expactation, the second expression on the right-hand side vanishes. Therefore, using \eqref{eq:470} and \eqref{eq:531} we conclude
$
\erwartung{I_3} \lesssim \dist(F,\SO d)|G|\frac{1}{\sqrt{T}}$.
Collecting all the estimates, the claim follows.

\step 3 Estimate \eqref{eq:986} with $\ell=2$.

As in the first step, it suffices
\begin{equation}\label{eq:1059}
|\erwartung{D^2W_{\eta,2T}(F)H\cdot G} - \erwartung{D^2W_{\eta,T}(F)H\cdot G}|\lesssim |H||G| \frac{1}{\sqrt{T}}.
\end{equation}
We compute (dropping $\omega$ from the notation)
\begin{eqnarray*}
& &  D^2W_{\eta,2T}(F)H\cdot G -  D^2W_{\eta,T}(F)H\cdot G\\
& = &
\int_{\R}\eta(D^2W(F+ \nabla \varphi_{F}^{2T})(H+ \nabla \varphi_{F,H}^{2T})\cdot (G+ \nabla \varphi^{2T}_{F,G})\\
& &
- D^2W(F+ \nabla \varft)(H+ \nabla \varphi_{F,H}^{T})\cdot (G + \nabla \varfgt) )dx_d \\
& & 
+ \int_{\R}\eta \brac{DW(F+\nabla \varphi_{F}^{2T})\cdot \nabla \varphi_{F,G,H}^{2T}- DW(F+\nabla \varft)\cdot \nabla \varfght}dx_d\\
& = & 
\int_{\R}\eta \brac{D^2 W(F+ \nabla \varphi_{F}^{2T})-D^2 W(F+ \nabla \varphi_{F}^{T})}\brac{H+\nabla \varphi_{F,H}^{T}}\cdot\brac{G+\nabla \varphi_{F,G}^{T}}dx_d\\
& & 
-\int_{\R}\eta D^3W(F+\nabla \varft)(\nabla \varphi_{F}^{2T}-\nabla \varphi_{F}^{T})\brac{H+\nabla \varphi_{F,H}^{T}}\cdot \brac{G+\nabla \varphi_{F,G}^{T}}dx_d\\
& & 
+\int_{\R}\eta D^3W(F+\nabla \varft)(\nabla \varphi_{F}^{2T}-\nabla \varphi_{F}^{T})\brac{H+\nabla \varphi_{F,H}^{T}}\cdot \brac{G+\nabla \varphi_{F,G}^{T}}dx_d\\
& & 
+\int_{\R}\eta D^2W(F+\nabla \varftt)\brac{\nabla\varfhtt - \nabla \varfht}\cdot \brac{G+ \nabla \varfgt}dx_d\\
& & 
+\int_{\R}\eta D^2W(F+\nabla \varftt)\brac{H + \nabla \varfhtt}\cdot \brac{\nabla \varfgtt- \nabla \varfgt}dx_d\\
& & 
+ \int_{\R}\eta \brac{DW(F+ \nabla \varftt)-DW(F+ \nabla \varft)-D^2W(F+ \nabla \varft)(\nabla \varftt - \nabla \varft))}\cdot \nabla \varfght dx_d \\
& & 
+ \int_{\R}\eta DW(F+\nabla \varftt)\cdot\brac{\nabla \varfghtt - \nabla \varfght}dx_d\\ 
& &
+ \int_{\R}\eta D^2W(F+ \nabla \varft)\brac{\nabla \varftt - \nabla \varft}\cdot \nabla \varfght dx_d\\
& =: & I_1+ I_2+ I_3+ I_4+ I_5 + I_6+ I_7+ I_8.
\end{eqnarray*}
A Taylor expansion combined with \eqref{eq:470} and \eqref{eq:520:new} yields
\begin{equation*}
|\erwartung{I_1+ I_2}| \lesssim |G||H| \erwartung{\dashint_{B_{\sqrt{T}}}|\nabla \varftt - \nabla \varft|^2} \lesssim |G||H||\dist^2(F,\SO d)\frac{1}{\sqrt{T}}.
\end{equation*}
Testing equation \eqref{eq:611:new} with $\eta(\varphi_F^{2T}-\varphi_F^T)$, we obtain 
\begin{align*}
&I_3+I_8= -\int_{\R} \frac{1}{T}\eta \varfght \brac{\varftt - \varft}dx_d-\int_{\R}\nabla\eta\otimes \brac{\varftt - \varft}\cdot S\,dx_3
\end{align*}
where $S:=D^2W(F+\nabla \varft)\nabla \varfght + D^3W(F+\nabla \varft)(H+ \nabla\varfht)(G+ \nabla \varfgt)$ is a stationary random field. Hence, the expectation of the second term above vanishes and thus \eqref{eq:470} and \eqref{eq:621} yield
$
|\erwartung{I_3 + I_8}|\lesssim|G||H|\frac{1}{\sqrt{T}}.
$
Analogously to the treatment of $I_1+ I_2$, we obtain with help of a Taylor expansion, \eqref{eq:615:new} and \eqref{eq:470}
\begin{equation*}
|\erwartung{I_6}|\lesssim\dist^2(F,\SO d)|G||H|\frac{1}{\sqrt{T}}. 
\end{equation*}
Similar computations as in estimate for $I_3+I_8$ yield (using \eqref{eq:516:new} and \eqref{eq:372:new})
\begin{align*}
|\erwartung{I_4+ I_5}|+|\erwartung{I_7}|\lesssim |G| |H|\frac{1}{\sqrt{T}}.
\end{align*}
Collecting all these bounds we conclude \eqref{eq:1059}.

\end{proof}

\begin{proof}[Proof of Lemma~\ref{lem:1279}]

Throughout this proof, we set $\widehat{\omega} := \pi\lol \omega$, see \eqref{def:piL}, and denote by $\hatvarft$, $\hatvarfgt$, $\hatvarfght$ the solutions to \eqref{eq:372:new}, \eqref{eq:516:new} and \eqref{eq:611:new} with $\omega$ replaced by $\hatomega$. The relevant estimates for $\hatvarft$, $\hatvarfgt$, $\hatvarfght$ are contained in Lemma~\ref{L:esttruncatedcorrector}.

\step 1 Proof of \eqref{eq:1109} with $\ell=0$. 

Note that $\sqrt{T}<\frac{L}4$ (recall $\sqrt{T}<\frac{L}{\log(L)}$) and thus $\omega=\hatomega$ on the support of $\eta_L$. Hence, a Taylor expansion implies
\begin{align*}
 &W_{\eta_L,T}(\hatomega,F)-W_{\eta_L,T}(\omega,F)\notag\\
 =&\int_\R \eta\lol \brac{W(\omega(x_d),F+ \nabla \hatvarft) - W( \omega(x_d),F+ \nabla \varft)}dx_d\notag\\
  =&\int_\R \eta\lol \int_0^1DW(\omega(x_d),F+ \nabla \varft+t\nabla(\varft-\hatvarft))(\nabla(\varft-\hatvarft))dx_d.
\end{align*}
Using $DW(\cdot,R)=0$ for all $R\in \SO d$ and the deterministic estimates $\|\nabla \varft\|_{L^\infty}+\|\nabla \hatvarft\|_{L^\infty}\lesssim \dist(F,\SO d)$ (see \eqref{eq:393:new}), we obtain
\begin{align*}
\erwartung{W_{\eta,T}(\pi\lol\cdot,F)-W_{\eta,T}(F)}\lesssim&\dist(F,\SO d)\erwartung{\dashint_{B_{\sqrt{T}}}|\nabla \varft-\nabla \hatvarft|^2\,dx_3}^\frac12
\end{align*}
and in combination with \eqref{eq:1122} and the choice of $T$ we obtain \eqref{eq:1109}.

\step 2 Proof of \eqref{eq:1109} with $\ell=1$. 

As in the previous step, we use $\omega=\hatomega$ on ${\rm supp}\,\eta$, and Taylor expansion to obtain
\begin{eqnarray*}
& & DW_{\eta\lol,T}(F)\cdot G- DW_{\eta\lol,T}(\pi\lol \omega,F)\cdot G \\ 
&\stackrel{\eqref{def:DWetaT}}{=} & \int_{\R}\eta\lol \brac{DW(\omega(x_d),F+\nabla \varft)\cdot\brac{G+\nabla \varfgt} - DW(\omega(x_d),F+\nabla \hatvarft)\cdot\brac{G+\nabla \hatvarfgt}} dx_d\\
& = & \int_{\R}\eta\lol \brac{(DW(\omega(x_d),F+\nabla \varft)-DW(\omega(x_d),F+\nabla \hatvarft))\cdot\brac{G+\nabla \varfgt}}\,dx_d\\
& &+\int _{\R}\eta\lol DW(\omega(x_d),F+\nabla \hatvarft)\cdot\nabla (\varfgt-\hatvarfgt) dx_d.
\end{eqnarray*}
Hence, the deterministic estimates \eqref{eq:393:new} and \eqref{eq:520:new} imply
\begin{align*}
&(DW_{\eta\lol,T}(F)- DW_{\eta\lol,T}(\pi\lol(\cdot),F))\cdot G\\
\lesssim&|G|\biggl(\dashint_{B_{\sqrt{T}}}|\nabla (\varft-\hatvarft)|^2\,dx_d\biggr)^\frac12+\dist(F,\SO d)\biggl(\dashint_{B_{\sqrt{T}}}|\nabla (\varfgt-\hatvarfgt)|^2\,dx_d\biggr)^\frac12.
\end{align*}
Finally, \eqref{eq:1122} and \eqref{eq:1123} yield the claim.

\step 3 Proof of \eqref{eq:1109} with $\ell=2$. 

Using $\hatomega=\omega$ on ${\rm supp}\, \eta_L$, we compute
\begin{eqnarray*}
& & D^2W_{\eta\lol,T}(\pi\lol \omega,F)H\cdot G-D^2W_{\eta\lol,T}(F)H\cdot G \\
& \stackrel{\eqref{def:DDWetaT}}{=} &
 \int_{\R}\eta\lol \left[ D^2W(\hatomega,F+ \nabla \hatvarft)(H+\nabla \hatvarfht)\cdot (G+\nabla \hatvarfgt) + DW(\hatomega,F+\nabla \hatvarft)\cdot \nabla \hatvarfght \right] dx_d\\
& &
- \int_{\R}\eta\lol \left[ D^2W(\omega,F+ \nabla \varft)(H+\nabla \varfht)\cdot (G+\nabla \varfgt) + DW(\omega,F+\nabla \varft)\cdot \nabla \varfght\right] dx_d  \\
& = &   \int_{\R}\eta_L(D^2W(\omega,F+\nabla \hatvarft)-D^2W(\omega,F+\nabla \varft))(H+\nabla \varfht)(G+\nabla \varfgt)dx_d \\
& &
+  \int_{\R} \eta_L D^2W(\omega,F+\nabla \hatvarft)(\nabla \hatvarfht - \nabla \varfht)(G+\nabla \varfgt) dx_d
\\ &&
+ \int_{\R} \eta_LD^2W(\omega,F+\nabla \hatvarft)(H+\nabla \hatvarfht)(\nabla \hatvarfgt-\nabla \varfgt)dx_d
\\ &&
+\int_{\R}\eta_L(DW(\omega, F+ \nabla\hatvarft)-DW(\omega, F+ \nabla\hatvarft))(\nabla\varfght)dx_d
\\ &&
+\int_{\R} \eta_LDW(\omega,F+\nabla\hatvarft)(\nabla \hatvarfght - \nabla\varfght)dx_d \\
& \lesssim & |H||G|\brac{\dashint_{B_{\sqrt{T}}}|\nabla \hatvarft - \nabla \varft|^2 dx_d}^{\frac{1}{2}} + |G|\brac{\dashint_{B_{\sqrt{T}}}|\nabla \hatvarfht - \nabla \varfht|^2 dx_d}^{\frac{1}{2}}\\ 
& & +|H|\brac{\dashint_{B_{\sqrt{T}}}|\nabla \hatvarfgt - \nabla \varfgt|^2 dx_d}^{\frac{1}{2}}+\dist(F,\SO d)|G||H|\\
& &+ \dist(F,\SO d)\brac{\dashint_{B_{\sqrt{T}}}|\nabla \hatvarfght - \nabla \varfght|^2 dx_d}^{\frac{1}{2}},
\end{eqnarray*}
where the last estimate follows from the properties of $W$ and \eqref{eq:393:new}, \eqref{eq:520:new} and \eqref{eq:615:new}. Finally, estimates \eqref{eq:1122}-\eqref{eq:1124} and the choice of $T$ yield the claim.

\end{proof}

\appendix

\section{Linear corrector equation}

We recall the following standard result. We state it in terms of our specific probability space, however, it holds also in a more general stationary and ergodic setting.

\begin{proposition}\label{prop:appendix1}
Let $(\Omega,\mathcal{S},\mathbb{P})$ satisfy \ref{stationarity}-\ref{ass:ergodicity} and for $T>0$ we consider a measurable function $\mathbb{L}_T : \Omega \times \R^d\to L(\R^{d\times d}, \R^{d\times d})$. We assume that there exists $c>0$ such that for all $T>0$ and for $\mathbb{P}$-a.a. $\omega \in \Omega$, it holds
\begin{align}\label{eq:3069:2}
\begin{split}
\frac{1}{c} \int_{\R^d}|\nabla \eta|^2 dx & \leq \int_{\R^d} \mathbb{L}_T(\omega,x) \nabla \eta \cdot \nabla \eta dx  \quad \text{for all }\eta \in C^{\infty}_c(\R^d), \\
|\mathbb{L}_{T}(\omega(x_d))| & \leq c.
\end{split}
\end{align} 

\begin{enumerate}[label = (\roman*)]
\item Then, for any $g_T \in L^{\infty}(\Omega \times \R^d; \R^{d \times d})$, there exists $\varphi^T \in H^1_{\mathrm{uloc}}(\R^d; \R^d)$ a unique  solution to 
\begin{equation*}
\frac{1}{T}\varphi^T - \dive \brac{\mathbb{L}_T(\omega,x)\nabla \varphi^T} = \dive \brac{g_T(\omega, x)} \quad \text{in }\R^d.
\end{equation*}
 
\item We additionally assume:
\begin{enumerate}[label = (\alph*)]
\item There exists $\mathbb{L}: \Omega \times \R^d \to L(\R^{d\times d}, \R^{d \times d})$ satisfying the analogous assumption to \eqref{eq:3069:2}. There exists $g \in L^{\infty}(\Omega \times \R^{d}; \R^{d\times d})$.
\item The mapping $(\omega,x)\mapsto \brac{\mathbb{L}_T(\omega,x), \mathbb{L}(\omega,x),g_T(\omega,x),g(\omega,x)}$ is a stationary random field.
\item It holds
\begin{align*}
\limsup_{T \to \infty} \norm{g_T}_{L^{\infty}(\Omega\times \R^{d})} & <\infty,\\
\lim_{T\to \infty}\brac{\mathbb{L}_T(\omega,x), g_T(\omega,x)} & = \brac{\mathbb{L}(\omega,x), g(\omega,x)} \quad \text{a.e.}
\end{align*}

\end{enumerate}
Then, there exists $\varphi \in H^1_{\rm uloc}(\R^d; \R^d)$ a unique solution to 
\begin{equation*}
-\dive \brac{\mathbb{L}(\omega,x)\nabla \varphi } = \dive\brac{g(\omega,x)} \quad \text{in }\R^d
\end{equation*}
with the following properties
\begin{equation*}
\dashint_{\Box} \varphi dx = 0, \quad \nabla \varphi \text{ is a stationary random field,} \quad \erwartung{\nabla \varphi} =0, \quad \erwartung{|\nabla \varphi|^2}< \infty.
\end{equation*}
Moreover, it holds
\begin{align*}
\limsup_{R\to \infty}\frac{1}{R^2} \dashint_{\Box_{R}}|\varphi|^2 dx = 0,\\
\lim_{T\to \infty}\erwartung{\dashint_{\Box}|\nabla \varphi_{T}- \nabla \varphi|^2 dx} = 0. 
\end{align*}
\end{enumerate}
\end{proposition}
The proof of this proposition is a simple variation of \cite{Neukamm_Lecture_notes} and for this reason we omit it.

\section*{Acknowledgments}
SN and MV acknowledge funding by the Deutsche Forschungsgemeinschaft (DFG, German
Research Foundation) – project number 405009441.

\bibliographystyle{alpha}
\bibliography{refs}

\newcommand{\etalchar}[1]{$^{#1}$}
\begin{thebibliography}{CDMK06}

\bibitem[AFK20]{armstrong2020higher}
Scott Armstrong, Samuel~J Ferguson, and Tuomo Kuusi.
\newblock Higher-order linearization and regularity in nonlinear
  homogenization.
\newblock {\em Archive for Rational Mechanics and Analysis}, 237(2):631--741,
  2020.

\bibitem[AK81]{AkcogluKrengel81}
M.A. Akcoglu and U.~Krengel.
\newblock Ergodic theorems for superadditive processes.
\newblock {\em Journal f\"ur die reine und angewandte Mathematik},
  1981(323):53--67, 1981.

\bibitem[AKM16]{ArmstrongKuusiMourrat16}
Scott Armstrong, Tuomo Kuusi, and Jean-Christophe Mourrat.
\newblock Mesoscopic higher regularity and subadditivity in elliptic
  homogenization.
\newblock {\em Communications in Mathematical Physics}, 347(2):315--361, 2016.

\bibitem[AM11]{HM11}
Omar {Anza Hafsa} and Jean-Philippe Mandallena.
\newblock Homogenization of nonconvex integrals with convex growth.
\newblock {\em Journal de Mathématiques Pures et Appliquées}, 96(2):167--189,
  2011.

\bibitem[AS16]{armstrong2016quantitative}
Scott~N Armstrong and Charles~K Smart.
\newblock Quantitative stochastic homogenization of convex integral
  functionals.
\newblock In {\em Annales scientifiques de l'Ecole normale sup{\'e}rieure},
  volume~49, pages 423--481. Societe Mathematique de France, 2016.

\bibitem[BD98]{BraidesBookHom}
Andrea Braides and Anneliese Defranceschi.
\newblock {\em Homogenization of multiple integrals}.
\newblock Oxford University Press, 1998.

\bibitem[BFFO17]{BFFO17}
Peter Bella, Benjamin Fehrman, Julian Fischer, and Felix Otto.
\newblock Stochastic homogenization of linear elliptic equations: Higher-order
  error estimates in weak norms via second-order correctors.
\newblock {\em SIAM Journal on Mathematical Analysis}, 49(6):4658--4703, 2017.

\bibitem[Bra85]{Braides85}
Andrea Braides.
\newblock Homogenization of some almost periodic coercive functional.
\newblock {\em Rend. Accad. Naz. Sci. XL}, 103:313--322, 1985.

\bibitem[CDMK06]{CDKM06}
Sergio Conti, Georg Dolzmann, Stefan M\"uller, and Bernd Kirchheim.
\newblock Sufficient conditions for the validity of the cauchy-born rule close
  to so(n).
\newblock {\em Journal of the European Mathematical Society}, 8(3):515--539,
  2006.

\bibitem[CG21]{clozeau2021quantitative}
Nicolas Clozeau and Antoine Gloria.
\newblock Quantitative nonlinear homogenization: control of oscillations.
\newblock {\em arXiv preprint arXiv:2104.04263}, 2021.

\bibitem[DG16]{DG16}
Mitia Duerinckx and Antoine Gloria.
\newblock Stochastic homogenization of nonconvex unbounded integral functionals
  with convex growth.
\newblock {\em Archive for Rational Mechanics and Analysis}, 221(3):1511--1584,
  2016.

\bibitem[DG20a]{DuerinckxGloria2020b}
Mitia Duerinckx and Antoine Gloria.
\newblock Multiscale functional inequalities in probability: concentration
  properties.
\newblock {\em ALEA Lat. Am. J. Probab. Math. Stat.}, 17:133--157, 2020.

\bibitem[DG20b]{duerinckx2017multiscale}
Mitia Duerinckx and Antoine Gloria.
\newblock Multiscale functional inequalities in probability: {Constructive}
  approach.
\newblock {\em Annales Henri Lebesgue}, 3:825--872, 2020.

\bibitem[DMM86a]{DalMaso198621}
Gianni Dal~Maso and Luciano Modica.
\newblock Nonlinear stochastic homogenization.
\newblock {\em Ann. Mat. Pura Appl.}, 144(1):347--389, 1986.

\bibitem[DMM86b]{DalMaso1985}
Gianni Dal~Maso and Luciano Modica.
\newblock Nonlinear stochastic homogenization and ergodic theory.
\newblock {\em Journal f\"ur die reine und angewandte Mathematik},
  1986(368):28--42, 1986.

\bibitem[FJM02]{friesecke2002theorem}
Gero Friesecke, Richard~D James, and Stefan M{\"u}ller.
\newblock A theorem on geometric rigidity and the derivation of nonlinear plate
  theory from three-dimensional elasticity.
\newblock {\em Communications on Pure and Applied Mathematics: A Journal Issued
  by the Courant Institute of Mathematical Sciences}, 55(11):1461--1506, 2002.

\bibitem[FN21]{fischer2019optimal}
Julian Fischer and Stefan Neukamm.
\newblock Optimal homogenization rates in stochastic homogenization of
  nonlinear uniformly elliptic equations and systems.
\newblock {\em Archive for Rational Mechanics and Analysis (to appear)}, 2021.

\bibitem[FT02]{FT02}
Gero Friesecke and Florian Theil.
\newblock Validity and failure of the cauchy-born hypothesis in a
  two-dimensional mass-spring lattice.
\newblock {\em Journal of Nonlinear Science}, 12(5):445--478, 2002.

\bibitem[GMT93]{Mueller93}
Giuseppe Geymonat, Stefan M{\"u}ller, and Nicolas Triantafyllidis.
\newblock Homogenization of nonlinearly elastic materials, microscopic
  bifurcation and macroscopic loss of rank-one convexity.
\newblock {\em Archive for Rational Mechanics and Analysis}, 122(3):231--290,
  1993.

\bibitem[GN11]{GN11}
Antoine Gloria and Stefan Neukamm.
\newblock Commutability of homogenization and linearization at identity in
  finite elasticity and applications.
\newblock 28(6):941--964, 2011.

\bibitem[GNO15]{gloria2015quantification}
Antoine Gloria, Stefan Neukamm, and Felix Otto.
\newblock Quantification of ergodicity in stochastic homogenization: optimal
  bounds via spectral gap on glauber dynamics.
\newblock {\em Inventiones mathematicae}, 199(2):455--515, 2015.

\bibitem[GNO20]{gloria2019quantitative}
Antoine Gloria, Stefan Neukamm, and Felix Otto.
\newblock Quantitative estimates in stochastic homogenization for correlated
  coefficient fields.
\newblock {\em Analysis and PDE (accepted)}, 2020.

\bibitem[GO11]{GloriaOtto}
Antoine Gloria and Felix Otto.
\newblock An optimal variance estimate in stochastic homogenization of discrete
  elliptic equations.
\newblock {\em The annals of probability}, 39(3):779--856, 2011.

\bibitem[GO12]{GloriaOtto2}
Antoine Gloria and Felix Otto.
\newblock An optimal error estimate in stochastic homogenization of discrete
  elliptic equations.
\newblock {\em The annals of applied probability}, pages 1--28, 2012.

\bibitem[HNS21]{HNS21}
Helmer Hoppe, Stefan Neukamm, and Mathias Sch\"affner.
\newblock Stochastic homogenization of non-convex integral functionals with
  degenerate growth.
\newblock {\em in preparation}, 2021.

\bibitem[JKO94]{ZhikovBook}
V.~V. Jikov, S.~M. Kozlov, and O.~A. Oleinik.
\newblock {\em Homogenization of differential operators and integral
  functionals}.
\newblock Springer-Verlag, Berlin, 1994.

\bibitem[KFG{\etalchar{+}}03]{kanit}
Toufik Kanit, Samuel Forest, Isabelle Galliet, Val{\'e}rie Mounoury, and
  Dominique Jeulin.
\newblock Determination of the size of the representative volume element for
  random composites: statistical and numerical approach.
\newblock {\em International Journal of solids and structures},
  40(13-14):3647--3679, 2003.

\bibitem[MM94]{messaoudi1994stochastic}
K~Messaoudi and G{\'e}rard Michaille.
\newblock Stochastic homogenization of nonconvex integral functionals.
\newblock {\em ESAIM: Mathematical Modelling and Numerical Analysis},
  28(3):329--356, 1994.

\bibitem[MN11]{MN11}
Stefan M{\"u}ller and Stefan Neukamm.
\newblock On the commutability of homogenization and linearization in finite
  elasticity.
\newblock {\em Archive for rational mechanics and analysis}, 201(2):465--500,
  2011.

\bibitem[M{\"u}l87]{Mueller87}
Stefan M{\"u}ller.
\newblock Homogenization of nonconvex integral functionals and cellular elastic
  materials.
\newblock {\em Archive for Rational Mechanics and Analysis}, 99(3):189--212,
  1987.

\bibitem[Neu18]{Neukamm_Lecture_notes}
Stefan Neukamm.
\newblock An introduction to the qualitative and quantitative theory of
  homogenization.
\newblock {\em Interdisciplinary Information Sciences}, 24(1):1--48, 2018.

\bibitem[NS98]{naddaf1998estimates}
Ali Naddaf and Thomas Spencer.
\newblock Estimates on the variance of some homogenization problems.
\newblock {\em preprint}, 1998.

\bibitem[NS18]{neukamm2018quantitative}
Stefan Neukamm and Mathias Sch{\"a}ffner.
\newblock Quantitative homogenization in nonlinear elasticity for small loads.
\newblock {\em Archive for Rational Mechanics and Analysis}, 230(1):343--396,
  2018.

\bibitem[NS19]{neukamm2019lipschitz}
Stefan Neukamm and Mathias Sch{\"a}ffner.
\newblock Lipschitz estimates and existence of correctors for nonlinearly
  elastic, periodic composites subject to small strains.
\newblock {\em Calculus of Variations and Partial Differential Equations},
  58(2):46, 2019.

\bibitem[NSS17]{NSS17}
Stefan Neukamm, Mathias Schäffner, and Anja Schlömerkemper.
\newblock Stochastic homogenization of nonconvex discrete energies with
  degenerate growth.
\newblock {\em SIAM Journal on Mathematical Analysis}, 49(3):1761--1809, 2017.

\bibitem[SJO21]{schneider2021}
Matti Schneider, Marc Josien, and Felix Otto.
\newblock Representative volume elements for matrix-inclusion composites--a
  computational study on periodizing the ensemble.
\newblock {\em arXiv preprint arXiv:2103.07627}, 2021.

\end{thebibliography}
\end{document}